\documentclass[11pt,a4paper]{amsart}

\usepackage[foot]{amsaddr}
\usepackage{fullpage}

\usepackage{scalerel,amssymb}
\usepackage{amsthm,aliascnt}
\usepackage{nicefrac}
\usepackage{amsmath}

\makeatletter
\renewcommand{\email}[2][]{%
  \ifx\emails\@empty\relax\else{\g@addto@macro\emails{,\space}}\fi%
  \@ifnotempty{#1}{\g@addto@macro\emails{\textrm{(#1)}\space}}%
  \g@addto@macro\emails{#2}%
}
\makeatother

\usepackage[utf8]{inputenc}
\usepackage[T1]{fontenc}
\usepackage[british]{babel}

\usepackage[
backend=biber,
style=alphabetic,
doi=false,url=false,
maxbibnames=10 
]{biblatex}
\renewbibmacro{in:}{}

\usepackage{xcolor}
\usepackage[breaklinks]{hyperref}
\usepackage[capitalize,nameinlink,noabbrev]{cleveref}
\usepackage{breakurl}

\definecolor{blackred}{RGB}{183, 24, 82}
\definecolor{lgreen}{rgb}{0.0, 0.48, 0.0}
\definecolor{lpurple}{rgb}{0.48, 0.0, 0.48}
\definecolor{bblue}{rgb}{0.2, 0.4, 0.8}
\hypersetup{linktocpage,
            colorlinks=true,
            linkcolor=lgreen,
            citecolor=lpurple,
            linktoc=all}
\makeatletter
\renewcommand{\tocsection}[3]{%
  \indentlabel{\@ifnotempty{#2}{\bfseries\ignorespaces#1 #2\quad}}\bfseries#3}
\renewcommand{\tocsubsection}[3]{%
  \indentlabel{\@ifnotempty{#2}{\ignorespaces#1 #2\quad}}#3}
\makeatother

\renewcommand{\theenumi}{\alph{enumi}}%

\usepackage{floatrow}
\usepackage{float}
\usepackage{algorithm,algorithmic}
\Crefname{ALC@unique}{Line}{Lines}

\usepackage{graphicx}
\usepackage{epstopdf}
\usepackage{float}
\graphicspath{{./images/}}
\usepackage{caption, subcaption}
\usepackage{booktabs}

\usepackage{catoptions}
\makeatletter

\def\Autoref#1{%
  \begingroup
  \edef\reserved@a{\cpttrimspaces{#1}}%
  \ifcsndefTF{r@#1}{%
    \xaftercsname{\expandafter\testreftype\@fourthoffive}
      {r@\reserved@a}.\\{#1}%
  }{%
    \ref{#1}%
  }%
  \endgroup
}
\def\testreftype#1.#2\\#3{%
  \ifcsndefTF{#1autorefname}{%
    \def\reserved@a##1##2\@nil{%
      \uppercase{\def\ref@name{##1}}%
      \csn@edef{#1autorefname}{\ref@name##2}%
      \autoref{#3}%
    }%
    \reserved@a#1\@nil
  }{%
    \autoref{#3}%
  }%
}
\makeatother

\newcommand{\set}[1]{\{#1\}}
\newcommand{\seq}[1]{\left(#1\right)}
\newcommand{\idx}[1]{\mbox{\underline{\sf #1}}}
\def\vec{\boldsymbol}
\def\n{\overline}
\let\leq\leqslant
\let\geq\geqslant

\DeclareMathOperator*{\Bern}{\mbox{Bern}}

\DeclareMathOperator*{\Seq}{\mbox{\sc Seq}}
\DeclareMathOperator*{\Set}{\mbox{\sc Set}}
\DeclareMathOperator*{\MSet}{\mbox{\sc MSet}}
\DeclareMathOperator*{\PSet}{\mbox{\sc PSet}}
\DeclareMathOperator*{\Cycle}{\mbox{\sc Cyc}}
\DeclareMathOperator*{\Cov}{\mbox{\rm Cov}}

\DeclareMathOperator*{\Np}{\mbox{NP}}
\DeclareMathOperator*{\Rp}{\mbox{RP}}

\DeclareMathOperator*{\Select}{\mbox{\sc Select}}

\DeclareMathOperator*{\epi}{\mbox{\bf epi}}

\newcommand{\CS}[1]{\mathcal{#1}}

\newcommand \dd[1]  { \,\textrm d{#1}                       }   

\usepackage{microtype}

\theoremstyle{definition}
\newtheorem{theorem}{Theorem}
\newtheorem{proposition}[theorem]{Proposition}

\newtheorem{lemma}[theorem]{Lemma}

\theoremstyle{definition}
\newtheorem{Definition}[theorem]{Definition}
\newtheorem{Example}[theorem]{Example}
\newtheorem{Remark}[theorem]{Remark}

\definecolor{bblue}{rgb}{0.2, 0.4, 0.8}
\definecolor{bgreen}{rgb}{0.2, 0.6, 0.4}
\definecolor{bred}{rgb}{0.8, 0.4, 0.2}
\definecolor{bviolet}{rgb}{0.7, 0.2, 0.7}
\definecolor{blackred}{rgb}{0.6, 0.3, 0.3}
\definecolor{blackblue}{rgb}{0.3, 0.3, 0.6}
\definecolor{bgrey}{rgb}{0.3, 0.3, 0.3}

\usepackage{tikz}
\usetikzlibrary{arrows,trees,shapes,snakes,shapes.geometric,calc,
decorations.markings}
\tikzset{
  treenode/.style = {align=center, inner sep=0pt, text centered,
    font=\sffamily},
  arn_nn/.style = {treenode, circle, bblue, draw=bblue, 
    fill=bblue!10,
    minimum width=0.5em, minimum height=0.5em
  },
  arn_n/.style = {treenode, circle, bblue, draw=bblue, 
    text width=1.5em, very thick,
    fill=bblue!10},
  arn_g/.style = {treenode, circle, bgreen, draw=bgreen, 
    text width=1.5em, very thick,
    fill=bblue!10},
  arn_r/.style = {treenode, circle, bviolet, draw=bviolet, 
    text width=1.5em, very thick,
    fill=bviolet!10},
  arn_x/.style = {treenode, circle, black, draw=bgrey, 
    text width=1.5em, very thick,
    fill=bgrey!10},
  triangle/.style = {treenode, bred, draw=bred, fill=bred!20, regular polygon, regular polygon
    sides=3, very thick, text width=1.5em }
}

\usepackage{xargs} 

\newcommandx*{\fromto}[7][3=black,4=black,5=0,6=0,7=.9,usedefault]{
    \path ($(#2) !#7! (#1)$) edge [thick,#4,bend right=#5,
    decoration={markings,
        mark=at position 0.99 with{\arrow[thick,#3, rotate=-#6]{>}},
    }, postaction={decorate}
    ] node {} ($(#1) !#7! (#2)$);
}

\usepackage{xspace}
\newcommand*{\eg}{e.g.\@\xspace}
\newcommand*{\ie}{i.e.\@\xspace}

\begin{document}

\title{Tuning as convex optimisation: a polynomial tuner for multi-parametric
combinatorial samplers}\thanks{Maciej Bendkowski was partially supported by the
Polish National Science Center grant 2016/21/N/ST6/01032 and a French Government
Scholarship awarded by the French Embassy in Poland. Olivier Bodini and Sergey
Dovgal were supported by the French project ANR project MetACOnc,
ANR-15-CE40-0014. The current paper is an extended version
of~\cite{doi:10.1137/1.9781611975062.9} presented at ANALCO'18.}
\author{Maciej Bendkowski${}^1$}
\address[1]{
Theoretical Computer Science Department,
Faculty of Mathematics and Computer Science,
Jagiellonian University, {\L}ojasiewicza 6,
30-348 Krak\'ow, Poland.
}
\email{maciej.bendkowski@gmail.com}
\author{Olivier Bodini${}^2$}
\address[2]{
Institut Galilée,
Université Paris 13,
99 Avenue Jean Baptiste Clément, 93430
Villetaneuse, France.
}
\email{Olivier.Bodini@univ-paris13.fr}
\author{Sergey Dovgal${}^{2}$}
\email{dovgal.alea@gmail.com}
\date{\today}

\maketitle

\begin{abstract}
  Combinatorial samplers are algorithmic schemes devised for the approximate-
and exact-size generation of large random combinatorial structures, such as
context-free words, various tree-like data structures, maps, tilings,
RNA molecules. They can be
adapted to combinatorial specifications with additional parameters, allowing for
a more flexible control over the output profile of parametrised combinatorial
patterns. One can control, for instance, the number of leaves, profile of node
degrees in trees or the number of certain sub-patterns in generated strings.
However, such a flexible control requires an additional and nontrivial tuning
procedure.

Using techniques of convex optimisation, we present an efficient
tuning algorithm for multi-parametric combinatorial specifications.
Our algorithm works in polynomial time in the system description length, the number of tuning parameters, the number of combinatorial classes in the specification, and the logarithm of the total target size.
We demonstrate the effectiveness of our method on a series of practical
examples, including rational, algebraic, and so-called Pólya specifications. We
show how our method can be adapted to a broad range of less typical
combinatorial constructions, including symmetric polynomials, labelled sets and
cycles with cardinality lower bounds, simple increasing trees or substitutions.
Finally, we discuss some practical aspects of our prototype tuner implementation
and provide its benchmark results.
\end{abstract}

{
  \hypersetup{linkcolor=black}
  \tableofcontents
}

\section{Introduction}
Random generation of combinatorial structures forms a prominent research area of
theoretical computer science. Its wide applications include such topics as
simulation of large physical statistical models~\cite{landau2014guide},
automated software testing~\cite{Palka:2011:TOC:1982595.1982615,cgjlv2000} and
counterexample construction for interactive theorem
provers~\cite{10.1007/978-3-319-22102-1_22}, statistical analysis of queueing
networks~\cite{bouillard2014perfect}, RNA design~\cite{hammer2019fixed}, or
network theory, where one of the major challenges is to devise a realistic model
of random graphs reflecting the properties of real-world
networks~\cite{barabasi2016network}.

Given a formal specification defining a set of combinatorial structures, such as
graphs, proteins, or tree-like data structures, we are interested in designing
an efficient algorithmic sampling scheme, generating such structures according
to some prescribed and rigorously controlled distribution. For instance, being
interested in sampling certain plane trees following a \emph{uniform} outcome
distribution, where plane trees with an equal number of nodes share the same
probability of being constructed, we want to obtain a \emph{combinatorial
sampler} satisfying these input requirements.

There exists a number of different sampling techniques in the literature.
Depending on the considered class of structures, there exist different
\emph{ad-hoc} methods, such as the prominent sampler for plane binary trees due
to Rémy~\cite{Remy85}. If no direct sampling technique is applicable, a common
technique is to use \emph{rejection sampling} in which one generates objects
from a larger, yet simpler class and rejects unwanted samples. Although usually
straightforward to implement, rejection sampling may quickly become infeasible,
especially if the rejection rate grows exponentially fast. Another method, quite
popular in physical modelling, are various Monte Carlo Markov Chain algorithms.
If each of the chain states has an equal number of transitions, then the
stationary distribution is in fact uniform. Let us remark that this technique
and its modifications were successfully applied in sampling random walks and
dimer models~\cite{propp1996exact}.

One of the earliest examples of a universal sampling template is Nijenhuis
and Wilf's recursive method~\cite{NijenhuisWilf1978}, later systematised by
Flajolet, Zimmermann and Van Cutsem~\cite{FLAJOLET19941}.  In this approach, the
input specification, as well as the combinatorial structures it defines, are
recursively decomposed into primitive building blocks. Accordingly, sampling
such objects follows closely the recursive structure of their specification. The
generation scheme is split into two stages -- an initial preprocessing phase,
and the proper sampling itself. During the former, a set of decision
probabilities based on a \emph{target size} $n$ is computed and stored for later
use. The probabilities are chosen so to guarantee a uniform distribution among
structures of size $n$ constructed in the latter phase. Consequently, the
sampling process reduces to a series of random decisions following precomputed
distributions, dictating how to compose the output structure.

Although quite general, the recursive method poses considerable practical
limitations. In both phases the designed algorithm can manipulate integers of
size exponential in the target size $n$, turning its effective bit complexity to
$O(n^{3+o(1)})$, compared to $\Theta(n^2)$ arithmetic operations
required. Denise and Zimmermann reduced later the average-case bit complexity of
the recursive method to $O(n \log n)$ in time and $O(n)$ in space using a
certified floating-point arithmetic optimisation~\cite{DenZimm99}. Regardless,
worst-case space bit complexity remained $O(n^2)$ as well as bit complexity for
specifications defining non-algebraic languages. Remarkably, for rational
languages Bernardi and Giménez~\cite{Bernardi2012} recently linked the
floating-point optimisation of Denise and Zimmermann with a specialised
divide-and-conquer scheme reducing further the worst-case space bit complexity
and the average-case time bit complexity to $O(n)$.

For many years, the exact-size sampling paradigm was the \emph{de facto}
standard in combinatorial generation. In many applications, however, such a
precision is not necessary, and the outcome size may fluctuate around some
target value $n$. Such a relaxed paradigm was made possible with the seminal
paper of Duchon, Flajolet, Louchard and Schaeffer who proposed a universal
sampler construction framework of so-called \emph{Boltzmann
samplers}~\cite{DuFlLoSc}. The key idea of their approach is to relax the
previous exact-size setting, and allow for \emph{approximate-size} samplers,
generating structures within a target size window $[(1-\varepsilon) n,
(1+\varepsilon) n]$ centred around some input size $n$. Like in the recursive
method, Boltzmann samplers closely follow the recursive structure of the input
specification. However now, instead of directly manipulating large integers or
floating-point numbers in order to compute respective decision probabilities,
the preprocessing phase uses numerical \emph{oracles} to evaluate systems of
generating functions corresponding to the specified combinatorial structures.

Throughout the years, a series of important extensions and improvements of
Boltzmann samplers was proposed. Let us mention, for instance, linear
approximate-size (and quadratic exact-size) Boltzmann samplers for planar
graphs~\cite{fusy2005quadratic}, general-purpose samplers for unlabelled
structures~\cite{flajolet2007boltzmann}, efficient samplers for plane
partitions~\cite{bodini2010random} or the cycle pointing operator for P\'{o}lya
structures~\cite{bodirsky2011boltzmann}. The framework of Boltzmann samplers was
moreover generalised onto differential
specifications~\cite{bodini2012boltzmann,bodini2016increasing}. Finally, let us
mention linear exact-size samplers for Catalan and Motzkin trees exploiting the
shape of their holonomic specifications~\cite{bacher2013exact}.

What was left open since the initial work of Duchon, Flajolet, Louchard, and
Schaeffer was the development of (i) efficient, general-purpose Boltzmann
oracles providing effective means of evaluating combinatorial systems within
their disks of convergence, and (ii) an automated tuning procedure controlling
the expected parameter value sizes of generated objects. The former problem was
finally addressed by Pivoteau, Salvy and Soria~\cite{PiSaSo12} who defined a
rapidly converging combinatorial variant of the Newton oracle by lifting the
combinatorial version of Newton's iteration of Bergeron, Labelle and
Leroux~\cite{species} to a new numerical level. In principle, using their Newton
iteration and an appropriate use of binary search, it became possible to
approximate the singularity of a given algebraic combinatorial system with
arbitrarily high precision. However, even if the singularity $\rho$ is estimated
with high precision, say $10^{-10}$, its approximation quality does not
correspond to an equally accurate approximation of the generating function
values at $\rho$, often not better than $10^{-2}$.  Precise evaluation at $z$
close to $\rho$ requires an extremely accurate precision of $z$. Fortunately, it
is possible to trade-off the evaluation precision for an additional rejection
phase using the idea of analytic samplers~\cite{BodLumRolin} retaining the
uniformity even with rough evaluation estimates.

Nonetheless, frequently in practical applications such as for instance software
testing, uniform distribution of outcome structures might not be the most
effective choice~\cite{7107466}. In fact, it can be argued that most software
bugs are minuscule, neglected corner cases, which will not be caught using
large, typical instances of random data,
see~\cite{Palka:2011:TOC:1982595.1982615,Runciman:2008:SLS:1411286.1411292}. In
such cases, additional control over the internal structure of generated objects
is required, cf.~\cite{palka2012}. Non-uniform generation schemes are also
required in genomics~\cite{denise2000random}. Patterns observed in real genomic
sequences are tested against randomness and therefore sequences with given
nucleotide frequencies need to be sampled. Random generation becomes more
involved when the properties do not relate to simple motifs, but, for example,
relates to secondary protein structure~\cite{hammer2019fixed} or to evolution
histories~\cite{chauve2019counting}.

In~\cite{denise2000random} a multi-parametric random generation framework for
context-free languages was suggested, while the question of efficient numerical
tuning of the required generating function arguments was left open.
In~\cite{BodPonty} Bodini and Ponty proposed a multidimensional Boltzmann
sampler model, developing a tuning algorithm meant for the random generation of
words from context-free languages with a given target letter frequency vector.
This was a major improvement over~\cite{denise2000random}, however, their
algorithm converges only in an \emph{a priori} unknown vicinity of the target
tuning variable vector. In practice, it is therefore possible to control no
more than a few tuning parameters at the same time.

In the present paper we propose a novel polynomial-time tuning algorithm based
on convex optimisation techniques, overcoming the previous convergence
difficulties. We demonstrate the effectiveness of our approach with several
examples of rational, algebraic, Pólya structures, and labelled structures.
Remarkably, with our new method, we are able to easily handle large
combinatorial systems with thousands of combinatorial classes and tuning
parameters.

In order to illustrate the effectiveness of our approach, we implemented a
prototype Python library called {\sf Paganini} meant to provide suitable tuning
vectors, and a standalone sampler generator {\sf Boltzmann Brain}. Our software
is freely available as open
source\footnote{\url{https://github.com/maciej-bendkowski/paganini}}%
\footnote{\url{https://github.com/maciej-bendkowski/boltzmann-brain}}.

The paper is structured as follows. In~\cref{section:random:sampling} we outline
basic concepts of combinatorial random sampling, including the principles of the
recursive method and Boltzmann sampling. In~\cref{section:symmetric:polynomials}
we introduce a new (to our best knowledge) admissible \emph{combination}
operation of combinatorial classes, and explain its usefulness in certain
auxiliary transformations for combinatorial specifications used during the
construction of the tuning problem. In the
following~\cref{section:multiparametric:tuning}, we show how to express the
tuning problem as a convex optimisation problem and provide important
simplifications for context-free grammars, labelled structures, increasing trees
and other types of structures. In~\cref{section:barriers} we provide a detailed
complexity analysis of the convex optimisation problems obtained from various
combinatorial specifications using the notion of self-concordant barriers. Then,
in~\cref{section:biased:expectation} we address the problem of finding an
optimal target expectation for Boltzmann samplers when the target size is
constrained in a finite interval \( [n_1, n_2] \). Next, in~\cref{sec:paganini}
we describe our prototype implementation. Finally,
in~\cref{section:applications} we illustrate the effectiveness of our approach
providing several exemplary applications of our tuning algorithm.

\section{Combinatorial random sampling}
\label{section:random:sampling}

\subsection{Admissible combinatorial classes}
Let us consider the neutral class $\mathcal{E}$, commonly denoted as $1$,
consisting of a single object of size zero, and its atomic counterpart
$\mathcal{Z}$, which is a class consisting of a single object of size one. Both
are equipped with a finite set of admissible operators, such as the disjoint
union $+$, Cartesian product $\times$, and sequence $\Seq$, see~\cite[Section
I.2]{flajolet09}. Depending on whether we consider labelled or unlabelled
structures, we allow for more expressive admissible operators including for
instance the multiset $\MSet$, set $\Set$, or cycle $\Cycle$ constructions. In
such a setting, combinatorial specifications we consider in the current paper
are finite systems of equations (possibly mutually recursive) built from
elementary classes $\mathcal{E}$, $\mathcal{Z}$, and admissible operators.

\begin{Example}
Consider the following joint specification for $\mathcal{T}$ and $\mathcal{Q}$.
In the combinatorial class \( \mathcal T \) of trees, nodes at even level (the
root starts at level one) have either no or exactly two children, whereas each
node at odd level has an arbitrary number of non-planarily ordered children:
\begin{equation}
\label{eq:running:univariate}
\begin{cases}
    \mathcal{T}=\mathcal{Z} \MSet(\mathcal{Q})\\
    \mathcal{Q}=\mathcal{Z}+\mathcal{Z}\mathcal{T}^2
\end{cases}
\end{equation}
\end{Example}
In order to distinguish, in other words \emph{mark}, some additional
combinatorial parameters we consider the following natural multivariate
extension of specifiable classes.

\begin{Definition}[Specifiable $k$\nobreakdash-parametric combinatorial classes]
  A specifiable $k$\nobreakdash-parametric combinatorial class is a
combinatorial specification built, in a possibly recursive manner, from $k$
distinct atomic classes $\mathcal{Z}_1, \ldots, \mathcal{Z}_k$, the neutral
class $\mathcal{E}$, and admissible operators. In particular, a vector \(
\vec{\mathcal C} = \seq{\mathcal{C}_1,\ldots,\mathcal{C}_m} \) forms a
specifiable $k$\nobreakdash-parametric combinatorial class if its specification
can be written down as
\begin{equation}\label{eq:specifiable:class}
\begin{cases}
    \mathcal C_1
    =
    \Phi_1(\vec{\CS C}, \CS Z_1, \ldots, \CS Z_k)\\
        \vdots  \\
    \mathcal C_m
    =
    \Phi_m(\vec{\CS C}, \CS Z_1, \ldots, \CS Z_k)
\end{cases}
\end{equation}
    where the right-hand side expressions $ \Phi_i(\vec{\CS C}, \CS Z_1, \ldots,
\CS Z_k)$ are composed from $\vec{\CS C},\CS Z_1,\ldots,\CS Z_k$, admissible
operators, and the neutral class \( \mathcal{E} \).
\end{Definition}

\begin{Example}\label{example:multivariate}
Let us continue our running example, see~\eqref{eq:running:univariate}. Note
that we can introduce two additional \emph{marking} classes \( \mathcal U \) and
\( \mathcal V \), into~\eqref{eq:specifiable:class}, each of weight zero, turning
our example into a $k$\nobreakdash-specifiable combinatorial class:
\begin{equation}
\label{eq:running:multivariate}
\begin{cases}
    \mathcal{T}=\CS U \CS Z \MSet(\mathcal{Q}),\\
    \mathcal{Q}=\CS V \CS Z + \mathcal{Z}\mathcal{T}^2\, .
\end{cases}
\end{equation}
Here, $\mathcal{U}$ is meant to mark the occurrences of nodes at odd levels,
whereas $\mathcal{V}$ is meant to mark leaves at even levels. In effect, we
\emph{decorate} the univariate specification with explicit information regarding
the internal structural patterns of our interest.
\end{Example}

Much like in their univariate variants, $k$\nobreakdash-parametric combinatorial
specifications are naturally linked to ordinary multivariate generating
functions, see~\cite[Chapter III]{flajolet09}.

\begin{Definition}[Multivariate generating functions]
The multivariate ordinary generating function $C(z_1,\ldots,z_k)$ in variables
$z_1,\ldots,z_k$ associated to a specifiable $k$\nobreakdash-parametric
combinatorial class $\mathcal{C}$ is defined as
\begin{equation}
    \displaystyle C(z_1,\ldots,z_k)
    =
    \sum_{p_1\geq 0, \ldots, p_k\geq 0}
    c_{\boldsymbol{p}}
        \boldsymbol{z}^{\boldsymbol{p}}
\end{equation}
where $c_{\boldsymbol{p}}=c_{p_1,\ldots,p_k}$ denotes the number of structures
with $p_i$ atoms of type $\CS Z_i$, and $\boldsymbol{z}^{\boldsymbol{p}}$ denotes
the product $z_1^{p_1} \cdots z_k^{p_k}$. In the sequel, we call
$\boldsymbol{p}$ the \emph{(composition) size} of the structure.
\end{Definition}

In this setting, we can easily lift the usual univariate generating function
construction rules associated with admissible constructions to the realm of
multivariate generating functions associated to specifiable
$k$\nobreakdash-parametric combinatorial classes.

\subsection{Combination operator}\label{section:symmetric:polynomials}
Many examples of combinatorial structures are naturally expressed as collections
of disjoint unions and do not require to explicitly \emph{exclude} certain
undesired configurations. Consequently, right-hand side expressions
$\Phi_i(\vec{\CS C}, \CS Z_1, \ldots, \CS Z_k)$ of their combinatorial
specifications tend to be composed of summands with positive coefficients. One
notable exception, however, are specifications constructed using the
\emph{inclusion-exclusion principle} which, due to the induced class
subtraction, may cause considerable difficulties for both tuning and sampling.

From the sampling perspective, exclusion in the combinatorial specification
requires additional rejection, whose cost may undesirably dominate over the cost
of sampling. Even worse, later in~\cref{section:multiparametric:tuning} we will
argue that, in general, negative coefficients do not allow to express the
tuning problem in convex optimisation form.

Remarkably, in some cases it is possible to rewrite the initial system with
negative coefficients in such a way that the resulting system contains only
positive terms. A typical example of such a situation is the construction of a
non-empty
product containing at most one object from each of the classes \( \CS C_1, \ldots,
\CS C_d \). Symbolically:
\begin{equation}\label{eq:symmetric:polynomials:D}
  (1 + \mathcal C_1)(1 + \mathcal C_2)\cdots(1 + \mathcal C_d) - 1.
\end{equation}
Note that~\eqref{eq:symmetric:polynomials:D} is a combinatorial class consisting
of non-empty tuples of length $d$, such that their $i$-th coordinate is either
$\mathcal{E}$ or an object from $\mathcal{C}_i$. The above concise specification
explicitly excludes the empty tuple, however introduces subtraction in the
specification. We can eliminate the subtraction by expanding all brackets
in~\eqref{eq:symmetric:polynomials:D}, however, such a naive transformation
produces $2^d - 1$ summands in the resulting specification. Instead, we propose
another transformation.

\begin{Definition}[Combination operator]
  Let $\CS C_1, \ldots, \CS C_d$ be combinatorial classes. The admissible
\emph{$k$\nobreakdash-combination} or \emph{$k$\nobreakdash-selection} \(
\Select_k(\CS C_1, \ldots, \CS C_d) \) of $\CS C_1, \ldots, \CS C_d$ is a
combinatorial class consisting of all tuples $\seq{c_1,\ldots,c_d}$
such that
\begin{itemize}
\item  $c_i$ is either empty, i.e.~$\CS E$, or an element of $\CS C_i$, and
\item exactly $\binom{d}{k}$ elements of $\seq{c_1,\ldots,c_d}$ are non-empty.
\end{itemize}
\end{Definition}

\begin{Example}
  Consider the classes $\CS C_1, \ldots, \CS C_d$. Note that
  \begin{equation}\label{eq:combination:operator:iso}
    \Select\nolimits_k(\CS C_1, \ldots, \CS C_d) \cong \sum_{1 \leq i_1 < i_2 <
\cdots < i_k \leq d} \CS C_{i_1} \times \CS C_{i_2} \times \cdots \times \CS
C_{i_k}
  \end{equation}
  In words, the class $\Select\nolimits_k(\CS C_1, \ldots, \CS C_d)$ is
isomorphic with the disjoint union of all ordered $k$\nobreakdash-products of
classes in $\CS C_1, \ldots, \CS C_d$. In particular, $\Select\nolimits_0(\CS
C_1,\ldots, \CS C_d) \cong 1$. Furthermore, it holds $\Select\nolimits_d(\CS
C_1,\ldots, \CS C_d) \cong \prod_i \CS C_i$. It should be noted that the
isomorphism in~\eqref{eq:combination:operator:iso} cannot be replaced with
strict equality, as $\Select\nolimits_k(\CS C_1, \ldots, \CS C_d)$ consists of
tuples of length $d$ whereas the right-hand side sum consists of tuples of
length $k$.
\end{Example}

Let us denote \( \Select\nolimits_k(\CS C_1, \ldots, \CS C_d) \) as $\CS S_k$.
Note that
\begin{equation}
(1 + \mathcal C_1)(1 + \mathcal C_2)\cdots(1 + \mathcal C_d) - 1 =
\CS S_1 + \cdots + \CS S_d.
\end{equation}
The resulting classes \( \CS S_k \) are, in fact,
\emph{elementary symmetric polynomials} consisting of all $\binom{d}{k}$
summands of the expanded $\prod_{i = 1}^{d} (1 + \mathcal{C}_i)$ of
length $k$. Since a direct, verbose representation of $\Select\nolimits_{k}(\CS
C_1, \ldots, \CS C_d)$ is exponential in $d$, we suggest an indirect dynamic
programming approach, allowing us to obtain a much more succinct representation
of each of the selection operations using a total of $O(d^2)$ of auxiliary
combinatorial classes.

Let $\mathcal{P}_{k, m}$ be a subset of $\mathcal{S}_k$ in which products are
restricted to involve classes from $\mathcal{C}_{1}, \ldots, \mathcal{C}_m$.
Note that $\mathcal{P}_{k,m}$ is empty if and only if $m < k$. Moreover,
$\mathcal{P}_{k,d} = \mathcal{S}_k$. We start with
\begin{align}
  \begin{split}
    \mathcal{P}_{1,1} &= \mathcal{C}_1 \\
    \mathcal{P}_{1,m + 1} &= \mathcal{P}_m + \mathcal{C}_{m+1}.
  \end{split}
\end{align}

Now, suppose that we have computed $\mathcal{P}_{0,m},
\mathcal{P}_{1,m},\ldots,\mathcal{P}_{k,m}$ for all values of $m$, and wish to
compute the next row of classes corresponding to $k + 1$. We start to iterate
$m$ in increasing order. For all $m$ such that $m < k + 1$, we set
$\mathcal{P}_{k+1, m} = \emptyset$. Otherwise, if $m \geq k+1$ we note that
\begin{equation}
 \mathcal{P}_{k+1, m} = \mathcal{P}_{k+1, m - 1} + \mathcal{C}_{m} \times
 \mathcal{P}_{k, m - 1}.
\end{equation}

The correctness of the above scheme can be proven by induction. With its help,
we can compute a lower-triangular matrix of the symbolic representations for
$\mathcal{P}_{k,m}$ and so also the elementary symmetric polynomials
$\mathcal{S}_k$. As a byproduct, we can therefore efficiently rewrite the
initial system~\eqref{eq:symmetric:polynomials:D} into an equivalent one with
positive terms.

This transformation can be used as an auxiliary construction in other
combinatorial specifications. For example, we apply the described technique
in~\cref{subsection:weighted:partitions} in order to sample from a
multi-parametric class
\begin{equation}
    \MSet\left(
        \Seq(\CS Z_1)
        \Seq(\CS Z_2)
        \cdots
        \Seq(\CS Z_d) - 1
    \right).
\end{equation}

\subsection{Boltzmann samplers and the recursive
  method}\label{section:recursive}
There exists a number of different sampling techniques in the literature. In the
current paper, we focus on two most prominent, general purpose frameworks ---
the recursive method~\cite{NijenhuisWilf1978} and Boltzmann
samplers~\cite{DuFlLoSc}. In what follows, we focus specifically on their
 distinguishing features.

\subsubsection{The recursive method}
The three basic building bricks of the \emph{recursive method} are, similarly to
the symbolic method and admissible classes, disjoint union, Cartesian product,
and pointing. In this framework, the counting sequences of the combinatorial
classes need to be readily available, as they determine the branching
probabilities of the random generation process. The three operations are
processed as follows.

\textbf{Disjoint union.} If the counting sequence \( (C_n)_{n \geq 0} \) of the
target class \( \CS C \) is given by \( C_n = A_n + B_n \) where \( A_n \) and
\( B_n \) are the counting sequences of the classes \( \CS A \) and \( \CS B \),
respectively, then with probability \( \frac{A_n}{A_n + B_n} \) an object from
\( \CS A \) is constructed, otherwise an object from class \( \CS B \) is
constructed.

\textbf{Cartesian product.} If the counting sequence of the target class \( \CS
C \) satisfies the equation \( C_n = \sum_{k = 0}^n A_k \cdot B_{n-k} \), then a
tuple of sizes \( (k, n-k) \) is chosen with probability \( \frac{A_k
B_{n-k}}{C_n} \) and objects from \( \CS A \) and \( \CS B \) of respective
sizes \( k \) and \( n-k \) are generated.

\textbf{Pointing.} If \( \CS A \) is a combinatorial class, then an object in a
pointed class \( \Theta \CS A \) is isomorphic to an object from \( \CS A \)
which has a distinguished atom. Now, if the counting sequence of \( \CS A \) is
\( A_n \), then the counting sequence of \( \Theta \CS A \) is given by \( n \cdot A_n
\). Having one of the samplers for \( \CS A \) or \( \Theta \CS A \), it is easy
to obtain the other one by either distinguishing a label uniformly at random or,
the other way round, by forgetting which label is distinguished.

Note that the pointing operator substantially enriches the set of admissible
specifications and plays a central r\^ole in the design of recursive samplers.
For example, the two labelled operations, \( \Set \) and \( \Cycle \) are
bijectively transformed using the pointing operation as follows:
  \begin{equation}
    \Theta \Set(\CS A) = \Set(\CS A) \times \Theta \CS A \quad \text{and} \quad
    \Theta \Cycle(\CS A) = \Seq(\CS A) \times \Theta \CS A.
  \end{equation}
We are not aware of a direct and efficient sampling algorithm based on the
recursive method for unlabelled structures involving the \( \MSet \) and \(
\Cycle \) operators.

\subsubsection{Boltzmann samplers}
While the recursive method is applicable to specifications irrespective of the
analyticity of their generating functions, Boltzmann samplers work only with
classes whose generating functions are analytic. However, unlike the recursive
method, Boltzmann samplers do not require the underlying counting series.
Instead, they rely on the values of the generating functions. As a consequence,
the size of the outcome structure is no longer fixed, but follows a
\emph{Boltzmann distribution} \( \mathbb P(\text{size} = k) = \frac{A_k
z^k}{\sum_{n \geq 0} A_n z^n} \) with \( z \) being its parameter.
Nevertheless, conditioned on size, such samplers
generate a uniformly chosen object from the target class.

Boltzmann samplers support three basic operations, i.e.~disjoint union,
Cartesian product, and substitution. For instance, assuming that that \(
\varphi(x) \) is analytic, we can consider a family of simply generated trees
satisfying
  \begin{equation}
    T(z) = z \varphi(T(z)).
  \end{equation}
 Let us remark however, that while substitutions of type \( \varphi(T(z)) \) are
usually easier to handle, substitutions in form of \( T(\varphi(z)) \) are much
more involved. Note that the recursive method does not easily support substitutions.
Now, let us focus on how these three basic operations are processed.\bigskip

\textbf{Disjoint union.}
Consider a target class \( \CS C \) with a generating function \( C(x) \). Let
\( \CS A \) and \( \CS B \) be two combinatorial classes with generating
functions \( A(x) \) and \( B(x) \), respectively, such that \( \CS C = \CS A +
\CS B \). Then, with probability \( \frac{A(x)}{A(x) + B(x)} \) an object from
\( \CS A \) is drawn, otherwise an object from class \( \CS B \) is generated.

\textbf{Cartesian product.}
Consider a target class \( \CS C \) with a generating function \( C(x) \). Let
\( \CS A \) and \( \CS B \) be two combinatorial classes with generating
functions \( A(x) \) and \( B(x) \), respectively, such that \( \CS C = \CS A
\times \CS B \). Then, an independent pair of recursively generated objects from
\( \CS A \) and \( \CS B \) is drawn.

\textbf{Substitution.}
Let \( C(x) = \varphi(B(x)) \) where \( \varphi(t) = \sum_i \phi_i t^i \). Then,
a Boltzmann sampler for \( \CS C \) is obtained as follows. Fix \( t = B(x)\)
and sample a random integer \( k \) from the distribution \( \mathbb P(k) =
\phi_k t^k / \varphi(t) \). Finally, draw \( k \) independent copies of
recursively sampled objects from \( \CS B \).

The substitution rule, in particular, provides Boltzmann samplers for the
labelled \( \Set \), \( \Seq \) and \( \Cycle \) constructions. It turns out
that such samplers also cover a wide interesting family of combinatorial classes
and constructions, including unlabelled structures, first-order differential
specifications, Hadamard product and Dirichlet generating
series, see~\cite{bodini2010autour} for further details.

We summarise the most common rules in
Tables~\ref{table:constructions},~\ref{table:constructions:unlabelled},
and~\ref{table:constructions:labelled}. We write \( X \implies \Gamma \) to
denote a procedure generating random objects from \( \Gamma \) based on the
random discrete distribution \( X \) --- we draw an integer \( r \) from \( X \)
and then, repeatedly and independently, invoke the respective sampler \( \Gamma
\) $r$ times. As a result, we obtain an $r$\nobreakdash-tuple of sampled objects.
For more details we refer the reader to~\cite{DuFlLoSc}
and~\cite{flajolet2007boltzmann}.

\renewcommand{\arraystretch}{1.3}
\begin{table*}[htbp!]
    \begin{equation*}
\begin{array}{r|l|l|l}
     \text{Class}
     &
     \text{Description}
     &
     C(\boldsymbol{z})
     &
     \Gamma \mathcal{C} ( \boldsymbol{z} )
     \\
\hline
\hline
     \text{Neutral}
     &
     \CS{C} = \{\varepsilon\}
     &
     C(\boldsymbol{z}) = 1
     &
     \varepsilon
     \\
\hline
     \text{Atom}
     &
     \CS{C} = \{\square_i\}
     &
     C(\boldsymbol{z}) = z_i
     &
     \square_i\\
\hline
     \text{Union}
     &
     \CS{C} = \CS{A} + \CS{B}
     &
     A(\boldsymbol{z}) +B(\boldsymbol{z})
     &
     \Bern\big(
       \frac{A(\boldsymbol{z})}{C(\boldsymbol{z})},
       \frac{B(\boldsymbol{z})}{C(\boldsymbol{z})}
     \big)
       \longrightarrow
     \Gamma \mathcal{A}(\boldsymbol{z})
     \;|\;
     \Gamma \mathcal{B}(\boldsymbol{z})
     \\
     \hline
     \text{Product}
     &
     \CS{C} = \CS{A} \times \CS{B}
     &
     A(\boldsymbol{z})
       \times
     B(\boldsymbol{z})
     &
     (
       \Gamma \mathcal{A}(\boldsymbol{z});
       \Gamma \mathcal{B}(\boldsymbol{z})
     )
     \\
\hline
      \mbox{Sequence}
      &
      \CS{C} = \Seq(\CS{A})
      &
      (1-A(\boldsymbol{z}))^{-1}
      &
      \text{Geom}(1 - A(\boldsymbol{z})) \implies
      (\Gamma \mathcal A(\vec z))
      \\
\end{array}
\end{equation*}
\caption{Multivariate generating functions and corresponding Boltzmann samplers
$\Gamma\mathcal{C}(\boldsymbol{z})$. }\label{table:constructions}
\end{table*}

\begin{table*}[htbp!]
    \begin{equation*}
\begin{array}{r|l|l|l}
     \text{Class}
     &
     \text{Description}
     &
     C(\boldsymbol{z})
     &
     \Gamma \mathcal{C} ( \boldsymbol{z} )
     \\
\hline
\hline
      \text{MultiSet}
      &
      \MSet(\mathcal{A})
      &
      \exp\left(
          \sum_{m=1}^\infty
          \tfrac{1}{m}A(\boldsymbol{z}^m)
      \right)
      &
      \text{
          see~\cite[Algorithm 1]{bodini2010autour}
      }
      \\
\hline
   \text{Cycle}
   &
   \Cycle(\mathcal{A})
   &
   \sum_{m=1}^\infty \!\!
   \frac{\varphi(m)}{m}
   \ln \frac{1}{1-A(\boldsymbol{z}^m)}
   &
   \text{
      see~\cite[Algorithm 2]{bodini2010autour}
   }
   \\
\end{array}
\end{equation*}
\caption{Multivariate generating functions and corresponding Boltzmann samplers
$\Gamma\mathcal{C}(\boldsymbol{z})$ for unlabelled
    constructions.}\label{table:constructions:unlabelled}
\end{table*}
\begin{table*}[htbp!]
    \begin{equation*}
\begin{array}{r|l|l|l}
     \text{Class}
     &
     \text{Description}
     &
     \widehat{C}(\boldsymbol{z})
     &
     \Gamma \widehat{\mathcal{C}} ( \boldsymbol{z} )
     \\
\hline
\hline
      \text{Set}
      &
      \Set(\mathcal{A})
      &
      \exp(
        \widehat{A}(\boldsymbol{z})
      )
      &
      \text{Pois}(\widehat{A}(\boldsymbol{z})) \implies
      (\Gamma \widehat{\mathcal{A}}(\vec z))
      \\
\hline
   \text{Cycle}
   &
   \Cycle(\mathcal{A})
   &
    \log (1 - \widehat{A}(\boldsymbol{z}))^{-1}
    &
      \text{Loga}(\widehat{A}(\boldsymbol{z})) \implies
      (\Gamma \widehat{\mathcal{A}}(\vec z))
   \\
\end{array}
\end{equation*}
\caption{Multivariate generating functions and corresponding Boltzmann samplers
$\Gamma\mathcal{C}(\boldsymbol{z})$ for labelled
    constructions.}\label{table:constructions:labelled}
\end{table*}

\subsubsection{Multi-parametric Boltzmann samplers}
Consider a multi-parametric combinatorial class \( \CS C \) with a multivariate
generating function \( C(\vec z) \). Let \( \omega \in \CS C \) be a
combinatorial object with composition size \( \vec p \). Then, a
multi-parametric Boltzmann sampler \( \Gamma \CS C(\vec z) \) outputs \( \omega
\) with probability
\begin{equation}
  \mathbb P_{\vec z}(\omega) = \dfrac{\vec z^{\vec p}}{C(\vec z)} \, .
\end{equation}
    Such samplers can be constructed from multi-parametric combinatorial
specifications in the same way as ordinary Boltzmann samplers are constructed.
When the expressions of generating functions involve different values of the
tuning variables, then these variables, after substitution, yield new branching
probabilities.

\begin{proposition}[Log-exp transform of the tuning problem]\label{proposition:expected:value}
Let \( \vec N = (N_1, \ldots, N_k) \) be the random vector where \( N_i \)
equals the number of atoms of type $\CS Z_i$ in a random combinatorial structure
returned by the $k$\nobreakdash-parametric Boltzmann sampler
$\Gamma\mathcal{C}(\vec z)$. Then, the expectation vector \( \mathbb E_{\vec z}
(\vec N) \) and the covariance matrix $\Cov_{\boldsymbol{z}}(\boldsymbol{N})$
are given by
\begin{equation}\label{eq:expectation:and:variance}
    \mathbb{E}_{\boldsymbol{z}}(N_i)
    =
    \left.
    \dfrac
        {\partial}
        {\partial \xi_i}
    \log C(e^{\vec \xi})
    \right|_{\vec \xi = \log \vec z}
    \quad \text{and} \quad
    \mathrm{Cov}_{\boldsymbol{z}}(\boldsymbol{N})
    =
    \left.
    \left[
        \dfrac
            {\partial^2}
            {\partial \xi_i \partial \xi_j}
        \log C(e^{\vec \xi})
    \right]_{i,j = 1}^k
    \right|_{\vec \xi = \log \vec z}
    \enspace .
\end{equation}
Hereafter, we use \( e^{\boldsymbol z} \) to denote coordinate-wise exponentiation.
\end{proposition}

\begin{proof}
  Let the following \emph{nabla-notation} denote the vector of derivatives
(so-called gradient vector) with respect to the variable vector \( \vec z =
(z_1, \ldots, z_k) \):
  \begin{equation}
    \nabla_{\vec z} f(\vec z) =
    \left(
      \dfrac{\partial}{\partial z_1} f(\vec z),
      \ldots,
      \dfrac{\partial}{\partial z_k} f(\vec z)
    \right)^\top\, .
  \end{equation}
    Let \( (c_{\vec n})_{\vec n \succeq \vec 0} \) be the counting sequence of
the combinatorial class \( \CS C \). The probability generating function \(
p(\vec u | \vec z) \) for \( \vec N \) with \( \vec u \) as an argument and \(
\vec z \) as a parameter takes the form
\begin{equation}
  p(\vec u\, | \, \vec z)
  =
  \sum_{\vec n \succeq \vec 0}
  \dfrac{c_{\vec n} \vec z^{\vec n} \vec u^{\vec n}}
  {C(\vec z)}
  =
  \dfrac{C(\vec u \bullet \vec z)}{C(\vec z)}
\end{equation}
where \( \bullet \) denotes component-wise vector multiplication. Hence, the
expected value and the covariance of \( \vec N \) can be immediately expressed
through its probability generating function as
      \begin{align}
        \begin{split}
        \mathbb E_{\vec z} \vec N &= \nabla_{\vec u}
        p(\vec u |  \vec z) |_{\vec u = \vec 1}\\
        \mathrm{Cov}_{\vec z} (\vec N) &=
        \left[
            \nabla^2_{\vec u}
            p(\vec u |  \vec z)
            +
            \mathrm{diag}(
            \nabla_{\vec u}
            p(\vec u |  \vec z))
            -
            \nabla_{\vec u}
            p(\vec u |  \vec z)
            \nabla_{\vec u}^\top
            p(\vec u |  \vec z)
        \right]_{\vec u = \vec 1}.
        \end{split}
  \end{align}
    The proof is finished by expanding the log-exp transform
in~\eqref{eq:expectation:and:variance} and comparing the result to the
expressions obtained from the probability generating functions.
\end{proof}

\begin{Remark}
\label{corrolary:convexity}
The function \( \gamma(\vec z) := \log C(e^{\vec z}) \) is convex because its
matrix of second derivatives, as a covariance matrix, is positive semi-definite
inside the set of convergence. This crucial assertion will later prove central
to the design of our tuning algorithm. The expressions for the expectation and
the covariance matrix are similar to those obtained in~\cite{BenderCLT} for
central limit theorem for multivariate generating functions.
\end{Remark}

\begin{Remark}
    Uni-parametric recursive samplers of Nijenhuis and Wilf take, as well as
Boltzmann samplers, a system of generating functions as their input. This system
can be modified by putting fixed values of tuning variables, in effect altering
the corresponding branching probabilities. The resulting distribution of the
random variable corresponding to a weighted recursive sampler coincides with the
distribution of the Boltzmann-generated variable conditioned on the structure
size. As a corollary, the tuning procedure that we discuss in the following
section is also valid for the exact-size approximate-frequency recursive
sampling.
\end{Remark}

\subsection{Complexity of exact parameter sampling}
While the current paper is devoted to tuning of the parameters \emph{in
expectation}, let us pause for a moment and ask the following, natural question
--- what is the complexity of \emph{exact-parameter} sampling for
multi-parametric combinatorial specifications?

In the current section we show that, unless both the classes of decision
problems solvable in randomised $\Rp$ and nondeterministic $\Np$ polynomial time
are equal, then already for unambiguous context-free languages there exists no
fully polynomial-time algorithm for almost-uniform exact-parameter sampling
problem. Since it is widely conjectured that $\Rp \neq \Np$,
cf.~\cite{welsh2001complexity}, this infeasibility result justifies parameter
tuning \emph{in expectation}, which can be regarded as a continuous relaxation
of the exact-parameter problem variant.

Consider a context-free grammar $G$ with derivation rules in form of
\begin{equation}\label{eq:cf:grammar}
    A_i \to T_{i,j}
\end{equation}
where \( A_i \) is a non-terminal symbol, and the right-hand side expression \(
T_{i,j} \) is a (possibly empty) word consisting of both terminal and
non-terminal symbols. Recall that the context-free grammar $G$ is said to be
\emph{unambiguous} if each word its generates has a unique derivation. Let \(
a_1, \ldots, a_d \) be distinct terminal symbols. The exact multi-parametric
sampling problem for unambiguous context-free grammars can be stated as
follows --- given natural numbers \( n_1, n_2, \ldots, n_d \), sample uniformly
at random a word of length \( n = n_1 + \cdots + n_d \) from the language $L(G)$
generated by $G$, such that the number of occurrences of each terminal symbol \(
a_j \) is equal to \( n_j \).

\begin{Example}
  As a simple illustrating example of the discussed problem, consider the
following grammar $B$ generating (unambiguously) all binary words over the
alphabet $\Sigma = \set{\mathtt{0},\mathtt{1}}$:
  \begin{equation}
    B \to \mathtt{0} B \mid \mathtt{1} B \mid \varepsilon.
  \end{equation}
  Recall that \( \varepsilon \) denotes the empty word. In this example, given
numbers $n_0$ and $n_1$, the multi-parametric sampling problem asks to generate
a uniformly random binary word over $\Sigma$ which has exactly $n_0$
$\mathtt{0}$s, and $n_1$ $\mathtt{1}$s.
\end{Example}

In what follows we prove that, in general, the problem of exact-size
multi-parametric sampling (even if the specification does not involve loops) can
be reduced to the $\#\mathcal{P}$\nobreakdash-complete $\#2$\nobreakdash-SAT
problem, which asks to count the number of satisfiable variable assignments of
a given $2$\nobreakdash-CNF formula. As suggested to us by Sergey Tarasov in
personal communication, such a complexity result might be folklore knowledge,
however we did not manage to find it in the literature. We therefore take the
liberty to fill this gap. Detailed definitions from complexity theory can be
found in the papers referenced during the proof of the following theorem.

\begin{theorem}[Infeasibility of exact parameter sampling]
    Unless \( \Np = \Rp \), there is no fully polynomial-time algorithm for
almost-uniform multi-parametric sampling from unambiguous context-free grammars.
\end{theorem}

\begin{proof}
  We start by showing that extracting the coefficients $[z_1^{k_1}\cdots
z_m^{k_m}] F_i(z_1,\ldots,z_m)$ of a multivariate generating
$F_i(z_1,\ldots,z_m)$ satisfying a system of polynomial equations in form of
  \begin{align}
    \begin{cases}
      F_1(z_1,\ldots,z_m) &= \Phi_1(z_1,\ldots,z_m,F_1,\ldots,F_n)\\
      \vdots&\\
      F_n(z_1,\ldots,z_m) &= \Phi_n(z_1,\ldots,z_m,F_1,\ldots,F_n)
    \end{cases}
  \end{align}
  is $\#P$\nobreakdash-hard. We proceed by reduction from the
  $\#P$\nobreakdash-complete $\#2$\nobreakdash-SAT problem.

  Consider a $2$\nobreakdash-SAT formula
  \begin{equation}
    F = \bigwedge_{j = 1}^m (\alpha_j \lor \beta_j)
  \end{equation}
  with $n$ Boolean variables $x_1,\ldots,x_n$ and $m$ clauses. For each clause
$(\alpha_j \lor \beta_j)$ we create a distinct (complex) variable $c_j$. Next,
for each literal $x \in \set{x_1,\ldots,x_n, \n x_1 \ldots, \n x_n}$ we introduce a
corresponding multivariate generating function $X(c_1,\ldots,c_m)$
(where $X \in \{ X_1, \ldots, X_n, \n X_1, \ldots, \n X_n \}$)
defined as
  \begin{equation}
    X(c_1,\ldots,c_m) = \prod_{j=1}^m {c_j}^{\mathbf 1_{c_j}(x)}
  \end{equation}
  where $x \mapsto \mathbf 1_{A}(x)$ denotes the indicator function with respect to a set \( A \) (i.e.~$\mathbf 1_{A}(x) = 1$ if
  $x \in A$, and $\mathbf 1_{A}(x) = 1$ if $x \not\in A$).  
  Finally, we
  create a multivariate generating function
  \begin{equation}\label{eq:exact-size:sampling:H}
    H(c_1, \ldots, c_m) = \prod_{i=1}^n (X_i(c_1,\ldots,c_m) + \n
X_i(c_1,\ldots,c_m)).
  \end{equation}

  Clearly, all of these generating functions form a polynomial system of
equations. Intuitively, generating functions $X_i(c_1,\ldots,c_m)$ are
products of variables $c_j$ corresponding to clauses which become satisfied once
the respective literal $x_i$ is true. Furthermore, $H(c_1,\ldots,c_m)$ encodes
all possible variable assignments. Indeed, if we expand all brackets
in~\eqref{eq:exact-size:sampling:H} we obtain $2^n$ summands, each corresponding
to a distinct variable assignment --- occurrences of $X_i(c_1,\ldots,c_m)$ in each
summand encode setting the respective variable $x_i$ to true, whereas $\n
X_i(c_1,\ldots,c_m)$ encode setting the respective variable $x_i$ to false.

Let us consider an arbitrary summand $H_\varphi$ in the
expanded~\eqref{eq:exact-size:sampling:H} corresponding to some variable
assignment $\varphi \colon \set{x_1,\ldots,x_n} \to \{ \mathsf{True}, \mathsf{False} \}$. 
Note that once
we unfold the respective definitions of $X_i(c_1,\ldots,c_m)$ and $\n
X_i(c_1,\ldots,c_m)$, $H_\varphi$ becomes a monomial consisting of $c_j$'s
satisfied by $\varphi$. Each clause $(\alpha_j \lor \beta_j)$ in $F$ consists of
two literals, hence its corresponding variable $c_j$ can occur at most twice in
$H_\varphi$. Consequently, the number of satisfiable assignments satisfying $F$
is equal to
\begin{equation}\label{eq:exact-size:sampling:H1}
  [c_1^{\geq 1} \cdots c_n^{\geq 1}]H(c_1,\ldots,c_n)
\end{equation}
i.e.~to the number of monomials $H_\varphi$ in which each clause variable occurs
at least once. In order to obtain the number~\eqref{eq:exact-size:sampling:H1}
of satisfiable assignments using a single, exact coefficient extraction, we note
that
\begin{equation}
  [c_1^2 \cdots c_n^2]H(c_1,\ldots,c_n) \prod_{i=1}^n (1 + c_i)
  = [c_1^{\geq 1} \cdots c_n^{\geq 1}]H(c_1,\ldots,c_n).
\end{equation}

And so, it is as hard to extract the coefficients of a generating function as
solving a \( \#2 \)\nobreakdash-SAT instance. Nevertheless, it should be noticed
that not every hard enumeration problem automatically corresponds to a hard
uniform random sampling problem.

    In order to complete the proof, we use the celebrated ~\cite[Theorem
6.4]{jerrum1986random} which proves that if there exists a fully polynomial
almost-uniform random sampling algorithm (with an exponentially small error),
then there exists a fully polynomial randomised approximation scheme for the
counting problem as well (within a polynomially small error). Moreover, unless
$\Rp = \Np$, there exists not fully polynomial randomised approximation scheme
for $\#2$\nobreakdash-SAT~\cite{welsh2001complexity}. Hence, indeed the theorem
statement must hold.
\end{proof}

\section{Multi-parametric tuning}\label{section:multiparametric:tuning}
A combinatorial specification typically involves several classes \( \CS C_1,
\ldots, \CS C_m \) which are defined in a mutually recursive manner. Let us
denote \( (\CS C_1, \ldots, \CS C_m) \) as \( \vec{\CS C} \). For tuning and
sampling purposes, one particular class in \( \vec{\CS C} \) is chosen. For
example, consider a system
\begin{equation}
\begin{cases}
    \mathcal C_1
    =
    \Phi_1(\vec{\CS C}, \CS Z_1, \ldots, \CS Z_k)\\
        \vdots  \\
    \mathcal C_m
    =
    \Phi_m(\vec{\CS C}, \CS Z_1, \ldots, \CS Z_k).
\end{cases}
\end{equation}
Suppose that we want to sample the objects from the class \( \mathcal C_1 \) and
we fix the expected values of parameters \( \CS Z_1, \ldots, \CS Z_k \) to be \(
N_1, \ldots, N_k \), respectively. Let \( C_1(z_1, \ldots, z_k), \ldots,
C_m(z_1, \ldots, z_k) \) be the related generating functions. Then, the system
of polynomial equations corresponding to the tuning problem (see
\eg~\cite[Proposition 2.1]{DuFlLoSc}) is given by
\begin{equation}
\label{eq:tuning:system}
    N_i =
    z_i \dfrac{\partial_{z_i} C_1(z_1, \ldots, z_k)}{C_1(z_1, \ldots, z_k)}
    \qquad
    \text{for } i = 1, \ldots, k.
\end{equation}

\begin{figure}[hbt!]
\[
\text{\color{bviolet}\framebox{
Parameters $\boldsymbol z$
}}
    \quad \Rightarrow \quad
\text{\color{bblue}\framebox{
        Expectations \( \mathbb E n_k \)
}}
\]
\centering
\begin{tikzpicture}[>=stealth',level/.style={thick}]
\draw
node[arn_r, minimum height = 0.8cm](x)  at (0,0)   {\( z_1 \)}
node[arn_r, minimum height = 0.5cm](y)  at (0, -2) {\( z_2 \)}
node[arn_r, minimum height = 0.6cm](z)  at (0, -4) {\( z_k \)}
node[arn_n, minimum height = 0.9cm](A)  at (4,0)   {\( \mathbb E n_1 \)}
node[arn_n, minimum height = 1.1cm](B)  at (4, -2) {\( \mathbb E n_2 \)}
node[arn_n, minimum height = 1.0cm](C)  at (4, -4) {\( \mathbb E n_k \)}
;
\path[->] (x) edge [bred,dashed,thick,bend right=20] node {} (A);
\path[->] (x) edge [bred,dashed,thick,bend right=20] node {} (B);
\path[->] (x) edge [bred,dashed,thick,bend right=20] node {} (C);
\path[->] (y) edge [bred,dashed,thick,bend right=20] node {} (A);
\path[->] (y) edge [bred,dashed,thick,bend right=20] node {} (B);
\path[->] (y) edge [bred,dashed,thick,bend right=20] node {} (C);
\path[->] (z) edge [bred,dashed,thick,bend right=20] node {} (A);
\path[->] (z) edge [bred,dashed,thick,bend right=20] node {} (B);
\path[->] (z) edge [bred,dashed,thick,bend right=20] node {} (C);
\end{tikzpicture}
\caption{Dependency of parameters and expectations for multi-parametric tuning.
Circle radii represent the possible values of the parameters.}
    \label{fig:handles:tuning}
\end{figure}
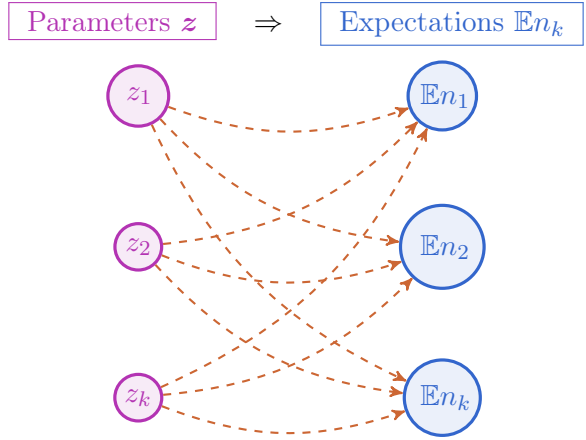

Note that, in general, the values of the tuning parameters cannot be obtained by
independently solving each of the equations~\eqref{eq:tuning:system}. Each of
the functions is depending on all the arguments at the same time (see
also~\cref{fig:handles:tuning}). Hence, we propose an alternative procedure to
achieve this goal.

From a technical point of view, the easiest case for theoretical analysis are
combinatorial specifications corresponding to unambiguous context-free grammars.
Here, the system defining the generating functions becomes a \emph{well-founded
system} of polynomial equations, see~\cite{PiSaSo12}. The most general framework
of Boltzmann sampling comprises much more cases, including labelled objects,
Pólya structures, or first- and second-order differential specifications. There
also exist specifications whose equations include subtractions (related to the
inclusion-exclusion principle), substitutions, and in these cases different
sampling strategies should be applied, \eg recursive sampling or sampling with
rejections, see~\cref{section:recursive}. In this section, we are only concerned
with tuning, setting thus all of these sampling issues aside.

\subsection{Tuning as a convex optimisation problem}
\label{section:tuning:as:convex:optimisation}

It turns out that instead of solving a system of polynomial
equations~\eqref{eq:tuning:system} involving the derivatives of the generating
function, a simpler convex optimisation problem, involving only the values of
the generating functions, can be considered. Having the derivatives of this
function comes as an advantage, because it allows to use first-order subroutines
to solve the optimisation problem. The following theorem contains the most
general form of our tuning approach. For convenience, we will write \( f(\cdot)
\to \min_{\boldsymbol z} \), \( f(\cdot) \to \max_{\boldsymbol z} \) to denote
the minimisation (maximisation, respectively) problem of the target function \(
f(\cdot) \) with respect to the vector variable~\( \vec z \).

\begin{theorem}[Tuning as convex optimisation]
\label{theorem:tuning}
\label{theorem:general}
    Let \( F(z_1, \ldots, z_d) \) be a formal power series with non-negative
coefficients analytic in an open \( d \)-dimensional set \( \Omega \subseteq \{
z_1 > 0, \ldots, z_d > 0 \} \). Assume that the solution of the multi-parametric
tuning problem
    \begin{equation}
        \label{eq:tuning:system:th}
        N_i =
        z_i \dfrac{\partial_{z_i} F(z_1, \ldots, z_k)}{F(z_1, \ldots, z_k)}
        \qquad
        \text{for } i = 1, \ldots, k
    \end{equation}
    belongs to \( \Omega \).

    Let $\boldsymbol{N}$ denote the vector
$(N_1,\ldots,N_d)$. Then, the tuning problem~\eqref{eq:tuning:system:th} is
equivalent to the following convex optimisation problem over real variables \(
\varphi \) and $\boldsymbol{\zeta} = (\zeta_1, \ldots, \zeta_d)$:
    \begin{equation}
        \label{expression:main:optimisation}
        \begin{cases}
            \varphi - \vec N^\top \boldsymbol \zeta \to \min_{\boldsymbol
            \zeta, \varphi}, \\
            \varphi \geq \log F(e^{\boldsymbol \zeta})
        \end{cases}
    \end{equation}
    provided that the arguments of \( F \) meet its domain, and the logarithm
    is well-defined. The respective tuning parameters $\boldsymbol{z}^*$ satisfy
    then $\boldsymbol{z}^* = e^{\boldsymbol{\zeta}}$.
\end{theorem}
\begin{proof}
  Using a log-exp transformation, we note that the tuning
problem~\eqref{eq:tuning:system:th} is equivalent to
$\nabla_{\boldsymbol{\zeta}} \log F(e^{\boldsymbol{\zeta}}) = \boldsymbol{N}$,
cf.~\cref{proposition:expected:value}. Since the right-hand side vector
$\boldsymbol{N}$ is equal to $\nabla_{\boldsymbol{\zeta}}(\boldsymbol{N}^{\top}
\boldsymbol{\zeta})$, the tuning problem is further equivalent to
\begin{equation}\label{eq:convex:optimisation:gradient:i}
  \nabla_{\boldsymbol{\zeta}} \left( \log F(e^{\boldsymbol{\zeta}}) -
\boldsymbol{N}^{\top} \boldsymbol{\zeta} \right) = 0.
\end{equation}
Note that the function under the gradient is a sum of a convex and linear
function, and so necessarily convex itself. We can therefore equivalently
express~\eqref{eq:convex:optimisation:gradient:i} as a convex minimisation
problem in form of
\begin{equation}
  \log F(e^{\boldsymbol{\zeta}}) -
  \boldsymbol{N}^{\top} \boldsymbol{\zeta} \to  \min_{\boldsymbol \zeta}.
\end{equation}
This problem can be reduced to a standard
form~\eqref{expression:main:optimisation} by adding an auxiliary
variable \( \varphi \).
\end{proof}

\begin{Remark}
    We do not require the exponents of the generating function to be
non-negative integers. Depending on the specific application, they might, for
instance, be positive real or rational numbers~\cite{hammer2019fixed}.
Even the non-negativity requirement could be omitted as well, as illustrated by a following univariate example.
Let the exponents of the generating function \( F(z) \) belong to the set \( A \) which may potentially include negative elements. Then, the log-exp transform of the function \( F(z) \) is a composition of a convex function and a set of linear functions, which is again, convex:
\[
    \log F(e^z) = \log \left(
        \sum_{s \in A} a_s e^{xs}
    \right).
\]
In contrast, it is crucial that the \emph{coefficients} of the generating functions remain non-negative, as otherwise, the logarithm of the sum of exponents with potentially negative weights \( a_s \) ceases to be convex.
\end{Remark}

In subsequent sections we show that having a combinatorial specification is a
for the generating function \( F(z_1, \ldots, z_d) \) is a great advantage ---
instead of requiring additional oracles providing the values of the generating
function and its derivatives, a more direct approach is available.

\subsection{Unambiguous context-free grammars}
\label{section:cf:grammars}
In the current section we refine our general tuning procedure for \emph{regular}
and \emph{unambiguous context-free} specifications, avoiding any use of external
oracles. Recall that such systems are often used in connection with the
Drmota--Lalley--Woods framework (see~\cite[Section VII.6]{flajolet09}). However,
contrary to~\cite[VII.6]{flajolet09}, we do not distinguish linear and
non-linear cases. Instead, we develop a general result allowing to cover both
cases. Also note that we do not require the system to be strongly connected,
replacing this condition by a weaker requirement that every state is reachable
from the initial one.

\begin{theorem}[Tuning with finite parameter expectation]
\label{theorem:algebraic:tuning}
Let \( \vec{\mathcal C} = \vec \Phi(\vec{\mathcal C}, \vec{\mathcal Z}) \) be a
multi-parametric system with \( \vec{\CS C} = \seq{\CS C_1,\ldots,\CS C_m} \)
and \( \vec \Phi = \seq{\Phi_1, \ldots, \Phi_m} \), where all the functions \(
\Phi_1, \ldots, \Phi_m \) are positive polynomials. Suppose that in the
dependency graph corresponding to \( \vec \Phi \) all the states are reachable from
the initial state \( \mathcal C_1 \). Let \( \vec N = (N_1, \ldots, N_k) \) be
the vector of target atom occurrences of each type. Fix the expectations \( N_i \) of
the parameters of objects sampled from \( \mathcal C_1 \) to \( \mathbb E_{\vec
z} \boldsymbol N = \boldsymbol \nu \). Then, the tuning vector \( \boldsymbol z
\) is equal to \( e^{\vec \xi} \) where \( \vec \xi \) comes from the convex
problem:
\begin{equation}
    \label{eq:optimisation:algebraic}
\begin{cases}
    c_1 - \boldsymbol \nu^\top \boldsymbol \xi \to \min_{\vec \xi, \vec c}\\
    \vec c \geq \log \vec \Phi(e^{\vec c}, e^{\vec \xi}).
\end{cases}
\end{equation}
Hereafter, \( e^{\vec \xi} \) and \( \log \vec \Phi \) denote coordinate-wise
exponentiation and logarithm, respectively.
\end{theorem}

\begin{proof}
    Consider the vector \( \vec z^\ast \) such that
    \(
    \mathbb E_{\vec z^\ast} (\vec N) = \vec \nu.
    \)
Let \( \vec c \) denote the logarithms of the values of generating functions at
point \( \vec z^\ast = e^{\vec \xi^\ast} \). Clearly, for such a choice of the
vectors \( \vec c \) and \( \vec \xi = \vec \xi^\ast \) all inequalities
in~\eqref{eq:optimisation:algebraic} become equalities.

Let us show that if the point \( (\vec c, \vec \xi ) \) is optimal, then all the
inequalities in~\eqref{eq:optimisation:algebraic} become equalities. Firstly,
consider the case when the inequality
\begin{equation}
    c_1 \geq \log \Phi_1(e^{\vec c}, e^{\vec \xi})
\end{equation}
    does not turn to an equality. Certainly, there is a \emph{gap} and the value
\( c_1 \) can be decreased without affecting the validity of other inequalities.
In doing so, the target function value is decreased as well. Hence, the point \(
(\vec c, \vec \xi) \) cannot be optimal.

Now, suppose that the initial inequality does turn to equality, however \( c_k >
\log \Phi_k(e^{\vec c}, e^{\vec \xi}) \) for some $k \neq 1$. Since each of the
states is reachable from the initial state \( \mathcal C_1 \), it means that
there exists a path \( P = c_1 \to c_2 \to \cdots \to c_k \) (indices are chosen
without loss of generality) in the corresponding dependency graph. Note that for
pairs of consecutive variables $(c_i, c_{i+1})$ in $P$, the function $\log
\Phi_i(e^{\vec c}, e^{\vec \xi})$ is strictly monotonic in $c_{i+1}$ (as it is
obtained as a log-exp transform of a positive polynomial and it references
$c_{i+1}$). In such a case we can decrease $c_{i+1}$ so to assure that $c_i >
\log \Phi_i(e^{\vec c}, e^{\vec \xi})$ while the point \( (\vec c, \vec \xi) \)
remains feasible. Decreasing $c_{i+1},c_i,\ldots,c_1$ in order, we finally
arrive at a feasible point with a decreased target function value. In
consequence, \( (\vec c, \vec \xi) \) could not have been optimal to begin with.

So, eventually, the optimisation problem reduces to minimising the expression,
subject to the system of equations
\(
    \vec c = \log \vec \Phi(e^{\vec c}, e^{\vec \xi})
\)
or, equivalently,
\(
    \vec C(\vec z) = \vec \Phi(\vec C(\vec z), \vec z)
\)
and can be therefore further reduced to~\cref{theorem:general}.
\end{proof}

\begin{Remark}
  Let us note that the above theorem extends to the case of labelled structures
with \( \Set \) and \( \Cycle \) operators. For unlabelled Pólya operators
like \( \MSet \) or \(\Cycle \), we have to truncate the specification to bound
the number of substitutions. In consequence, it becomes possible to sample
corresponding unlabelled structures, including partitions, functional graphs,
series-parallel circuits, etc. For more details,
see~\cref{section:implementation:details}.
\end{Remark}

 Singular Boltzmann samplers (also defined in~\cite{DuFlLoSc}) are the limit
variant of ordinary Boltzmann samplers with an infinite expected size of
generated structures. In their multivariate version, samplers are considered
\emph{singular} if their corresponding variable vectors belong to the boundary
of the respective convergence sets. We present an alternative option to tune
such singular samplers corresponding to Drmota--Lalley--Woods
framework~\cite[Section VII.6]{flajolet09}, provided that their corresponding
dependency graphs are strongly connected.

\begin{theorem}[Tuning with infinite parameter expectation]
\label{theorem:singular:tuning}
    Let \( \vec{\mathcal C} = \vec \Phi(\vec{\mathcal C}, \mathcal{Z},
\vec{\mathcal U}) \) be a strongly connected multi-parametric system of positive
polynomial equations with \( \vec{\CS C} = \seq{\CS C_1,\ldots,\CS C_m} \), the
atomic class \(\mathcal{Z}\) marking the corresponding structure size, and \(
\vec{\CS U} = \seq{\CS U_1,\ldots,\CS U_k} \) being a vector (possibly empty) of
distinguished atoms. Assume that the target expected frequencies of the atoms \(
\mathcal U_i \) are given by the vector \( \boldsymbol \alpha \). Then, the
variables \( (z, \vec u) \) that deliver the tuning of the corresponding
singular Boltzmann sampler are the result of the following convex optimisation
problem, where \( z = e^{\xi} \), \( \vec u = e^{\vec \eta} \):
\begin{equation}
\begin{cases}
\xi + \vec \alpha^\top \vec \eta \to \max_{\xi, \vec \eta, \vec c}\\
    \vec c \geq
    \log \vec \Phi(e^{\vec c}, e^{\xi}, e^{\vec \eta}).
\end{cases}
\end{equation}
\end{theorem}

\begin{proof}
By similar reasoning as in the previous proof, we can show that the maximum is
attained when all the inequalities turn to equalities. Indeed, suppose that at
least one inequality is strict, say
\(
    c_j > \log \Phi_j(e^{\vec c}, e^{\xi}, e^{\vec \eta})
    \).

Because all right-hand sides of each of the inequality are monotonic with
respect to $c_j$, we note that the value \( c_j \) can be slightly decreased by
choosing a sufficiently small distortion \( \varepsilon > 0 \), turning all the
equalities containing \( c_j \) in the right-hand side \( \log \Phi_i(e^{\vec
c}, e^{\xi}, e^{\vec \eta}) \) into strict inequalities. Clearly, we can repeat
this process until all equalities turn into inequalities.

Finally, let us focus on the target function. Again, because all right-hand
sides of each inequality are monotonic with respect to $\xi$, we can slightly
increase its value and increase the target function so to remain inside the
feasible set.

Let us fix \( \vec u = e^{\vec \eta} \). For rational and algebraic grammars,
within the Drmota--Lalley--Woods framework, see for
instance~\cite{drmota1997systems}, the corresponding generating function
singular approximation takes the form
\begin{equation}
\label{eq:asymptotic:singular}
    C(z, \vec u) \sim a_0(\vec u) - b_0(\vec u) \left(1 - \dfrac{z}{\rho(\vec
    u)}\right)^{t}.
\end{equation}
If \( t < 0 \), then the asymptotically dominant term becomes \( -b_0 \left( 1 -
\frac{z}{\rho(\vec u)} \right)^{t}\). In this case, tuning the target expected
frequencies corresponds to solving the following equation as $z \to \rho(u)$:
\begin{equation}
\label{eq:asymptotic:expectation}
\mathrm{diag}(\vec u)
    \dfrac
    { [z^n] \nabla_{\vec u} C(z, \vec u) }
    { [z^n] C(z, \vec u) } = n \vec \alpha.
\end{equation}
Let us substitute the asymptotic expansion~\eqref{eq:asymptotic:singular}
into~\eqref{eq:asymptotic:expectation} to track how \( \vec u \) depends on \(
\vec \alpha \):
\begin{equation}
\mathrm{diag}(\vec u)
    \dfrac
        {
            [z^n]
            t b_0(\vec u)
            \left(
                1 - \dfrac{z}{\rho(\vec u)}
            \right)^{t-1}
            z \dfrac{\nabla_{\vec u}\rho(\vec u)}{\rho^2(\vec u)}
        }
        {
            [z^n]
            b_0(\vec u)
            \left(
                1 - \dfrac{z}{\rho(\vec u)}
            \right)^{t}
        }
        = - n \vec \alpha.
\end{equation}
Only dominant terms are accounted for. Then, by the binomial theorem
\begin{equation}
\mathrm{diag}(\vec u)
 b_0(\vec u)
 \dfrac{t}{n} { t-1 \choose n }
    \dfrac{z\nabla_{\vec u}\rho(\vec u)}{\rho^2(\vec u)}
    b_0(\vec u)^{-1}
    { t \choose n}^{-1}
    = - \vec \alpha.
\end{equation}
With \( z = \rho(\vec u) \), as \( n \to \infty \),
we obtain
after cancellations
\begin{equation}
    \mathrm{diag}(\vec u)
    \dfrac
        {\nabla_{\vec u} \rho(\vec u)}
        {\rho(\vec u)}
    = - \vec \alpha
\end{equation}
which can be rewritten as
\begin{equation}\label{eq:asymptotic:expectation:two}
    \nabla_{\vec \eta} \log \rho( e^{\vec \eta} ) = -\vec \alpha.
\end{equation}
Passing to exponential variables~\eqref{eq:asymptotic:expectation:two} becomes
\begin{equation}
    \nabla_{\vec \eta} ( \xi(\vec \eta) + \vec \alpha^\top \vec \eta ) = 0.
\end{equation}
As we already discovered, the dependence \( \xi(\vec \eta) \) is given by the
system of equations because the maximum is achieved only when all inequalities
turn to equations. That is, tuning the singular sampler is equivalent to
maximising \( \xi + \vec \alpha^\top \vec \eta \) over the set of feasible
points.
\end{proof}

\begin{Remark}
    For ordinary and singular samplers, the corresponding feasible set remains
the same; what differs is the optimised target function. Singular samplers
correspond to imposing an infinite target size. In practice, however, the
required singularity is almost never known \emph{exactly} but rather calculated
up to some feasible finite precision. The tuned structure size is therefore
enormously large, but still, nevertheless, finite. In this context, singular
samplers provide a natural \emph{limiting} understanding of the tuning
phenomenon and as such, there are several possible ways of
proving~\cref{theorem:singular:tuning}.
\end{Remark}

\autoref{fig:binary:picture} illustrates the feasible set for the class of
binary trees and its transition after applying the log-exp transform, turning
the set into a convex collection of feasible points.  In both figures, the
singular point is the rightmost point on the plot. Ordinary sampler tuning
corresponds to finding the tangent line which touches the set, given the angle
between the line and the abscissa axis.
\begin{figure}[hbt!]
    \begin{center}
        \includegraphics[width=0.49\textwidth]{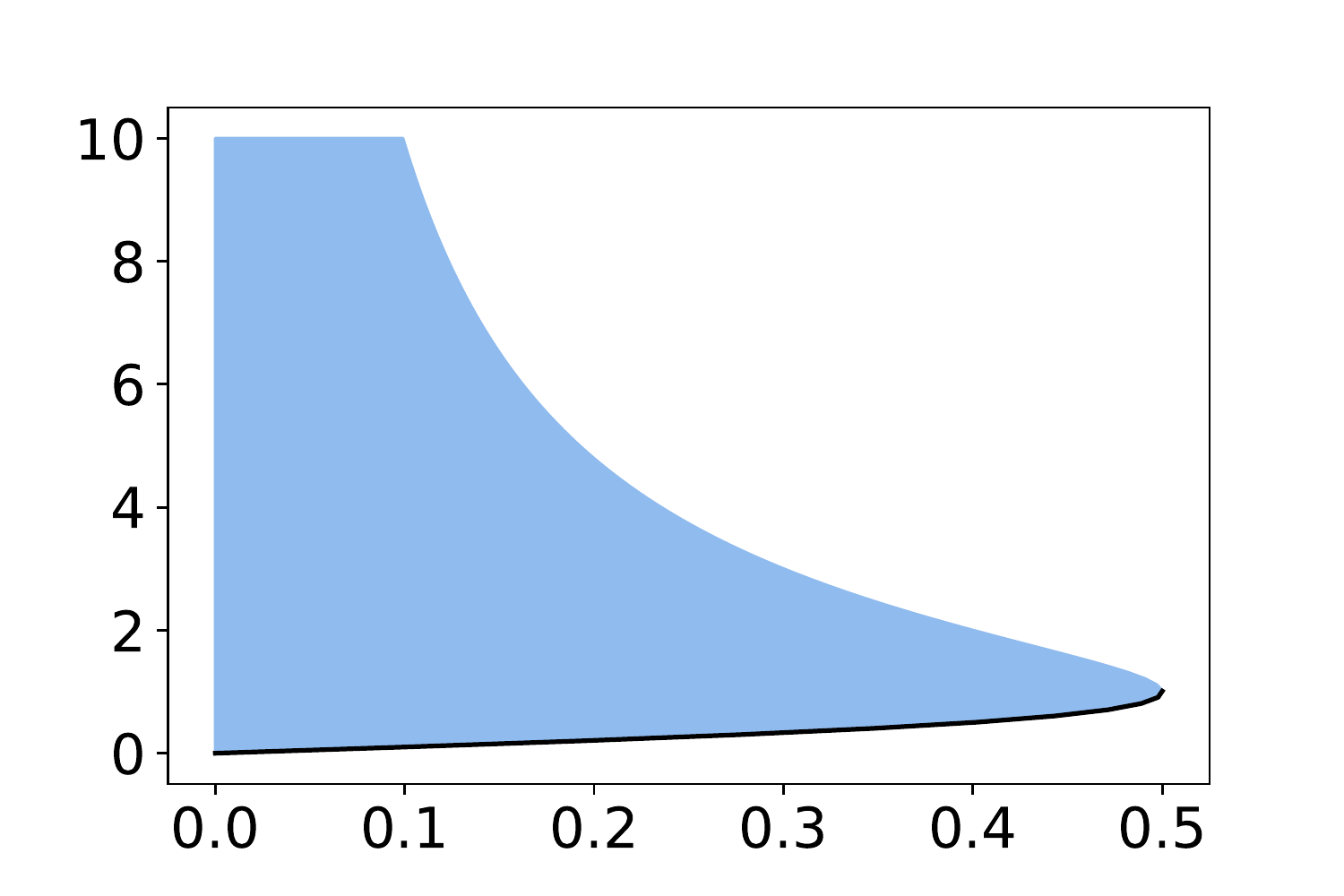}
        \includegraphics[width=0.49\textwidth]{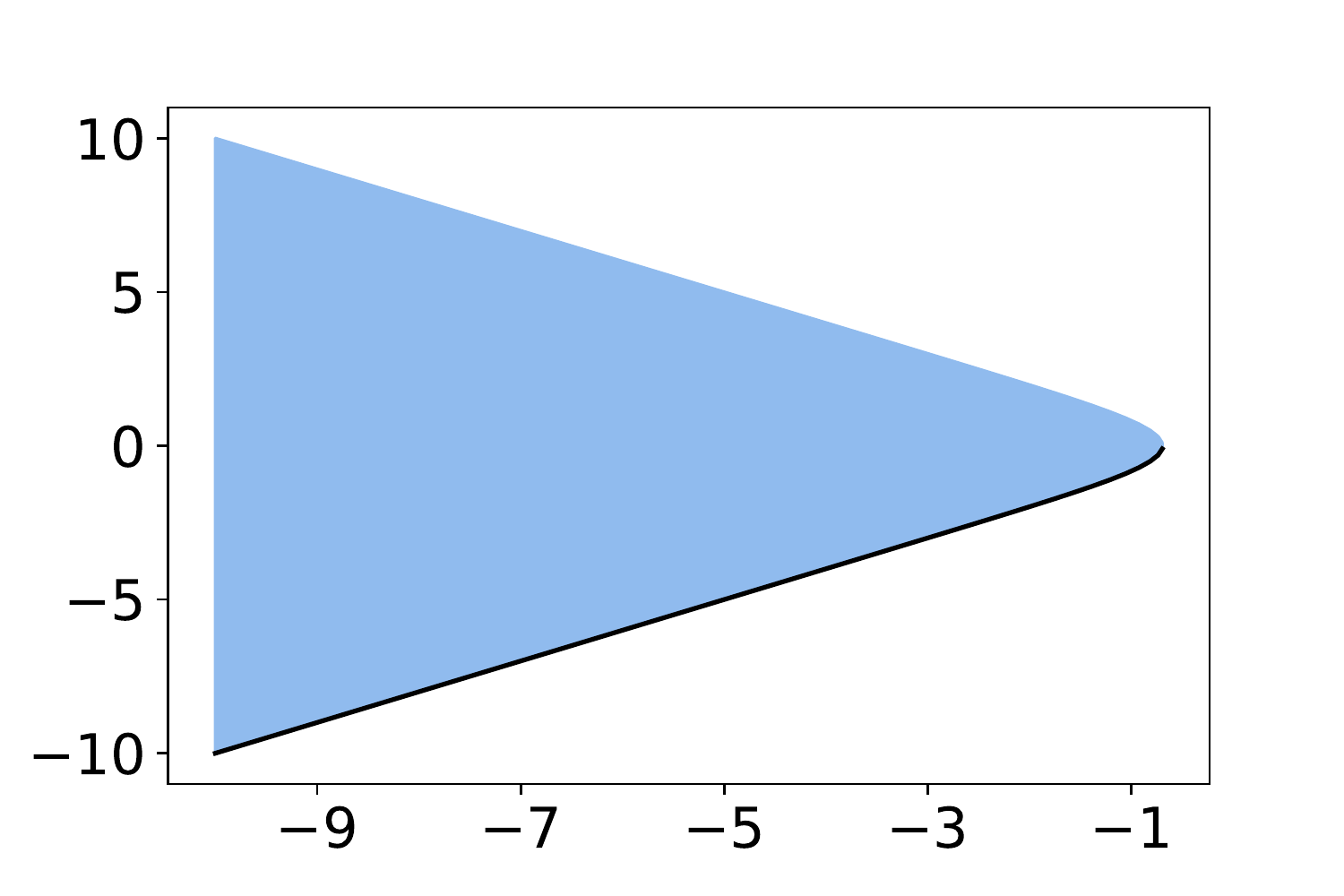}
    \end{center}
    \caption{Binary trees \( B \geq z + zB^2 \) and log-exp transform of the
    feasible set. The black curve denotes the principal branch of the generating
function $B(z)$ corresponding to the class of binary trees.}
    \label{fig:binary:picture}
\end{figure}

\begin{Remark}
    As an interesting by-product, our tuning algorithm provides a way of
obtaining the singularities of a system and the values of the generating
functions at the point of this singularity.
\end{Remark}

\begin{Remark}
    It is also possible to consider singular tuning the case of non-strongly
connected specifications. However, practically speaking, it should be noted that
a notion of a \emph{singular sampler} for non strongly-connected specifications
such as
    \begin{equation}
        F(z) = \dfrac{1}{1 - 2T(z)} \quad \text{and} \quad T(z) = z e^T(z)
      \end{equation}
    is ambiguous --- both singular samplers for \( T \) and \( F \) admit
    different values of the variable \( z \). For $T$ its \( z = e^{-1} \) whereas
    for \( F \) we have \( z = e^{-1/2}/2 \).

If the substitutions in the dependency graph of the specification include only
\emph{subcritical compositions} (see~\cite[Section VI.9]{flajolet09}) it is
possible to incrementally tune its parts in topological order. Further
theoretical analysis of singular samplers involving supercritical compositions
is somewhat more complicated, but can be developed as well. Nevertheless, the
easiest practical way to derive a singular tuner is to tune the target class
with a large, yet finite object size.
\end{Remark}

\subsection{Labelled and unlabelled structures}
Systems originating from labelled specifications, i.e.~whose generating
functions are typically of exponential type, feature such admissible operators
as $\Set$ or $\Cycle$, both in their unrestricted and cardinality restricted
variants. Consider a labelled multi-parametric combinatorial class \( \mathcal A
\) with a generating function \( A(\vec z) \). Then, the resulting exponential
generating functions obtained using these operators take form
\begin{gather}
  \begin{aligned}
    \Set\nolimits_{k}(\mathcal A)(\vec z) &= \dfrac{A(\vec z)^k}{k!} \quad&
    \Cycle\nolimits_k(\mathcal A)(\vec z) &= \dfrac{A(\vec z)^k}{k}\\
    \Set(\mathcal A)(\vec z) &= e^{A(\vec z)} \quad& \Cycle(\mathcal A)(\vec z) &=
    \log \dfrac{1}{1 - A(\vec z)}.
  \end{aligned}
\end{gather}
Classes whose definition involves one of the above operators can be incorporated
into the convex optimisation problem using a log-exp transformation \( F(z)
\mapsto e^{\varphi} \) similarly to the case context-free grammars,
see~\cref{section:cf:grammars}. Broadly speaking, the application of
such operators results in a composition with one of the basic functions
  \begin{equation}
    \dfrac{x^k}{k!}, \quad
    \dfrac{x^k}{k}, \quad
    e^x \quad \text{and} \quad
    \log \dfrac{1}{1 - x}
  \end{equation}
  expressing, respectively, the exponential generating functions for the class
of labelled sets with \( k \) elements, labelled cycles with \( k \) elements,
and both unrestricted labelled sets and cycles.

On the other hand, ordinary generating functions, used for enumeration of
unlabelled structures, feature such operators as $\MSet$, $\PSet$, and $\Cycle$,
standing for the multiset, set, and a cycle constructions, respectively. These,
in contrast, are are evaluated differently than their labelled counterparts.
Specifically, if applied to a class with an ordinary generating
function \( A(\vec z) = \sum_{\vec n \succ \vec 0} a_{\vec n} \vec z^{\vec n}
\), we obtain, respectively,
  \begin{equation}\label{eq:unlabelled:structures:expr}
    \exp\left(
        \sum_{k \geq 1} \dfrac{A(\vec z^k)}{k}
    \right), \quad
    \exp\left(
        \sum_{k \geq 1} \dfrac{(-1)^{k-1}A(\vec z^k)}{k}
    \right) \quad \text{and} \quad
    \sum_{k \geq 1}
    \dfrac{\varphi(k)}{k} \log \dfrac{1}{1 - A(\vec z^k)}.
  \end{equation}
  Here \( \varphi(k) \) denotes Euler's totient function.
  
Let us note that bot the first and the third expressions, after log-exp
transformations, become convex. The resulting infinite series can be then
truncated at a finite threshold, given the fact that if the corresponding
singular value $\boldsymbol{\rho} < 1$, the sequence \( A(\vec z^k) \) converges
at geometrical speed to \( A(\boldsymbol{0}) \). Such a truncation is common
practice and has been applied a number of time in the context of sampling
unlabelled structures, see \eg~\cite{BodLumRolin,flajolet2007boltzmann}. A more
detailed discussion regarding these transformations is given
in~\cref{section:implementation:details}.

Unfortunately, the remaining \( \PSet \) operator, in its aforementioned series
form with negative coefficients~\eqref{eq:unlabelled:structures:expr}, does not
easily fit the convex programming framework. In principle, an alternative form
could be used
  \begin{equation}
    \PSet(\mathcal A)(z) =
    \prod_{\vec n \succ \vec 0}
    \left(
        (1 + z_1^{n_1})
        \cdots
        (1 + z_d^{n_d})
    \right)^{a_{\vec n}}
  \end{equation}
though in this case, it is not clear how to convert the problem into a polynomially
tractable form.

\subsection{Increasing trees}
Boltzmann samplers for first- and second-order differential specifications have
been developed in~\cite{bodini2012boltzmann,bodini2016increasing}. In
particular, in~\cite{bodini2012boltzmann}, the authors solve the problem of
Boltzmann sampling from specifications of type \( \CS T' = \CS F(\CS Z, \CS T)
\). There exist several particular cases of such a differential specification
which admit an explicit solution of the corresponding differential equation. For
example, a differential equation is said to be \emph{autonomous} if it does not depend explicitly on the independent
variable, i.e.~\( \CS T' = \CS F(\CS T) \). In this case, the underlying differential
equation can be solved by separation of variables:
\begin{equation}
    \label{eq:autonomous:solution}
    \begin{cases}
    T'(z) = F(T(z)) \\
    T(0) = t_0
    \end{cases}
    \quad \Rightarrow \quad
    \int_{t_0}^{T} \dfrac{dt}{F(t)} = z(T).
\end{equation}
The final solution \( T(z) \) is obtained by inverting \( z(T) \), which, in the
case of a differential equation with combinatorial origin, can be obtained using
binary search.

Several different strategies for evaluating \( \int\frac{1}{F(t)} \) may be also
suggested. For instance, if \( F(t) \) is a polynomial in \( t \), all its
complex roots (together with multiplicity structure) can be efficiently found
through a numerical procedure\footnote{ Efficient polynomial root-finding is
directly related to Smale's 17th problem about finding complex solutions of
complex polynomial systems, recently positively solved by Lairez,
see~\cite{lairez2017deterministic}.}, see~\cite{zeng2005computing}. In the
context of multi-parametric tuning, we are mostly interested in the case when \(
F(T) \) also depends on auxiliary variables \( \vec u=(u_1, \ldots, u_d) \),
which need to be tuned. In this case, the roots of the polynomial \( F(T) \)
become dependent on \( \vec u \).

\begin{theorem}[Multi-parametric increasing trees]
    \label{theorem:increasing:trees}
Suppose that the generating function \( T(z, \vec u) \) corresponding to a
family of increasing trees is described by a functional equation
  \begin{equation}
    T'_z(z, \vec u) = F(T(z, \vec u), \vec u) \quad \text{with} \quad T(0) = t_0
  \end{equation}
and that \( T \) is a formal power series with non-negative coefficients, \(
\vec u = (u_1, \ldots, u_d) \).

Let $N_0, N_1,\ldots,N_d$ be the excepted size, and expected parameter value of
$u_1, \ldots,u_d$, respectively. Then, the tuning problem is equivalent to a convex
optimisation problem over real variables \( \varphi, \zeta, \eta_0, \ldots,
\eta_d \):
\begin{equation}
    \begin{cases}
        \varphi - N_0 \zeta - N_1 \eta_1 - \cdots - N_d \eta_d \to \min,\\
        \zeta \leq
        \log \displaystyle\int\limits_{t_0}^{e^\varphi}
        \dfrac{dt}{F(t, e^{\eta_1}, \ldots, e^{\eta_d})}
    \end{cases}
\end{equation}
where the solution of the tuning problem is obtained by assigning
  \begin{equation}
    z = e^{\zeta}, \quad u_1 = e^{\eta_1},
    \quad \ldots, \quad u_d = e^{\eta_d}, \quad \text{and} \quad
    T(z, u_1, \ldots, u_d) = e^{\varphi}.
  \end{equation}
\end{theorem}
\begin{proof}
Following~\cref{theorem:tuning}, multi-parametric tuning is equivalent to the
following convex optimisation problem over real variables \( \varphi, \zeta,
\eta_1, \ldots, \eta_d \):
  \begin{equation}
    \begin{cases}
        \varphi - N_0 \zeta - N_1 \eta_1 - \cdots - N_d \eta_d \to \min,\\
        \varphi \geq \log T(e^{\zeta}, e^{\eta_1}, \ldots, e^{\eta_d})
    \end{cases}
  \end{equation}
where \( T(z, u_1, \ldots, u_d) \) is the solution of differential equation
\(
    T'_z(z, \vec u) = F(T(z, \vec u), \vec u)
\) with initial conditions \( T(0, \vec u) = t_0 \).
The target solution is given by
\( (z, u_1, \ldots, u_d) = (e^{\zeta}, e^{\eta_1}, \ldots, e^{\eta_d}) \)
and \( F(z, u_1, \ldots, u_d) = e^{\varphi} \).
Let us denote by \( Z(\tau, \eta_1, \ldots, \eta_d) \) the inverse function
of \( \log T(e^\zeta, e^{\vec \eta}) \) with respect to \(
\zeta \), so that
  \begin{equation}
    \log T(e^{Z(\tau, \eta_1, \ldots, \eta_d)}, e^{\eta_1}, \ldots, e^{\eta_d})
    = \tau.
  \end{equation}
Since \( T \) is a formal power series with non-negative coefficients, the
function \( \log T(e^\zeta, e^{\eta_1}, \ldots, e^{\eta_d}) \) is convex and
increasing in both \( \zeta \) and \( \vec \eta \). Therefore, its inverse
function with respect to \( \zeta \) is a concave increasing function. Moreover,
the function \( Z(\tau, \eta_1, \ldots, \eta_d) \) is jointly concave in
\emph{all} of its arguments That is because a function \( f(\vec z) \) is
concave if and only if its hypograph \( \mathbf{hyp} f = \{ (y, \vec z) \colon y
\leq f(\vec z) \} \) is a convex set.

The hypograph of the inverse function \( Z(\tau, \vec \eta) \) directly
corresponds to the epigraph of the initial function \( \log T(e^z, e^{\vec
\eta}) \). Consequently, a convex constraint \( \varphi \geq \log T(e^{\zeta},
e^{\eta_1}, \ldots, e^{\eta_d}) \) can be replaced by a different convex
constraint obtained by taking the inverse with respect to \( \zeta \), resulting
in a new convex optimisation problem
  \begin{equation}
    \begin{cases}
        \varphi - N_0 \zeta - N_1 \eta_1 - \cdots - N_d \eta_d \to \min,\\
        \zeta \leq Z(\varphi, \eta_1, \ldots, \eta_d).
    \end{cases}
  \end{equation}
In order to construct \( Z(\varphi, \eta_1, \ldots, \eta_d) \), we use the
solution of the autonomous differential equation,
see~\eqref{eq:autonomous:solution}. Denote by \( z(t, \vec u) \) the inverse of
the generating function \( T(z, \vec u) \), \ie
  \begin{equation}
    T(z(t, \vec u), \vec u) = t
  \end{equation}
  and recall that
\( z(t, \vec u) \) is given by
  \begin{equation}
    z(t, \vec u) = \int_{t_0}^t \dfrac{d \tau}{F(\tau, \vec u)}.
  \end{equation}
Finally, a direct calculation shows that
\(
    Z(\varphi, \eta_1, \ldots, \eta_d) = \log z(e^\varphi, e^{\eta_1}, \ldots,
    e^{\eta_d})
\).
\end{proof}

\begin{Remark}
  With an external oracle available, systems of differential equations
involving more than one variable can still be tuned using~\cref{theorem:tuning}.
However, there is little hope to generalise~\cref{theorem:increasing:trees} onto
multivariate differential equations, even if we assume that all its differential
equations are autonomous. In fact, by introducing an additional dimension to the
problem, any first-order system of differential equations
\begin{equation}
\frac{d}{dz} \vec y(z) = \vec F(z, \vec y(z))
\end{equation}
can be transformed into a system of autonomous
differential equations. Moreover, starting from dimension four, systems of
functional equations can admit solutions which exhibit chaotic behaviour,
see~\cite[Remark VII.51]{flajolet09}.

On the other hand, any system of differential equations of arbitrarily high
order
\begin{equation}
\vec y^{(n)}(z) = \vec F(z, \vec y(z), \vec y^{(1)}(z), \ldots, \vec
y^{(n-1)}(z))
\end{equation}
    can be reduced to a system of first-order differential equations by
expanding the dimension space, and therefore, Boltzmann samplers for such
specifications may apply. As discussed in~\cite{bodini2012boltzmann}, for such
systems, only the first sampling step is the most expensive one, because it
requires to generate a random variable defined on the interval \( [0, \rho) \),
where \( \rho \) is the radius of convergence of the formal power series. For
all the consecutive operations, the support of the required random variables is
separated from \( \rho \), so a variety of approximation methods, including
Runge--Kutta can be applied.
\end{Remark}

\subsection{Other types of combinatorial structures}
The general technique described in~\cref{section:tuning:as:convex:optimisation}
can be applied to any analytic multivariate generating function provided that the solution to the multiparametric tuning problem lies inside the respective domain of analyticity.
It can be applied even to those
coming from somewhat exotic systems, including partial differential equations,
systems with negative coefficients, catalytic equations, systems with
non-trivial substitutions, etc. As long as the oracle providing values and the
derivatives of the generating functions is given, the source of the equations is
irrelevant.

Still, for most of these systems no efficient oracle is known. Typically, the
following methods, which could be called the \emph{oracles} for simplicity, can be
used for the mentioned types of functional equations
\begin{itemize}
    \item Runge--Kutta method for ordinary differential equations,
    \item Grid approximation methods for partial differential equations,
    \item Gr\"obner bases for systems of polynomial equations, or
    \item Plain coefficient-wise evaluation for arbitrary systems.
\end{itemize}

All of these methods, however, do not guarantee polynomial complexity in terms
of the bit length of the target size and the number of equations. When
additional parameters are introduced, then the approach that we propose
in~\cref{section:tuning:as:convex:optimisation} reduces the number of
equations for an oracle to solve, but adds the optimisation procedure on the
top. Specifically, with this approach, the tuning system of
equations~\eqref{eq:tuning:system:th} is not anymore included as an input to
the oracle, and the tuning parameters are chosen by the optimisation
procedure, and so are considered fixed real numbers from the viewpoint of the
oracle.

\begin{Remark}
\label{remark:non:tunable}
It is worth noting that for certain combinatorial classes beyond the scope of the current paper, the tuning problem does not have a solution in the analyticity domain. The simplest non-trivial example (with an infinite number of objects in a class) is the case of unrooted trees \( U(z) = \sum_{n \geq 0} n^{n-2} \frac{z^n}{n!} \): near its dominant singular point \( z = e^{-1} \), the Puiseux expansion takes form
\[
    U(z) = \dfrac{1}{2} - (1 - ez) + \dfrac{2^{3/2}}{3} (1 - ez)^{3/2}
    + O((1-ez)^2).
\]
This means that the maximal possible expected value of an unrooted tree generated from Boltzmann distribution is equal to
\[
    \lim_{z \to e^{-1}}
    \dfrac{z \partial_z U(z)}{U(z)} = 2.
\]
However, for most examples of this kind, the Boltzmann sampler itself is not directly available, not speaking of the fact that there would be no practical motivation to apply it for expectations higher than the maximal available value.

Still, there exist techniques different from the Boltzmann sampling which can be used to draw the objects from such classes in a controlled manner.
The reason why the solution does not exist in the above example is because locally around the singularity \( \rho \), the singular term of the generating function takes form
\( (1 - \frac{z}{\rho})^\alpha \), where \( \alpha > 1 \). By repeatedly applying pointed derivative to such a class, the singular term can be modified, and the exponent \( \alpha \) shall be consequently replaced by a value strictly smaller than \( 1 \). From the combinatorial viewpoint, pointed derivative serves to distinguish atoms in the underlying object, which can be un-distinguished later after the sampling is performed. The possible drawbacks of this approach are that
firstly, successive pointing can increase the size of the grammar,
and secondly, the resulting distribution is not anymore Boltzmann, which is not surprising, because the tuning problem does not have a solution in the Boltzmann case.

It is still important to discover and understand controlled sampling techniques in the cases not covered by admissible tuning in the sense of~\eqref{eq:tuning:system:th}: when inclusion-exclusion is used, or even in more sophisticated cases involving stretched exponents or any kinds of essential singularities.
\end{Remark}

\section{Complexity of convex optimisation}
\label{section:barriers}
In the previous section, we have reduced the problem of multiparametric tuning
to a problem of convex optimisation. In this section, we provide the
complexity results, case by case. Let us introduce the key elements of the
interior-point optimisation and discuss how to evaluate its complexity.

A convex optimisation problem can be expressed as minimisation of a linear multivariate function (or, in an equivalent formulation, a convex one) subject to a set of
constraints, expressed as inequalities of the form \( f_i(\vec z) \leq 0 \) involving multivariate convex functions \( f_i(\vec z) \). In addition, affine inequalities and linear equations are often treated separately.
Each of these constraints
describes a convex set, and therefore, an intersection of these sets is also
convex. A classical convex optimisation program can be formulated as
\begin{equation}
\label{eq:convex:opt}
    \begin{cases}
        \vec c^\top \vec z \to \min_{\vec z}\\
        f_i(\vec z) \leq 0 \text{ for } i = 1, \ldots, m
    \end{cases}
\end{equation}
where \( f_1, \ldots, f_m \) are convex functions of the vector argument \( \vec
z \). Such programs are widely believed to be \emph{solvable} in polynomial
time, due to the popularisation of interior-point methods. Nevertheless, a
detailed answer regarding the algorithmic complexity of convex optimisation
needs a careful investigation.

For instance, if \( f_i(\vec z) \) are linear functions with integer
coefficients, then the solution to the optimisation problem is a rational
number. In this case the problem is known to be solvable \emph{exactly} using \(
O(n^{3.5} L) \) arithmetic operations, where \( n \) denotes the number of
variables, and \( L \) is the bit length of the coefficient
representation~\cite{karmarkar1984new}. Remarkably, the existence of a
polynomial algorithm whose arithmetic complexity does not depend on \( L \)
(so-called \emph{strongly-polynomial algorithms}) is listed among 18 problems in
Smale's celebrated list~\cite{smale1998mathematical}. If the functions \(
f_i(\vec z) \) are not necessarily linear, the optimal solution does not have to
be a rational number and so we usually ask for an \( \varepsilon
\)-approximation, instead.

Nesterov and Nemirovskii~\cite{nesterov1994interior} developed a theoretically
and practically efficient framework covering a large subset of convex programs.
Their method relies on the notion of a \emph{self-concordant barrier} whose
construction depends on the problem formulation and needs to exploit the global
structure of the optimisation problem. The number of Newton iterations required
by their method to solve the optimisation problem~\eqref{eq:convex:opt} can be
bounded by
\begin{equation}
    O\left(
    \sqrt{\nu} \log\left(
        \dfrac{\nu \mu_0}{\varepsilon}
    \right)
    \right)
  \end{equation}
where \( \varepsilon \) is the required precision, \( \nu \) is the complexity
of the constructed barrier (typically proportional to the problem dimension),
and \( \mu_0 \) is related to the choice of the starting point. More
specifically, the value of the target function is guaranteed to lie within an \(
\varepsilon \)-neighbourhood of the optimal value. Since \( \mu_0 \) does not
depend on the target precision, it is often omitted in the final
analysis~\cite{potra2000interior}. Further on, we compute \( \nu \) for
various types of combinatorial specifications.

The cost of each Newton iteration is a sum of the two parts: the cost of the
composition of the matrix of second derivatives and the cost of the inversion
of this matrix. In~\cref{section:self:concordant}, we specify in more detail
what is a Newton iteration in the interior-point method. While the cost of the
inversion part is assumed to be cubic in the number of variables,
the cost of the composition of the matrix of second derivatives depends on the problem
structure.

One of the most significant achievements of interior-point programming was a
proof of existence of a \emph{universal family of barriers}, which admits a \(
O(n) \) self-concordance complexity for any convex body in \( \mathbb R^n \).
Unfortunately, such an existence result seems to be currently only of
theoretical interest --- a construction of such a barrier requires computing the
convex body's volume, which in itself is known to be an NP-hard problem. In
practice, for each concrete convex optimisation problem the barriers have to be
constructed separately. There is no known constructive, general-purpose
polynomial-time barrier construction algorithm for convex optimisation problems.

To summarise, the practical complexity of the convex optimisation is determined
by the complexity parameter \( \nu \) of the associated self-concordant barrier,
which requires the problem structure to be exploited in order to be constructed.
More details on the algorithmic complexity of convex optimisation methods can be
found in~\cite{potra2000interior}.

\subsection{Disciplined Convex Programming}
Disciplined Convex Programming ({\sf DCP}, sometimes also referred to as {\sf
CVX}) is another framework meant for modelling convex optimisation
problems~\cite{grant2006disciplined}. Both the objective and constraints are
build using certain \emph{atomic expressions} such as $\log(\cdot)$,
$\|\cdot\|_2$, or $\log \sum e^{\cdot}$, and a restricted set of
\emph{composition rules} following basic principles of convex analysis. The
inductive definition of expressions allows the framework to track their
curvature and monotonicity, and, in consequence, automatically determine whether
the given program is indeed a valid convex program. Essentially, {\sf DCP} can
be viewed as both a domain-specific language and a compiler which transforms the
given convex program to a mixture of linear, quadratic, and exponential conic
optimisation problems. Such problems are passed to corresponding solvers in the
necessary standard form. {\sf DCP} covers therefore a strict subset of convex
programs of Nesterov and Nemirovskii's framework, but serves well as a
prototyping interface for provably convex programs. In our implementation,
see~\Cref{sec:paganini}, we rely on {\sf DCP} with two particular conic solvers,
a second-order Embedded Conic
Solver~\cite{domahidi2013ecos} and recently developed first-order Splitting
Conic Solver~\cite{o2016conic}.

Some of the classical combinatorial constructions, for example \( \MSet \), \(
\Cycle \) or \( \Set_{>0} \) can be expressed in {\sf DCP}, however involve
using (at least in principle) an infinite amount of summands. Some of the
operators, for example, \( \MSet \) can be efficiently represented in a
truncated form, cf.~\cite{BodLumRolin}, but for others it is intrinsically
impossible to provide a reasonable truncation level, since a large number of
summands take non-negligible values. Therefore, for such constructions {\sf DCP}
is not the preferable, and the classical interior-point methods should be selected instead.
In what follows, we provide some additional theoretical
background on self-concordant barriers and propose our constructions which can
be further used to cover these operations.

\subsection{Self-concordant barriers}
\label{section:self:concordant}
While the precise descriptions of the interior-point optimisation schemes and
the definitions of the self-concordant barriers can be found
in~\cite{nemirovski2004interior}, we only need to borrow a couple of statements.
We intentionally do not give the definition of a self-concordant barrier, but
instead use a sufficient construction condition. More specifically, we rely on
the fact that self-concordant barriers are functions assigned to various convex
sets in \( \mathbb R^n \) and we will only need to construct such sets for
\emph{epigraphs} of convex functions, defined as
\begin{equation}
\epi f = \{
(t, x) \in \mathbb R^2 \, \mid \,
f(x) \leq t
\}.
\end{equation}

\begin{lemma}[{\cite[Proposition 9.2.2]{nemirovski2004interior}}]\label{lemma:nemirovski}
    Let \( f(x) \) be a three times continuously differentiable real-valued
convex function on the ray \( \{x > 0\} \) such that
      \begin{equation}
        |f'''(x)| \leq 3 \beta x^{-1} f''(x), \, x > 0.
      \end{equation}
      for some parameter $\beta > 0$.
      
    Then, the function
      \begin{equation}
        F(t, x) = - \log(t - f(x)) - \max[1, \beta^2] \log x
      \end{equation}
    is a self-concordant barrier with complexity parameter
    \( \nu = 1+\max[1, \beta^2] \) for the two-dimensional convex domain
       $\{(t, x) \in \mathbb R^2 \, \mid \,
            x > 0, \, f(x) \leq t \}$.
\end{lemma}

The two important operations for composing convex problems are addition and
composition (the composition of a convex function with an \emph{increasing}
convex function is also convex). Each composition is treated by introducing an
additional variable; for example, if \( F(x) \) is convex and increasing, and \(
G(x) \) is convex, then the epigraph of the composition \( \{ (t, x) \, \mid \,
F(G(x)) \leq t \} \) can be expressed as a projection of a three-dimensional
set
  \begin{equation}
    \{
        (x, t) \, \mid \, \exists y \, \mid \,
        F(y) \leq t, \, G(x) \leq y
    \}.
  \end{equation}
The behaviour of self-concordant barriers with respect to taking linear
combinations or with respect to combining multiple dimensions
is described by the following proposition.

\begin{lemma}[{\cite[Proposition 3.1.1]{nemirovski2004interior}}]
    \label{lemma:stability:of:barriers}
    \renewcommand{\theenumi}{(\roman{enumi}}%
    The following propositions hold:
    \begin{enumerate}
        \item Let \( F_i \) be self-concordant barriers with complexity
            parameters \( \nu_i \) for closed convex domains \( Q_i \subset
            \mathbb R^{n_i} \), \( i = 1, \ldots, m \). Then the function
              \begin{equation}
                F(x_1, \ldots, x_m) = F_1(x_1) + \cdots + F_m(x_m)
              \end{equation}
            is a \( (\sum_i \nu_i) \)-self-concordant barrier for
            \( Q = Q_1 \times \cdots \times Q_m \).
        \item Let \( F_i \) be self-concordant barriers with complexity
            parameters \( \nu_i \) for closed convex domains \( Q_i \subset
            \mathbb R^{n} \) and \( \alpha_i \geq 1 \) be reals, \( i = 1,
            \ldots, m \). Assume that \( Q = \cap_{i=1}^m Q_i \) has a non-empty
            interior. Then the function
              \begin{equation}
                F(x) = \alpha_1 F_1(x) + \cdots + \alpha_m F_m(x)
              \end{equation}
            is a \( (\sum_i \alpha_i \nu_i) \)-self-concordant barrier for
            \( Q \).
    \end{enumerate}
    \renewcommand{\theenumi}{\alph{enumi}}%
\end{lemma}

Self-concordant barrier design is one of the central issues in modern convex
optimisation --- knowing these barriers allows one to reduce any convex program
to a standard form~\cite[Section 4.2.6]{nesterov1998introductory}. The reduction
proceeds as follows.
  \begin{equation}
    \begin{cases}
        f_0(\vec z) \to \min \\
        f_i(\vec z) \leq 0 \text{ for } i = 1, \ldots, m
    \end{cases}
    \quad \Leftrightarrow \quad
    \begin{cases}
        \tau \to \min \\
        f_0(\vec z) \leq \tau \\
        f_i(\vec z) \leq \kappa \text{ for } i = 1, \ldots, m \\
        \tau \leq C, \, \kappa \leq 0.
    \end{cases}
  \end{equation}
Here, \( C \) is a technical constant, bounding the value of the target function
from above. Suppose that \( F_0(\vec z, \tau) \), \( F_1(\vec z, \kappa),
\ldots, F_m(\vec z, \kappa) \) are the \( \nu_0 \)-self-concordant and \( \nu_i
\)-self-concordant barrier functions for the epigraphs of the functions \(
f_0(\vec z), f_1(\vec z),\ldots,f_m(\vec z) \), respectively. Then,
according to~\cref{lemma:nemirovski,lemma:stability:of:barriers}, the resulting
self-concordant barrier
  \begin{equation}
    \widehat F(\vec z, \tau, \kappa)
    =
    F_0(\vec z, \tau)
    + \sum_{i = 1}^m F_j(\vec z, \kappa) - \log(C - \tau) - \log(-\kappa)
  \end{equation}
has a complexity parameter
  \begin{equation}
    \widehat \nu = \nu_0 + \sum_{i = 1}^m \nu_i + 2.
  \end{equation}
As specified in~\cite[Section 2]{nesterov1994interior}, each Newton step is a step of the form
\[
    x_{i+1} = x_i - \mu_i \cdot (\widehat F''(x_i))^{-1} \widehat F'(x_i),
\]
where \( \mu_i \) is a dampening parameter defined by the procedure, and \( x_i \in
\mathbb R^m \), where \( m \) is the dimension of the parameter space. The
computation of the dampening parameter, as well as the Newton step, involves
computing the inverse of the matrix of second derivatives. It
costs \( O(m^2 L) \) arithmetic operations to compose this matrix, which may
potentially vary if some exotic constructions are introduced, and
\( O(m^3) \) steps to compute its numerical inverse.

Finally, let us note that there is room for a variety of additional fine-tuned
optimisation techniques. For instance, the Newton iterations can use faster than
$O(n^3)$ matrix inversion algorithms, use $O(n^\omega)$ matrix multiplication
algorithms, or even use algorithms dedicated to sparse matrices.

\begin{Remark}
    The {\sf DCP} software, that we use in our prototype tuner implementation
uses a completely different principle, and expresses the problem in a standard
form using so-called \emph{exponential cones}. The approach with barriers
gives here more theoretical guarantees as it allows to cover a broader range
of constructions, but was not chosen for prototyping reasons.
\end{Remark}

\subsection{Barriers for context-free specifications}
In this section, we give rigorous complexity bounds for the case of context-free
grammars, see~\cref{section:cf:grammars}. Each of the constraints in the
auxiliary convex optimisation problem that we construct takes form
  \begin{equation}
    x_i \geq \log\left(
    \sum_{j} \exp(\vec a_{ij}^\top \vec x)
    \right) \quad \text{for } i = 1, \ldots, m
  \end{equation}
where \( \vec x \) is the vector of unknowns, and \( \vec a_{ij} \) are vectors of
coefficients. Such a convex constraint can be rewritten into an equivalent form
  \begin{equation}
    e^{x_i} \geq
    \sum_{j} \exp(\vec a_{ij}^\top \vec x)
    \quad \Leftrightarrow \quad
    \sum_{j} \exp( (\vec a_{ij} - \vec e_i)^\top \vec x) \leq 0
  \end{equation}
where \( \vec e_i \) is \( i \)th basis vector. Each of the functions \( \exp(
(\vec a_{ij} - \vec e_i)^\top \vec x) \) is convex, as 
\( (\vec a_{ij} - \vec e_i)^\top \vec x \) is linear and \( e^x \) is both
increasing and convex. Clearly, their sum is also convex. Using the
sum-substitution technique from~\cref{lemma:stability:of:barriers}, we conclude
that in order to provide a fast convex optimisation interior-point procedure, it
is sufficient to construct a self-concordant barrier for the epigraph of \(
e^x\). Note that \( \epi e^x = \{ (t, x) \in \mathbb R^2 \, \mid \, e^x
\leq t \}\) can be equivalently rewritten as
\( \{
    (t, x) \in \mathbb R^2 \, \mid \,
    t > 0, \, - \log t \leq -x
\}\). Furthermore, a linear variable change \( x \mapsto -x \), and variable swapping
\( (x, t) \mapsto (t, x) \) converts it into an epigraph of the convex function
\( f(x) = - \log x \):
  \begin{equation}
    \epi [-\log(x)] =
    \{
    (t, x) \in \mathbb R^2 \, \mid \,
    x > 0, \, - \log x \leq t
    \}.
  \end{equation}
Finally, we note that the condition from~\cref{lemma:nemirovski} is fulfilled
with \( \beta = \frac{2}{3}\). Therefore, a logarithmic barrier
  \begin{equation}
    F(t, x) = -\log(t + \log(x)) - \log x
  \end{equation}
is a \( 2 \)-self-concordant barrier for an epigraph of $-\log(x)$. This also
gives a \( 2 \)-self-concordant barrier for an epigraph of $e^x$. Consequently,
each unambiguous context-free specification can be converted into a standard
form for an interior-point method with total barrier complexity parameter \(
O(L) \) where \( L \) is the total number of terms in the specification.

\subsection{Cycle and positive set constructions}
\label{section:barriers:cycle:set}
Among the basic labelled operators, the \( \Cycle \) operator cannot be easily
fitted into the more restricted {\sf DCP} framework that we use in our prototype
implementation. The same issue concerns all of the restricted versions of
labelled operators including \( \Set_{>0} \), \( \Set_{>k} \), and \(
\Cycle_{>k} \). We provide further, more heuristic analysis of the logarithmic
barriers for these operators and hint as to why they can be efficiently
supported. We do not provide practical implementations of these modified
versions, however, in principle, they can be handled as well.

Let us start with two particular functions obtained as a log-exp transform of
the \( \Cycle \) and \( \Set_{>0} \) constructors,
see~\cref{section:implementation:details}. These two functions are
  \begin{equation}
    L(x) = \log \log\dfrac{1}{1 - e^x} \quad \text{and} \quad
    E(x) = \log(e^{e^x} - 1).
  \end{equation}
Note that the domain of \( L(x) \) is equal to \( \{x \in \mathbb R \, \mid \, x
< 0 \} \). The convexity of each of \( E(x) \) and \( L(x) \) is ensured by
noticing that they are log-sum-exp expressions with an infinite number of
positive summands:
  \begin{equation}
    L(x) = \log \sum_{n = 1}^\infty \dfrac{e^{kx}}{k} \quad \text{and} \quad
    E(x) = \log \sum_{n = 1}^\infty \dfrac{e^{kx}}{k!}.
  \end{equation}
Unfortunately, we cannot afford an infinite number of summands, as adding each
summand contributes to the final barrier complexity. Instead, we are going to
explicitly construct self-concordant barriers for the epigraphs of \( L(x) \) and
\( E(x) \) with the help of~\cref{lemma:nemirovski}.

\textbf{The case of \( f(x) = L(x) \)}.
Direct numerical evaluation shows that when \( 0 < x \lesssim 3 \beta \) it
holds $|f'''(x)| \leq \frac{3 \beta}{x} f''(x)$. Relative errors of
$\widehat{x}$ satisfying $|f'''(\widehat{x})| = \frac{3 \beta}{\widehat{x}}
f''(\widehat{x})$ for varying values of $\beta$ is listed in~\cref{fig:cycle:barrier}.
\begin{table}[ht!]
  \begin{tabular}{@{}c|c|c|c|c|c@{}}
    $\beta$ &  6 &  9 &  12 & 15 & 30 \\ \midrule 
    $|\widehat x - 3 \beta|$ & $10^{-1}$ & $10^{-2}$  & $10^{-3}$ & $10^{-4}$  & $10^{-11}$  \\
  \end{tabular}
  \caption{$|f'''(\widehat{x})| = \frac{3 \beta}{\widehat{x}} f''(\widehat{x})$
  for different $\beta$ and corresponding $\widehat{x}$.}
  \label{fig:cycle:barrier}
\end{table}

Let us notice the practical consequences of this fact. For most typical
combinatorial systems, tuning the expected size to $n$ gives the following
behaviour of the tuning parameter:
\begin{itemize}
    \item \( x_n = \rho \left(1 - O(n^{-1})\right) \) for algebraic and rational
      systems, and
    \item \( x_n = O(Poly(n)) \) for some specifications with singularity \(
      \rho \) at infinity, such as entire functions (e.g. labelled sets or
      similar constructs).
\end{itemize}
In fact, we are not aware of any natural example combinatorial system where
the tuning parameter approaches the singular value at an exponential speed \(
x_n = \rho(1 - e^{-n}) \).

Therefore if we choose, say, \( \beta = 10 \), then according
to~\cref{lemma:nemirovski}, the resulting logarithmic barrier will be
\(\beta^2\)-self-concordant for \( x \lesssim 30 \). In practical terms, \(
\beta = 10 \) covers expected values up to \( N \sim e^{30} \approx 10^{13} \).
More generally, for \( N \to \infty \), the complexity of the whole optimisation
scheme is proportional to a square root of the sum of all the barrier
complexities, which yields an additional \( \log N \) factor for the resulting
barrier complexity, and for the complexity of the final tuning algorithm.

\textbf{The case of \( f(x) = E(x) \)}.
If we try to apply~\cref{lemma:nemirovski} directly to \( E(x) \), we discover that
the conditions of the lemma cannot be satisfied. Therefore, we pre-process the
epigraph of \( E(x) \) by taking its inverse, similarly as we did for the
2-self-concordant barrier corresponding the exponent function \( e^x \).
Specifically, we transform \( \epi E \) into 
  \begin{equation}
    \epi E = \{
        (x, t) \in \mathbb R^2 \, \mid \,
        E(x) \leq t
    \}
    =
    \{
        (x, t) \in \mathbb R^2 \, \mid \,
        t > 0, \,
        x \leq \log \log (1 + e^t)
    \}.
  \end{equation}
It can be proven that \( f(x) = - \log \log(1 + e^x) \) is convex for \( x > 0
\). Moreover, we can numerically verify that the condition
of~\cref{lemma:nemirovski} is satisfied for all \( x \leq 10^{100} \). We
conjecture that it is valid for all \( x \in \mathbb R \), but, as we discussed
above, in practice we only need it to be valid for relatively small values of \(
x \).

\subsubsection{Restricted versions of cycles and sets.}
It is also possible to consider cycles and sets containing more than \( k \)
elements, so that we need to construct self-concordant barriers for the
functions
  \begin{equation}
    L_{> k}(x) = \log\left(
        \log\dfrac{1}{1 - e^x} - \sum_{n = 1}^k \dfrac{e^{kx}}{k}
    \right) \quad \text{and} \quad
    E_{> k}(x) = \log\left(
        e^{e^x} - \sum_{n = 0}^k \dfrac{e^{kx}}{k!}
    \right).
  \end{equation}
For the cycle construction, we again empirically discover that for \( x \lesssim
3 \beta\) the logarithmic barrier is self-concordant. We have numerically tested
this assertion for \( k \in \{ 0, 1, 2, 3\} \) and \( \beta \leq 10 \).

For the restricted construction \( \Set_{>k} \), we proceed in the same way as
we did for \( \Set_{>0} \) by constructing the epigraph of the negative inverse
function. In this case, however, the inverse function cannot be expressed
explicitly, but can only be computed numerically. All the numerical computations
can be done efficiently using the binary search for the inverse function, and
implicit derivative theorem for its derivatives. Again, this gives an
algorithmic definition of a self-concordant barrier for restricted sets without
the necessity to operate with infinite sums.

\subsubsection{Unlabelled restricted cycles and multi-sets.}
Note that the unlabelled version of the \( \Cycle \) operator expressed by
  \begin{equation}
    \Cycle(\mathcal A)(\vec z)
    =
    \sum_{k \geq 1} \dfrac{\varphi(k)}{k} \log \dfrac{1}{1 - A(\vec z^k)}
  \end{equation}
is in fact a weighted sum of already familiar labelled operators, for which we
have already provided efficient self-concordant barriers. Again, restricted
variations of the corresponding operators \( \MSet_{>0} \), \( \MSet_{>k} \) and
\( \Cycle_{>k} \) cannot be efficiently treated by {\sf DCP}, but the usage of
self-concordant barriers can provide some insight. For example, \( \MSet_{>0} \)
defined as
\begin{equation}
    \MSet\nolimits_{>0}(\mathcal A)(z) = \exp\Big(
        \sum_{k \geq 0} \dfrac{A(z^k)}{k}
    \Big) - 1,
\end{equation}
can be represented as a composition with \( e^x - 1 \), whose log-exp transform
is \( E(x) = \log(e^{e^x}-1) \).

Next, an unlabelled cycle construction with a lower bound on the number of
components \( \Cycle_{\geq m} \) can be expressed as
  \begin{equation}
    \Cycle\nolimits_{\geq m}(\mathcal A)(\vec z)
    =
    \sum_{k \geq 1} \dfrac{\varphi(k)}{k}
    \log_{\geq \lceil m/k \rceil} \dfrac{1}{1 - A(\vec z^k)}
    \quad \text{where} \quad
    \log_{\geq m} \dfrac{1}{1 - x}
    = \sum_{k \geq m} \dfrac{x^k}{k}.
  \end{equation}
Again, as in the case of unlabelled compositions, the sequence of values of the
generating functions \( \seq{A(\vec z^k)}_{k=1}^\infty \) tends to \( A(0) = 0\)
at the speed of a geometric progression. Note that this gives an efficient
composition scheme for unlabelled cycles as well, using
\begin{equation}
 L_{> k}(x)
= \log\Big(
    \log \frac{1}{1 - e^x} - \sum_{n = 1}^k \frac{e^{kx}}{k}
\Big).
\end{equation}
Finally, let us sketch how to express restricted unlabelled sets using the
self-concordant barriers for restricted labelled sets. We start by considering
two examples, \( \MSet_{\geq 2} \) and \( \MSet_{\geq 3} \).

For sets with at least two elements we have
  \begin{equation}
    \MSet\nolimits_{\geq 2}(\mathcal A)(\vec z) =
    \MSet(\mathcal A)(\vec z) - 1 - \mathcal A(z)
    =
    \exp\left(
        \sum_{k \geq 1} \dfrac{A(\vec z^k)}{k}
    \right) - 1 - A(\vec z).
  \end{equation}
If we set \( h_1 = A(\vec z) \) and \( h_2 = \sum_{k \geq 2} \frac{A(\vec
z^k)}{k} \), then the resulting value can be rewritten as
  \begin{equation}
    \exp(h_1 + h_2) - 1 - h_1 =
    \exp(h_1 + h_2) - 1 - (h_1 + h_2) + h_2.
  \end{equation}
The required composition scheme is obtained by composing with the log-exp
transform of the function \( e^x - x - 1\) which is \( E_{\geq 2}(x) :=
E_{>1}(x) = \log(e^{e^x} - e^x - 1) \).

For sets with at least three elements it holds
  \begin{equation}
    \MSet\nolimits_{\geq 3}(\mathcal A)(\vec z)
    =
    \exp\left(
        \sum_{k \geq 1} \dfrac{A(\vec z^k)}{k}
    \right) - 1 - A(\vec z) - \dfrac{A(\vec z)^2}{2}
    - \dfrac{A(\vec z^2)}{2}.
  \end{equation}
Again, by introducing auxiliary variables \( h_1 = A(\vec z) \), \( h_2 =
\frac{A(\vec z^2)}{2} \), \( h_3 = \sum_{k \geq 3} \frac{A(\vec z^k)}{k} \), and \(
H = h_1 + h_2 + h_3 \), we can rewrite the operator as
  \begin{equation}
    \exp(h_1 + h_2 + h_3) - 1 - h_1 - \frac{h_1^2}{2} - h_2
    =
    \exp(H) - 1 - H - \frac{H^2}{2} + h_3 + R
  \end{equation}
where \( R \) is a positive quantity obtained by subtracting \( \frac{h_1^2}{2} \)
from \( \frac{(h_1 + h_2 + h_3)^2}{2} \). Again, we obtain a composition scheme by
composing with \( E_{\geq 3}(x) := E_{>2}(x) \).

Finally, for sets with general restrictions the above scheme can be repeated for
multi-sets with forbidden cardinalities by subtracting \( \MSet\nolimits_{=m} \)
obtained by
  \begin{equation}
    \MSet\nolimits_{=m}(\mathcal A)(\vec z)
    =
    [u^m] \exp\left(
        \sum_{k \geq 1}
        \dfrac{A(\vec z^k)u^k}{k}
    \right).
  \end{equation}
Afterwards, the result can be represented by adding missing parts to \( \MSet_{>
m} \). We do not provide a detailed analysis of this case, as for large \( m \)
this procedure becomes less and less efficient due to the potential
super-polynomial increase in \( m \) of the number of the required terms.

\section{Optimal biased expectation and tuning precision}
\label{section:biased:expectation}

In the current section we address the choice of the optimal expected value for
the Boltzmann sampler with prescribed anticipated rejection domain. We also
investigate the precision of the control vector $\boldsymbol{x}$ required to
guarantee a linear time, approximate-size anticipated rejection sampling scheme.

In what follows, we analyse the rejection cost independently for each parameter.
That is to say, we evaluate the cost of the anticipated rejection until the
value of the
investigated parameter \( N \) falls into its prescribed tolerance window
\( [n_1, n_2] \). The sum of these costs taken over all the parameters
is an upper bound for costs of a global rejection scheme
(which, of course, could be better because of some overlapping rejections).

We distinguish two types of behaviors. The first concerns parameters admitting a
singular behavior of type ${(1-\frac{z}{\rho})}^{-\alpha}$ which corresponds,
for instance, to the size parameter of \emph{peak} and \emph{flat} distributions
described in \cite{BodGenRo2015}. In this case, we will prove that for each
$\alpha$ the required precision is of order $O\left(\frac{1}{N}\right)$ where $N$
is the approximate size of the expected outcome.

The second case concerns so-called \emph{bumpy} distributions, appearing in
numerous situations, such as parameter analysis in strongly connected algebraic
or rational grammars. Since the distribution is concentrated around its mean
value, the needed precision is related to the standard deviation and is,
usually, of order $O(\frac{1}{\sqrt{N}})$. In contrast, the size of the target
window is not linear, as for the non-concentrated case, but instead is of order
$\sqrt{N}$, or of a different order \( n^\gamma \).

In the first case, we show that the standard target tuning expectation $\mathbb{E}(N) =
n := \frac{n_1 + n_2}{2}$ involving the admissible window $[n_1,n_2]$ does
\emph{not} minimise the overall cost of rejections, if anticipated rejection is
taken into account.  Instead, $\mathbb{E}(N) = \beta n$, for some suitable
parameter $\beta(n_1,n_2)$ should be used, cf~\cref{fig:optimal:bias}.  In what
follows, we describe how to find the optimal value of  \( \beta \)
and how it improves the rejection cost, compared to the usual $n := \frac{n_1
+ n_2}{2}$ parameter.

\begin{figure}[hbt!]
\centering
\begin{tikzpicture}[>=stealth',thick, node distance=1.0cm,inner sep=0pt]
  \draw[very thick] (-3,0) -- (7,0);
  \draw
      node at (-2,0) {[}
      node at (-2,-0.5) {\( n_1 \)}
      ;
  \draw
      node at (4,0) {]}
      node at (4,-0.5) {\( n_2 \)}
      ;
  \draw
      node at (1,0) {|}
      node at (0.9,1.4) {\scalebox{0.9}{\( \mathbb E_x(N) = n \)}}
      node at (1,-0.45) {default}
      node (f) at (1,1.2) {}
      node (a) at (1,0.2) {}
      ;
  \fromto{f}{a}

  \draw
      node at (3,0) {|}
      node (g) at (3,1.2) {}
      node (b) at (3,0.2) {}
      node at (2.9,-0.45) {biased}
      node at (3.2,1.4) {\scalebox{0.9}{\( \mathbb E_x(N) = \beta n \)}}
      ;
  \fromto{g}{b}

  \draw
      node (e) at (7.2,1.2) {}
      node (c) at (7.2,0.2) {}
      node at (7.2,-0.45) {singular}
      node at (7.2,1.4) {\( \mathbb E_x(N) = \infty \)}
  ;
  \fromto{e}{c}
\end{tikzpicture}
\caption{\label{fig:optimal:bias}Optimal biased expectation}
\end{figure}
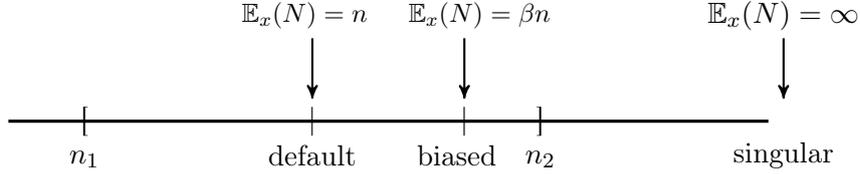

\subsection{Non-concentrated case}
For simplicity, we assume that the analysis concerns the size parameter $z$.
Let $C_{\vec u}(z)$ be the generating
function of a class $\mathcal{C}$ for which the parameters $\vec u$ are fixed
(i.e.~introduce some additional weights). Assume that $C_{\vec u}(z)$ is
$\Delta$\nobreakdash-analytic and there exist analytic functions $\alpha(z)$,
$\beta(z)$, and a corresponding singularity $\rho$ (all depending on $\vec u$)
such that as \( z \to \rho \) it holds
\begin{equation}\label{eq:algebraic:asymptotic:expansion}
    C(z) \sim a(z) - b(z) {\left(1 - \dfrac{z}{\rho}\right)}^{-\alpha}.
\end{equation}

Let $C^{< n_1}(z)$, $C^{> n_2}(z)$ and $C^{[n_1, n_2]}(z)$ denote the generating
functions corresponding to the subclasses of objects of size strictly smaller
than $n_1$, strictly greater than $n_2$, and in between (inclusively) $n_1$ and
$n_2$, respectively. Let $T_n$ stand for the cumulative size of the generated
and rejected objects produced by a sampler calibrated with parameter $x$, and an
admissible size window $[(1-\varepsilon)n, (1+\varepsilon)n]$.  From
\cite[Theorem 7.3]{DuFlLoSc}, the corresponding probability generating function
\begin{equation}
    F(u,x) = \sum_{k} \mathbb{P} (T_n = k) u^k
\end{equation}
satisfies
\begin{equation}
    F(u,x) = \left(1 - \frac{1}{C^{}(x)}(C^{<(1-\varepsilon)n)}(ux) +
    C^{>(1+\varepsilon)n)}(x) u^{\lceil(1+\varepsilon)n\rceil})
    \right)^{-1} \frac{C^{[(1-\varepsilon)n,(1+\varepsilon)n]}(x)}{C^{}(x)}
\end{equation}
and so
\begin{equation}\label{eq:precision:generated:expectation}
    \mathbb{E}(T_n) = \frac{x \frac{d}{d x}C^{<(1-\varepsilon)n}(x)
        + \lceil (1+\varepsilon)n \rceil
    C^{>(1+\varepsilon)n}(x)}{C^{[(1-\varepsilon)n,(1+\varepsilon)n]}(x)}.
\end{equation}

\begin{Remark}
Notice that even if the studied parameter does not denote the object size, the
    expectation $\mathbb{E}(T_n)$ stays essentially similar
    to~\eqref{eq:precision:generated:expectation}. Involved generating functions
    correspond then to the subclasses of objects whose \emph{size contribution
    of the studied parameter} is smaller, greater, or in between the respective
    window parameters. It ensues that the following analysis for the size
    parameter can be readily translated to others parameters.
\end{Remark}

\begin{proposition}
Let us consider an approximate-size, anticipated rejection Boltzmann sampler
$\Gamma(\mathcal{C})$ for a class $\mathcal{C}$ of the above kind. Assume that
$\Gamma(\mathcal{C})$ is calibrated using $x_n = \rho
\left(1-\frac{\delta}{n}\right)$ for some parameter $\delta$. Let $T_n$ denote
the cumulative size of objects rejected by $\Gamma(\mathcal{C})$. Then, as $n$
tends to infinity
\begin{align}\label{eq:precision:kappa}
    \begin{split}
    \mathbb{E}(T_n) &\sim n \kappa_c(\varepsilon,\alpha,\delta)
    \quad \text{where}\\
    \kappa_c(\varepsilon,\alpha,\delta)
        &=\frac{\displaystyle\int\limits_{w=0}^{1-\varepsilon}w^{\alpha}
        e^{-\delta w} \dd{w} +  \displaystyle\int\limits_{w=1+\varepsilon}^{\infty}
        (1+\varepsilon) w^{\alpha-1} e^{-\delta w} \dd{w} }
        {\displaystyle\int\limits_{w=1-\varepsilon}^{1+\varepsilon} w^{\alpha-1}
        e^{-\delta w} \dd{w}}.
    \end{split}
\end{align}
\end{proposition}

\begin{proof}
This theorem is very similar to~\cite[Theorem 3.2]{BodGenRo2015}.
It has an analogous proof using Euler--Maclaurin summation, with the difference that
we do not anymore suppose that \( \alpha = \delta \).
Let us emphasise again that keeping these two parameters distinct is the key to obtaining faster samplers with lower anticipated rejection cost.
\end{proof}

Note that $\kappa_c(\varepsilon,\alpha,\delta)$ is independent of $n$. Moreover,
$\kappa_c(\varepsilon,\alpha,\delta)$ is bounded for each $\delta > 0$.
Consequently, $\Gamma(\mathcal{C})$ calibrated with parameter $x_n$ of precision
of order $O\left(\frac{1}{n}\right)$ has expected linear time complexity.

\begin{Remark}
The value of $\kappa_c(\varepsilon,\alpha,\delta)$ is also independent of the
    weight $\vec u$, depending only on the type of the singularity $\alpha$ and
    the prescribed tolerance $\varepsilon$. Note that it provides strong
    stability guarantees regarding the sampler's performance under the change of
    involved parameters; most notably, in the vicinity of the values of $\vec u$
    where the type of the singularity changes discontinuously,
    cf.~\cite{BodPonty,banderier2012diversity}. Even in this particular cases,
    our result guarantees the linearity of the anticipated rejection sampling scheme.
\end{Remark}

With an explicit formula~\eqref{eq:precision:kappa} for
$\kappa_c(\varepsilon,\alpha,\delta)$ it is possible to optimise its value for
specific values of $\epsilon$ and $\alpha$. In other words, improve the
multiplicative constant in the expected number $T_n$.~\Cref{tab:precision:values} provides some exemplary values for
$\delta_{\text{min}}$ and the corresponding
$\kappa_c(\varepsilon,\alpha,\delta_{\text{min}})$.
\begin{table}[ht!]
\begin{center}

\begin{tabular}{|l||c|c|c|c|}
\hline
 & $\varepsilon=0.2$ & $\varepsilon=0.1$ & $\varepsilon=0.05$ & $\varepsilon=0.01$ \\
\hline
\hline
$-\alpha=-2$ & (2.77, 2.37) & (2.66, 5.58) & (2.61, 11.95)& (2.58, 62.78) \\
\hline
    $-\alpha=-3/2$ & (2.25, 2.66) & (2.16, 6.17) & (2.12, 13.12)& (2.09, 68.58) \\
\hline
$-\alpha=-1$ & (1.74, 3.06) & (1.66, 6.97) & (1.62, 14.72)& (1.60, 76.51) \\
\hline
    $-\alpha=-1/2$ &(1.23, 3.66) & (1.16, 8.18) & (1.13, 17.18) & (1.11, 88.38)\\
\hline
    $-\alpha=1/2$ & (0.25, 7.37) & (0.22, 15.68) & (0.21, 32.10)& (0.20, 162.91) \\
\hline
\end{tabular}
\end{center}
    \caption{Pairs of values $(\delta_{\text{min}},
    \kappa_c(\varepsilon,\alpha,\delta_{\text{min}}))$ for respective parameters
    $\varepsilon$ and $\alpha$. Numerical values are rounded up to the second
    decimal point.}\label{tab:precision:values}
\end{table}

\begin{Remark}
    The careful reader might be surprised by the fact that $\delta_{\text{min}}$
    is not equal to $\alpha$, as it is suggested in the seminal paper
    \cite{DuFlLoSc}. The reason behind this is the fact that anticipated
    rejection creates a small bias in the distribution, initially not taken into
    account. For instance, for $\alpha=1$ and tolerance $\varepsilon=0.1$, the
    best choice for $\beta$ is not $1$, but $1.6572067$. Notably, this
    decreases the rejection complexity from $8.05n$ to $6.97n$.

    In consequence of the introduced bias, it is no longer optimal to find $x$ by
    solving $\mathbb{E}_x(N) = n$. Instead, we have to make a small correction
    accounting for the effect of anticipated rejection. Luckily, this is easy to
    provide. Recall that, asymptotically, $\mathbb{E}_x(N) = n$ is attained for
    $x_n=\rho(1-\frac{\alpha}{n})$. After computing the expectation, it follows
    that in order to obtain $x_n = \rho(1-\frac{\delta_{\text{min}}}{n})$ we
    have to solve the corrected $\mathbb{E}_x(N) =
    \frac{\alpha}{\delta_{\text{min}}} n$, instead.
\end{Remark}

\subsection{Concentrated cases}
In following part we consider cases where the investigated size parameter tends
asymptotically to a Gaussian law, provided that the target expectation is fixed
and large.  Let \( C(z) \) be, as usual, the generating function corresponding
to the class \( \mathcal C \), and \( N \) be the random variable representing
the size of an object generated according to the associated Boltzmann
distribution.  As previously, we want to evaluate the rejection cost and find an
appropriate bias parameter which minimises its value.

Before we begin, let us consider a few examples of concentrated distributions
analysed using the following generalised quasi-powers theorem.
\begin{lemma}[Generalised quasi-powers, {\cite[Theorem IX.13]{flajolet09}}]\label{lemma:generalised:qp}
    Assume that, for \( u \) in a fixed complex neighbourhood \( \Omega \) of \(
    1 \), the probability generating functions \( p_n(u) \) of non-negative
    discrete random variables \( X_n \) admit representations of the form
    \begin{equation}
        p_n(u) = \exp(h_n(u)) \left(1 + o(1)\right),
    \end{equation}
    uniformly with respect to \( u \), where each \( h_n(u) \) is analytic in
    \( \Omega \). Assume also the conditions
    \begin{equation}
        h_n'(1) + h_n''(1) \to \infty
        \quad \text{and} \quad
        \dfrac
        {h'''(u)}
        {( h_n'(1) + h_n''(1) )^{3/2}} \to 0
    \end{equation}
    as $n \to \infty$, uniformly for $u \in \Omega$.

    Then, the random variable
      \begin{equation}
        X_n^\star
        =
        \dfrac{X_n - \mathbb E X_n}
        { \sqrt{\operatorname{Var} X_n} }
        =
        \dfrac{X_n - h'_n(1)}
        {(h_n'(1) + h_n''(1))^{1/2}}
      \end{equation}
    converges in distribution to \( \mathcal N(0,1) \).
\end{lemma}

In what follows, we apply the generalised quasi-powers theorem to a series of
probability generating functions capturing the outcome size
distributions of example Boltzmann samplers.
Therefore, the discrete random variables of interest \( X_n \) denote
these outcome sizes, where \( n \) stands for the target expected size,
while \( n \to \infty \).
We use the standard formulas for
the mean and variance corresponding to the size $N$ of objects sampled according
to the Boltzmann distribution parametrised with $x$:
\begin{equation}
    \mathbb E_{x} N = \mu(x)=\dfrac{xC'(x)}{C(x)}
    \qquad
    \mathrm{Var}_x N = {\sigma(x)}^2 =x \mu(x)'.
\end{equation}
For convenience, we also use $x_n$ to denote the solution of the tuning equation
$\mu(x_n) = n$.

\subsubsection{Bicoloured sets}
Consider the combinatorial class \( \mathcal C \) consisting of finite sets with
atoms of two different types (colours). The corresponding generating function
satisfies \( C(x) = e^{2x} \).
Accordingly, the mean value $\mu(x)$ and standard deviation $\sigma(x)$ satisfy
\begin{equation}
    \mu(x) = 2x \quad \text{and} \quad \sigma(x) = \sqrt{2 x},
\end{equation}
whereas \( x_n = n/2 \).

The probability generating function $p_n(u)$ capturing the size
distribution calibrated with parameter $x_n$ satisfies therefore
\begin{equation}
    p_n(u) = \frac{C(x_n u)}{C(x_n)} = e^{n (u - 1)}.
\end{equation}
A direct calculation reveals that all the conditions
of~\cref{lemma:generalised:qp} hold, and the size distribution tends to
a Gaussian law as $x_n \to \infty$. In particular, we have
\begin{equation}
    \mathbb{E}_{x_n} N \sim n \quad \text{and} \quad \operatorname{Var}_{x_n} N
    \sim n.
\end{equation}

\subsubsection{Involutions}
Consider the combinatorial class of involutions, \ie
$n$-element permutations \( \pi \) satisfying \( \pi \circ \pi = {\sf id} \).
The respective generating function satisfies $C(x)=\exp(x+\frac{x^2}{2})$.
Consequently, the mean value $\mu(x)$,
variance ${\sigma(x)}^2$, and tuning parameter $x_n$ satisfy
\begin{equation}
    \mu(x) = x(1 + x),
    \quad {\sigma(x)}^2 = x (1 + 2 x),
    \quad \text{and} \quad x_n=\dfrac{\sqrt{1+4n}-1}{2}.
\end{equation}
And so
\begin{equation}
    p_n(u) = \exp \left(-\frac{1}{4} (u-1) \left(\left(\sqrt{1+4 n}-1\right) (u-1)-2 n
   (u+1)\right)\right).
\end{equation}
Again, the premises of~\cref{lemma:generalised:qp} can be easily verified. In
the end, we obtain a limit Gaussian distribution where
\begin{equation}
    \mathbb{E}_{x_n} N \sim n \quad \text{and} \quad \operatorname{Var}_{x_n} N
    \sim \frac{1}{2} \left(1 + 4 n-\sqrt{1+4 n}\right).
\end{equation}

\subsubsection{Set partitions}
Consider the combinatorial class of set partitions for which we have
\( C(x) = e^{e^x-1} \).
A direct computation provides the identities
\begin{equation}
    \mu(x) = x e^x, \quad
    {\sigma(x)}^2 = x e^x (1 + x), \quad
    \text{and} \quad x_n = W(n)
\end{equation}
where \( W(n) \) is the Lambert function defined as the positive solution of \(
W(n) e^{W(n)} = n \).

In this case, $p_n(u) = \exp(h_n(u))\left(1 + o(1)\right)$ where $h_n(u) = e^{u
W(n)}-e^{W(n)}$. We can easily check that
\begin{equation}
    h'_n(1) = n, \quad h''_n(1) = e^{W(n)} {W(n)}^2, \quad \text{and} \quad
    h'''_n(u) = W(n)^3 e^{u W(n)}.
\end{equation}
Since $W(n) = \ln n - \ln \ln n + o(1)$ we note that $h'_n(1) + h''_n(1) \to
\infty$.  Moreover
\begin{equation}
    \frac{h'''_n(u)}{{\left(h'_n(1) + h''_n(1)\right)}^{3/2}} = \frac{W(n)^3
    e^{u W(n)}}{(n (W(n)+1))^{3/2}} \sim
    {(\ln n)}^{3/2} n^{u - 3/2}
\end{equation}
which for $u$ fixed near one tends to $0$ as $n \to \infty$.

Hence, by~\cref{lemma:generalised:qp} the distribution tends, again, to a Gaussian law.
In the limit we obtain
\begin{equation}
    \mathbb{E}_{x_n}(N) \sim n \quad \text{and} \quad
    \operatorname{Var}_{x_n}(N) \sim n \left( 1 + W(n)\right).
\end{equation}

\subsubsection{Fragmented permutations}
Next, consider the class of fragmented permutations, \ie sets of non-empty
labelled sequences, see~\cite[Example VIII.7, p. 562]{flajolet09}.  The
corresponding generating function satisfies
\(
    C(x) = \exp \left(
        \frac{x}{1 - x}
    \right)
\).
Consequently, we find that
\begin{equation}
\mu(x) = \frac{x}{{\left(1-x\right)}^2},
    \quad {\sigma(x)}^2 = x \left(\frac{2 x}{(1-x)^3}+\frac{1}{(1-x)^2}\right) \quad
    \text{and} \quad
    x_n = \frac{1 + 2 n-\sqrt{1+4 n}}{2 n}.
\end{equation}

Note that as $n \to \infty$, the tuning parameter $x_n \to 1$.  Again, we verify
that, as $x_n \to \infty$, the function \( h_n(u) = \frac{u x_n}{1-u
x_n}-\frac{x_n}{1-x_n} \) satisfies both
  \begin{equation}
    h'(x_n) + h''(x_n) \to \infty \quad \text{and} \quad
    \dfrac
    { h'''(x_n u) }
    { ( h_n'(x_n) + h_n''(x_n) )^{3/2} } \to 0.
  \end{equation}
However, now $h_n(u)$ is not analytic at $u = 1$.  We cannot therefore
apply~\cref{lemma:generalised:qp}. Nonetheless, it is still possible to prove
that the limiting distribution is Gaussian, using the explicit formula for the
number of fragmented permutations from~\cite{flajolet09}.

\bigskip
From the above examples we see that there is a variety of different behaviours
regarding the variance of the limit Gaussian distribution. In what follows we
will show how the mean value is connected to the anticipated rejection
cost~\eqref{eq:precision:generated:expectation} in the Gaussian case.

\begin{theorem}[Tuning precision for concentrated distributions]
Let the size parameter asymptotically follow a Gaussian law with
expectation $\mu(x)$ and standard deviation $\sigma(x)$, and let \( X(n) \) be
the inverse function of \( \mu(x) \), i.e. \( \mu(X(n)) = n \).
Denote by $x_{n,\delta}$ the biased tuning value
  \begin{equation}
    x_{n, \delta} := x_n+\delta\sigma(x_n)\dfrac{d X(n)}{dn}.
  \end{equation}
Then the anticipated rejection cost \( T_n \) with the target tolerance window
$[n-\varepsilon\sigma(x_n),n+\varepsilon\sigma(x_n)]$ satisfies, when $n$ tends
to infinity:
  \begin{equation}
    \mathbb{E}_{x_{n,\delta}}(T_n) \sim n \kappa(\varepsilon,\delta),
    \quad
    \kappa(\varepsilon,\delta) =
    \dfrac{
        1 + \Phi(-\varepsilon - \delta) - \Phi(\varepsilon - \delta)
    }{
        \Phi(\varepsilon - \delta)
        - \Phi(-\varepsilon - \delta)
    }
  \end{equation}
where
\(
\Phi(x) := \dfrac{1}{\sqrt{2 \pi}} \displaystyle\int_{-\infty}^x
e^{-w^2/2} dw
\)
is the Gaussian distribution function. The minimal cost is achieved when
\( \delta = \delta_{\min} = 0 \).
\end{theorem}

\begin{proof}
Using the Euler--MacLaurin estimate similarly to~\cite{BodGenRo2015}, we obtain the
following estimates for \( C^{[n_1, n_2]}(x), C^{>n_2}(x) \) and
\( \frac{d}{dx} C^{<n_1}(x) \), when \( x = x_n \) as \( n \to \infty \):
\begin{align}
  \begin{split}
    \dfrac{C^{[n_1, n_2]}(x)}{C(x)} &= \sum_{n = n_1}^{n_2} \dfrac{[z^n]C(z)
      x^n}{C(x)} \sim \dfrac{1}{\sqrt{2 \pi} \sigma(x)} \int_{n_1}^{n_2} e^{
      -\frac{1}{2} \left( \frac{n - \mu(x)}{\sigma(x)} \right)^2 } dn
    \\
    & = \Phi\left( \dfrac{n_2 - \mu(x)}{\sigma(x)} \right) - \Phi\left(
      \dfrac{n_1 - \mu(x)}{\sigma(x)} \right)
    , \\
    \dfrac{C^{> n_2}(x)}{C(x)} & \sim 1 - \Phi\left( \dfrac{n_2 -
        \mu(x)}{\sigma(x)} \right)
    , \text{ and} \\
    \dfrac{x \frac{d}{dx} C^{< n_1}(x)}{C(x)} & = \sum_{n = 0}^{n_1} n
    \dfrac{[z^n]C(z) x^n}{C(x)} \sim \dfrac{1}{\sqrt{2 \pi} \sigma(x)}
    \int_0^{n_1} \left[ (n - \mu(x)) + \mu(x) \right] e^{ -\frac{1}{2} \left(
        \frac{n - \mu(x)}{\sigma(x)} \right)^2 } dn
    \\
    & \sim \dfrac{2 \sigma(x)}{\sqrt{2 \pi}} \left[ 1 - e^{ -\frac{1}{2} \left(
          \frac{n_1 - \mu(x)}{\sigma(x)} \right)^2 } \right] + \mu(x) \Phi
    \left( \dfrac{n_1 - \mu(x)}{\sigma(x)} \right).
  \end{split}
\end{align}
Since \( \mu(x) \gg \sigma(x) \) for \( x = x_n \) as \( n \to \infty \),
the second summand dominates the first summand in the last expression.

Combining these quantities, and by substituting into
\(
    \mathbb{E}(T_n)
    = \frac
    {x C'^{<n_1}(x)+ n_2 C^{>n_2}(x)}
    { C^{[n_1,n_2]}(x)}
\),
we get the estimated cost of anticipated rejection:
  \begin{equation}
    \mathbb E_x (T_n)
    =
    \dfrac
    {
        \mu(x) \Phi
        \left(
            \dfrac{n_1 - \mu(x)}{\sigma(x)}
        \right)
        +
        n_2 \left[
            1
            -
            \Phi\left(
                \dfrac{n_2 - \mu(x)}{\sigma(x)}
            \right)
        \right]
    }
    {
        \Phi\left(
            \dfrac{n_2 - \mu(x)}{\sigma(x)}
        \right)
        -
        \Phi\left(
            \dfrac{n_1 - \mu(x)}{\sigma(x)}
        \right)
    }
    .
  \end{equation}
Let us choose \( x = x_{n, \delta} \) in such a way that
\(
    \mu(x_{n, \delta}) = n + \delta \sigma(x_n)
\).
Using the first two terms of the Taylor expansion of \( \mu(x_{n,\delta}) \)
around \( x_n \), we can show that this value is asymptotically
\(
    x_{n, \delta} \sim x_n+\delta\sigma(x_n)\frac{d X(n)}{dn}
\).
This finally allows to obtain \( \kappa(\varepsilon, \delta)
:= \lim_{n \to \infty} \dfrac{\mathbb E T_n}{n}
\):
  \begin{equation}
    \kappa(\varepsilon, \delta) =
    \dfrac{
        1 + \Phi(-\varepsilon - \delta) - \Phi(\varepsilon - \delta)
    }{
        \Phi(\varepsilon - \delta)
        - \Phi(-\varepsilon - \delta)
    }
    .
  \end{equation}
This expression achieves its minimum when \( \Phi(\varepsilon - \delta) -
\Phi(-\varepsilon - \delta) \) achieves its maximum. Its derivative has only
one root \( \delta = 0 \) which corresponds to the global maximum. Therefore,
\( \delta_{\min} = 0 \).
\end{proof}

\begin{Remark}
    In the case when the standard deviation \( \sigma(x_n) \) is negligible
    compared to the mean value \( \mu(x_n) \) as \( n \to \infty \), the
    rejection cost can be shown to be asymptotically equal
      \begin{equation}
        \mathbb E_{x_{n,\delta}} T_n \sim n \kappa(\varepsilon, \delta)
        = n \dfrac{ \int_{(-\infty, n_1) \cup (n_2, +\infty)} p(x) dx}
        { \int_{(n_1, n_2)} p(x) dx},
      \end{equation}
    where \( p(x) \) is the probability density function of the limiting
    distribution. In the case when the limiting distribution is not symmetric,
    and the probability density function is known, the corresponding bias can be
    computed as \( \delta_{\min} = \arg\min \kappa(\varepsilon, \delta) \).
\end{Remark}

\section{Paganini: a multi-parametric tuner prototype}\label{sec:paganini}

To illustrate the effectiveness of our tuning procedure, we developed {\sf
Paganini\footnote{see \url{https://github.com/maciej-bendkowski/paganini}
and
\url{http://paganini.readthedocs.io}.}} ---
a lightweight Python library implementing the
\emph{tuning-as-convex-optimisation} idea. Our software relies on {\sf cvxpy}, a
Python-embedded modelling language for Disciplined Convex Programming ({\sf
DCP})~\cite{grant2006disciplined}. With its help, {\sf Paganini} is able to
automatically compose, and solve adequate optimisation problems so to compute
the parameter vector corresponding to the user-defined expectations.

\subsection{Implementation details}
\label{section:implementation:details}
Due to the imposed restrictions of Disciplined Convex Programming, {\sf
Paganini} supports a strict, though substantial subset of admissible
constructions. In both the labelled and unlabelled case, {\sf Paganini} provides
the basic empty class $\varepsilon$, disjoint sum $+$, and Cartesian product
$\times$ operations. More involved constructions are briefly discussed below.

\subsubsection{Sequence operator}
Consider $\Seq(\mathcal{A})$ for some class $\mathcal{A}$ (either labelled or
unlabelled). By definition, ${\Seq{(\mathcal{A})} = \varepsilon + \mathcal{A}
\times \Seq{(\mathcal{A})}}$ and so we can treat $\Seq{(\mathcal{A})}$ as a new,
auxiliary variable with a corresponding definition of the above shape. The
log-exp transform of $\Seq{(\mathcal{A})}$ takes then the form of an elementary
{\sf DCP} log-sum-exp function:
\begin{equation}
    \Seq{(\mathcal{A})}(\boldsymbol{z}) = \dfrac{1}{1 -
    \mathcal{A}(\boldsymbol{z})} \qquad \xrightarrow{\text{log-exp}}
    \qquad \sigma \geq \log \left(1 + e^{\alpha +
    \sigma}\right)
\end{equation}
where $e^{\alpha} = A(\boldsymbol{z})$ and $e^{\sigma} =
\Seq{(\mathcal{A})}(\boldsymbol{z})$.

Likewise, since
\begin{align}
    \begin{split}
    {\Seq}_{= k}{(\mathcal{A})}(\boldsymbol{z}) = {A(\boldsymbol{z})}^k,
        &\qquad{\Seq}_{\geq k}{(\mathcal{A})}(\boldsymbol{z}) =
        \dfrac{{A(\boldsymbol{z})}^k}{1 - A(\boldsymbol{z})},\quad\text{and}\\
    {\Seq}_{\leq
    k}{(\mathcal{A})}(\boldsymbol{z}) &= \sum_{i = 0}^{k} {A(\boldsymbol{z})}^k
    \end{split}
\end{align}
it is readily possible to translate $\Seq$ with its restricted variants into
valid {\sf DCP} constraints.

\subsubsection{Pólya structures}
Consider $\MSet{(\mathcal{A})}$ and its log-exp variant:
\begin{equation}
    \MSet{(\mathcal{A})}(\boldsymbol{z}) = \exp \left(\sum_{k \geq 1}
    \dfrac{A(\boldsymbol{z}^k)}{k}\right) \qquad \xrightarrow{\text{log-exp}}
    \qquad \mu \geq \sum_{k \geq 1} \dfrac{e^{\alpha_k}}{k}
\end{equation}
where $e^{\alpha_k} = {A(\boldsymbol{z}^k)}$ and $e^{\mu} =
\MSet{(\mathcal{A})}(\boldsymbol{z})$.

The right-hand side of the $\MSet$ log-exp transform is an infinite sum of exponents
with positive weights. For practical purposes, we can therefore truncate the
series at a finite threshold and notice that it conforms with {\sf DCP} rules.
What remains is to construct constraints for the respective \emph{diagonals}
$A(\boldsymbol{z}^k)$ based on the definition of $A(\boldsymbol{z})$.

It is also possible to handle the restricted $\MSet_{= k}$.
Following~\cite[Section 2.4]{flajolet2007boltzmann} we notice that
${\MSet}_{= k}{(\mathcal{A})}(\boldsymbol{z})$ can be expressed as
\begin{equation}\label{eq:paganini:mset:k}
    {\MSet}_{= k}{(\mathcal{A})}(\boldsymbol{z}) = \sum_{P \in \mathcal{P}_k} \prod_{i = 1}^k
    \dfrac{A(\boldsymbol{z}^i)}{i^{n_i} n_i!}
\end{equation}
where $\mathcal{P}_k$ consists of so-called \emph{partition sequences} of size $k$,
i.e.~sequences $\seq{n_i}$ of natural numbers satisfying the condition $\sum_{i
= 1}^k i n_i = k$. In this form~\eqref{eq:paganini:mset:k} unfolds to a
polynomial in $A(\boldsymbol{z}), A(\boldsymbol{z}^2),\ldots,
A(\boldsymbol{z}^k)$ with positive coefficients. Consequently, ${\MSet}_{= k}$
can be converted to a {\sf DCP} constraint just like a regular algebraic equation.
Following the same idea, we can handle ${\MSet}_{\leq k}$ as
\begin{equation}
    {\MSet}_{\leq k}{(\mathcal{A})}(\boldsymbol{z}) = 1 + \sum_{i = 1}^k {\MSet}_{=
    i}{(\mathcal{A})}(\boldsymbol{z}).
\end{equation}

The ${\MSet}_{\geq k}$ variant is much more involved. Since the difference of
convex functions is not necessarily a convex function itself, we cannot directly
translate the defining
\begin{equation}
    {\MSet}_{\geq k}(\mathcal{A}) = {\MSet}(\mathcal{A}) -
{\MSet}_{< k}(\mathcal{A})
\end{equation}
into a valid {\sf DCP} constraint. Let us recall,
however, that for $k = 1$
in the case when \( \mathcal A = z_1 + z_2 + \cdots + z_d \)
it is possible to rewrite ${\MSet}_{\geq k}$ so to
avoid subtraction altogether and compose a corresponding {\sf DCP} constraint
(cf.~\Cref{subsection:weighted:partitions} and more generally,~\Cref{section:symmetric:polynomials}).

For general $k$, a more heuristic approach using \emph{Disciplined
Convex-Concave Programming} ({\sf DCCP}) might be preferred, see~\cite{dccp}.
Consider the following \emph{exp} transform:
\begin{equation}
    {\MSet}_{\geq 1}(\mathcal{A})(\boldsymbol{z}) \qquad \xrightarrow{\text{exp}} \qquad
    e^{\mu} + 1 \geq \exp \left( \sum_{k \geq 1} \dfrac{e^{\alpha_k}}{k} \right)
\end{equation}
where $e^{\alpha_k} = {A(\boldsymbol{z}^k)}$ and $e^{\mu} =
{\MSet{(\mathcal{A})}}_{\geq 1}(\boldsymbol{z})$.

Here we have two convex expressions on both sides of the inequality. Although
such constraints do not conform with {\sf DCP} rules, they are allowed in the
{\sf DCCP} framework, and can be therefore included in the problem statement.

Now, let us focus the cycle construction $\Cycle{(\mathcal{A})}$. Note that
\begin{equation}\label{eq:paganini:cycle}
    \Cycle{(\mathcal{A})}(\boldsymbol{z}) = \sum_{k \geq 1} \dfrac{\varphi(k)}{k} \log
    \dfrac{1}{1- A(\boldsymbol{z}^k)}\qquad \xrightarrow{\text{log-exp}} \qquad
        \gamma \geq \log \left( \sum_{k \geq 1} \dfrac{\varphi(k)}{k} \log
        \dfrac{1}{1 - e^{\alpha_k}}  \right)
\end{equation}
where $e^{\alpha_k} = {A(\boldsymbol{z}^k)}$, $e^{\gamma} =
\Cycle{(\mathcal{A})}(\boldsymbol{z})$, and $\varphi(k)$ is the Euler
totient function.

Unfortunately, such a constraint does not meet the requirements of {\sf DCP},
even though its right-hand side is convex. On the other hand, the restricted ${\Cycle}_{=
k}{(\mathcal{A})}$ satisfies
\begin{equation}
    {\Cycle}_{= k}{(\mathcal{A})}{(\boldsymbol{z})} = \dfrac{1}{k} \sum_{i \mid k}
    \varphi(i) {A(\boldsymbol{z}^i)}^{\frac{k}{i}}
\end{equation}
and so is it possible to express its log-exp transform as a standard {\sf DCP}
log-sum-exp function.  In order to emulate the unrestricted
${\Cycle}{(\mathcal{A})}$ operator, one can either use {\sf DCCP},
reformulating~\eqref{eq:paganini:cycle} as a {\sf DCCP} constraint, or use the
relation
\begin{equation}
    {\Cycle}{(\mathcal{A})}{(\boldsymbol{z})} = \sum_{k \geq 0} {\Cycle}_{= k}{(\mathcal{A})}{(\boldsymbol{z})}
\end{equation}
with a (heuristically chosen) truncation threshold. See further comments about
the cycle construction and the function \( \log \log \frac{1}{1 - e^x} \)
in~\cref{section:barriers:cycle:set}.

Finally, consider the power set construction $\PSet{(\mathcal{A})}$:
\begin{equation}\label{eq:paganini:pset}
    \PSet{(\mathcal{A})}(\boldsymbol{z}) = \exp \left(\sum_{k \geq 1}
    \dfrac{{(-1)}^{k-1}}{k} A(\boldsymbol{z}^k) \right) \qquad \xrightarrow{\text{log-exp}}
    \qquad \pi \geq \sum_{k \geq 1} \dfrac{{(-1)}^{k-1} e^{\alpha_k}}{k}
\end{equation}
where $e^{\alpha_k} = {A(\boldsymbol{z}^k)}$ and $e^{\pi} =
\PSet{(\mathcal{A})}(\boldsymbol{z})$.

Due to the alternating summation and subtraction in~\eqref{eq:paganini:pset},
the right-hand side of the constraint cannot be expressed as an elementary,
convex, {\sf DCP} expression. Consequently, it is not supported in our prototype
implementation.

\subsubsection{Labelled constructions}
Consider the labelled set operator $\Set{(\mathcal{A})}$. Recall that for both
the restricted and unrestricted variants we have
\begin{align}
    \begin{split}
        \Set{(\mathcal{A})}(\boldsymbol{z}) = e^{A(\boldsymbol{z})} \qquad &\xrightarrow{\text{log-exp}}\qquad \sigma \geq
        e^{\alpha}\\
    {\Set}_{= k}{(\mathcal{A})}(\boldsymbol{z}) = \dfrac{1}{k!}
    {A(\boldsymbol{z})}^k  \qquad &\xrightarrow{\text{log-exp}}\qquad
    \sigma_k \geq \log \dfrac{e^{\alpha k}}{k!}
    \end{split}
\end{align}
where $e^{\alpha} = {A(\boldsymbol{z})}$, $e^{\sigma} =
\Set{(\mathcal{A})}(\boldsymbol{z})$, and $e^{\sigma_k} = {\Set}_{=
k}{(\mathcal{A})}(\boldsymbol{z})$.

In this form it is clear that both log-exp transformations can be expressed in
terms of elementary {\sf DCP} functions. Hence, both $\Set{(\mathcal{A})}$ and
${\Set}_{= k}{(\mathcal{A})}$ can be effectively handled.

Let us now focus on the final cycle operator $\Cycle{(\mathcal{A})}$. Note that
\begin{align}
    \begin{split}
    \Cycle{(\mathcal{A})}(\boldsymbol{z}) = \log \dfrac{1}{1 -
    A(\boldsymbol{z})}\qquad &\xrightarrow{\text{log-exp}}\qquad \gamma \geq \log \log
    \dfrac{1}{1- e^{\alpha}}\\
    {\Cycle}_{= k}{(\mathcal{A})}(\boldsymbol{z}) = \dfrac{1}{k}
    {A(\boldsymbol{z})}^k  \qquad &\xrightarrow{\text{log-exp}}\qquad
    \gamma_k \geq \log \dfrac{e^{\alpha k}}{k}
    \end{split}
\end{align}
where $e^{\alpha} = {A(\boldsymbol{z})}$, $e^{\gamma} =
\Cycle{(\mathcal{A})}(\boldsymbol{z})$, and $e^{\gamma_k} = {\Cycle}_{=
k}{(\mathcal{A})}(\boldsymbol{z})$.

The log-exp transform of ${\Cycle}_{= k}$ is an elementary {\sf DCP} constraint.
Unfortunately, the same does not hold for ${\Cycle}$ and, in general,
expressions in form of
\begin{equation}\label{eq:paganini:cycle:loglog}
    \log \log \dfrac{1}{1 - e^x}.
\end{equation}
Although~\eqref{eq:paganini:cycle:loglog} is convex, it cannot be modelled in
terms of basic {\sf DCP} functions. Consequently, heuristic approaches (as discussed
above) or alternative convex programming techniques like the interior point method should be applied, instead.
More labelled constructions, and the ways to bypass this limitation are discussed in~\cref{section:barriers:cycle:set}.

\subsection{Sampler construction}
Given a combinatorial specification enriched with user-specified parameters and
their target expectations, it is possible to \emph{mechanically} compute
respective tuning values, and compile dedicated samplers. To illustrate this
point, we implemented a \emph{sampler compiler} called {\sf Boltzmann
Brain}\footnote{see~\url{https://github.com/maciej-bendkowski/boltzmann-brain}.}
supporting algebraic, and in particular rational, specifications.

Constructed samplers for general algebraic specifications use the principle of
\emph{anticipated rejection} whereas samplers for rational specifications
implement the idea of \emph{interruptible sampling}, see~\cite{BodGenRo2015}.
In addition, both sampler types are endowed with near-optimal decision trees
based on established branching probabilities~\cite{KY76}.  Combined,
these optimisations lead to remarkably efficient samplers, supporting all of the
algebraic specifications included in the current paper,
see~\cref{sec:applications}.

\section{Applications}\label{sec:applications}
\label{section:applications}
In this section we present several examples illustrating the wide range of
applications of our tuning techniques.

\subsection{Polyomino tilings.}
We start with a benchmark example of a rational specification defining ${n
\times 7}$ rectangular tilings using up to $126$ different tile variants (a toy
example of so-called transfer matrix models, cf.~\cite[Chapter V.6, Transfer
matrix models]{flajolet09}).

\begin{figure}[hbt!]
    \begin{center}
        \includegraphics[height=0.02\textheight]{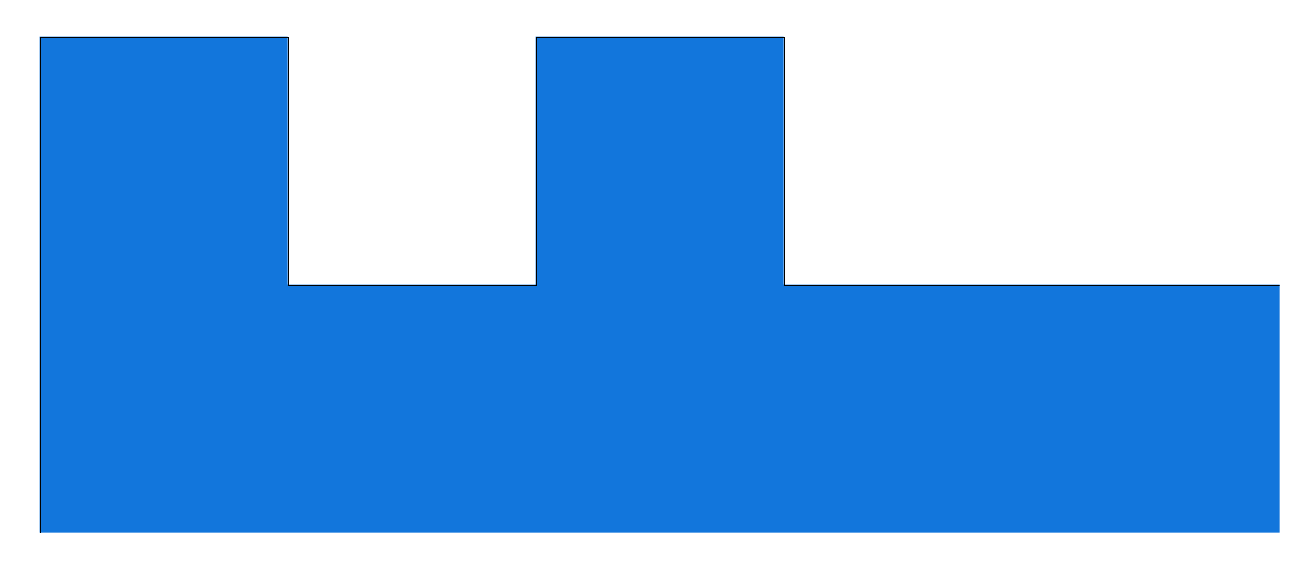} $\ $
        \includegraphics[height=0.02\textheight]{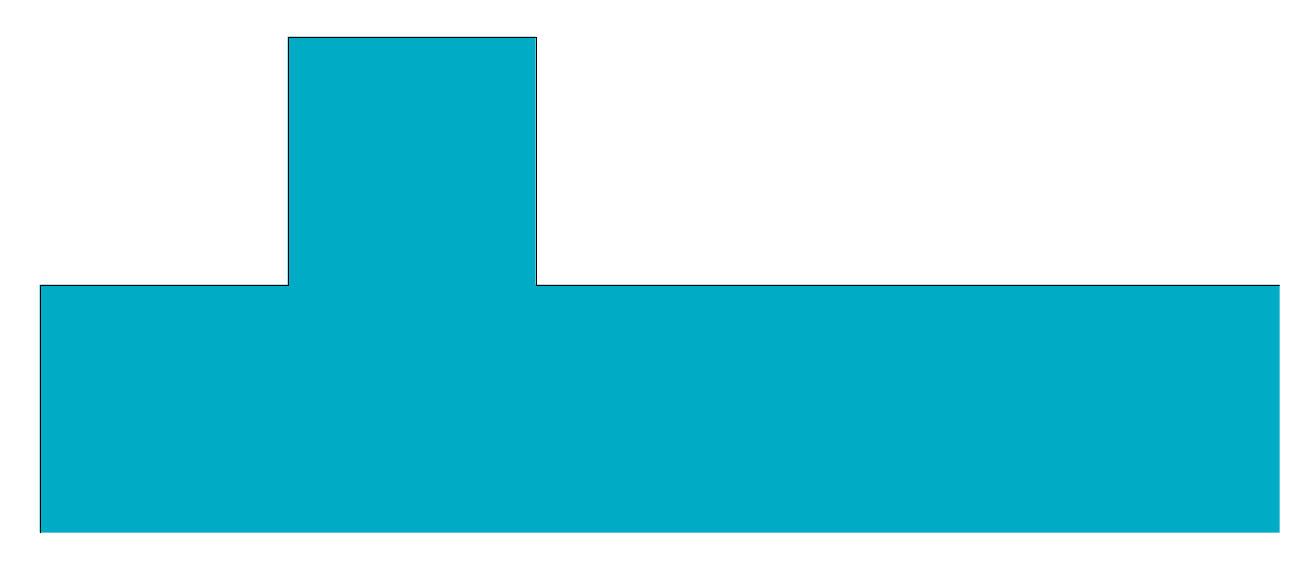} $\ $
        \includegraphics[height=0.02\textheight]{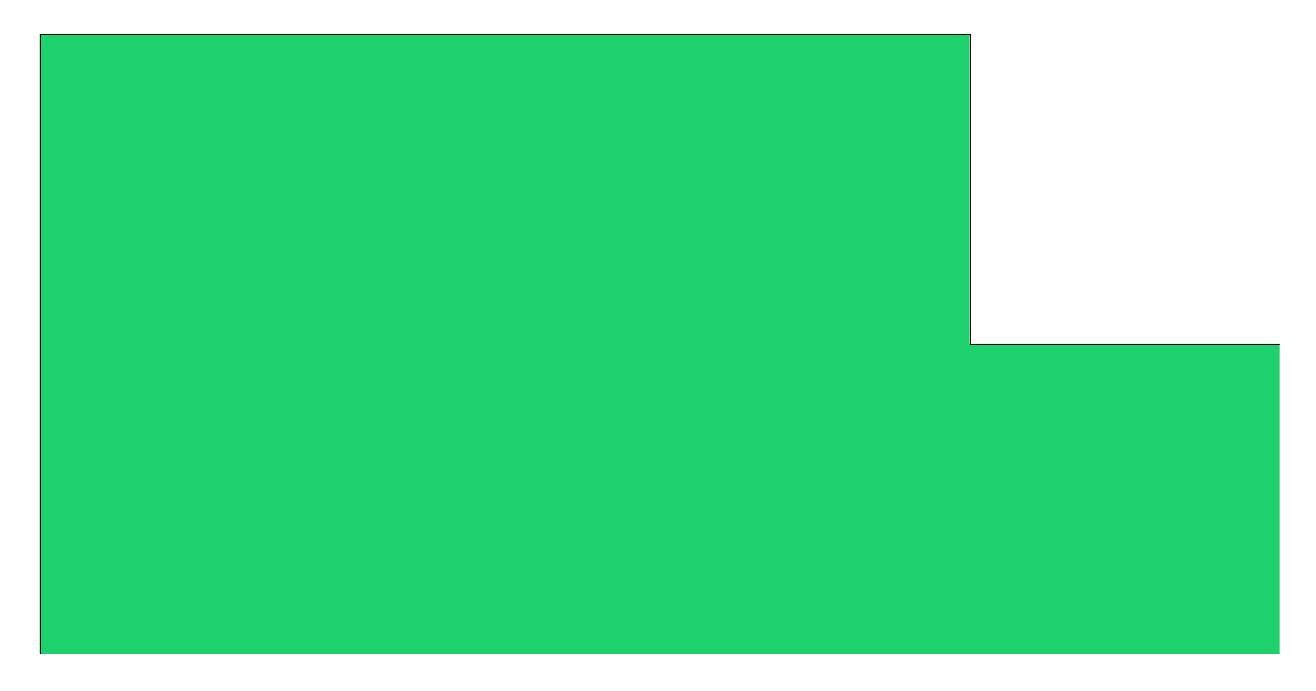} $\ $
        \includegraphics[height=0.02\textheight]{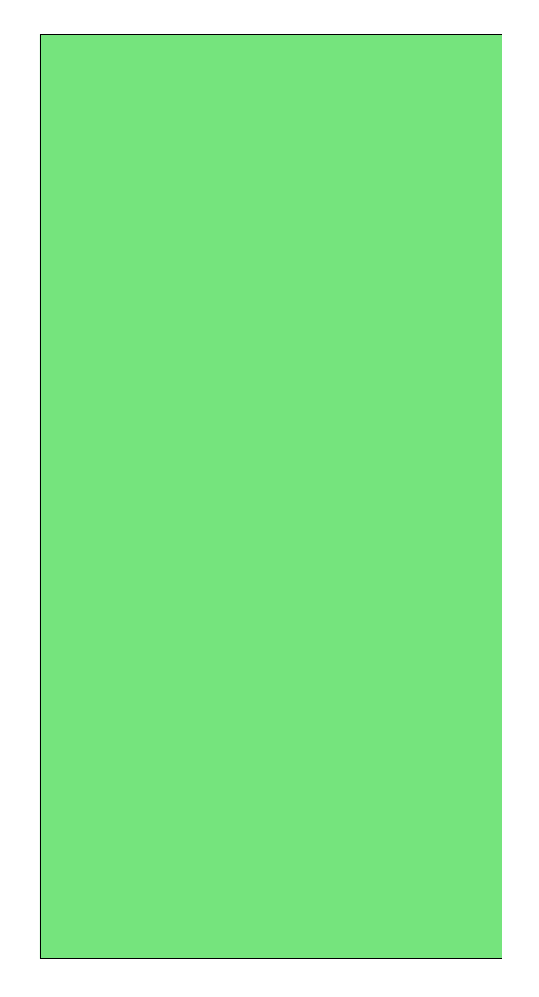} $\ $
        \includegraphics[height=0.02\textheight]{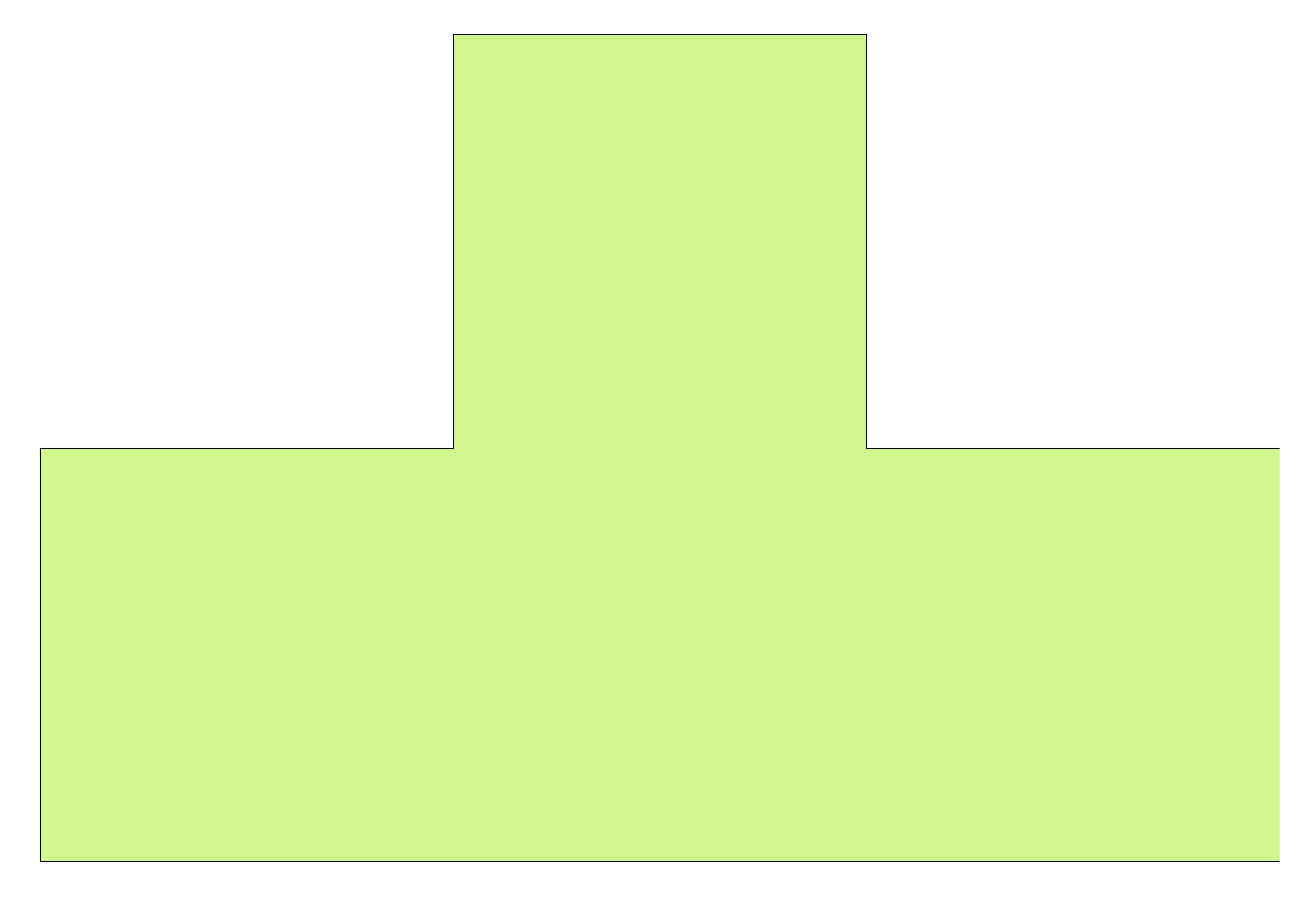} $\ $
        \includegraphics[height=0.02\textheight]{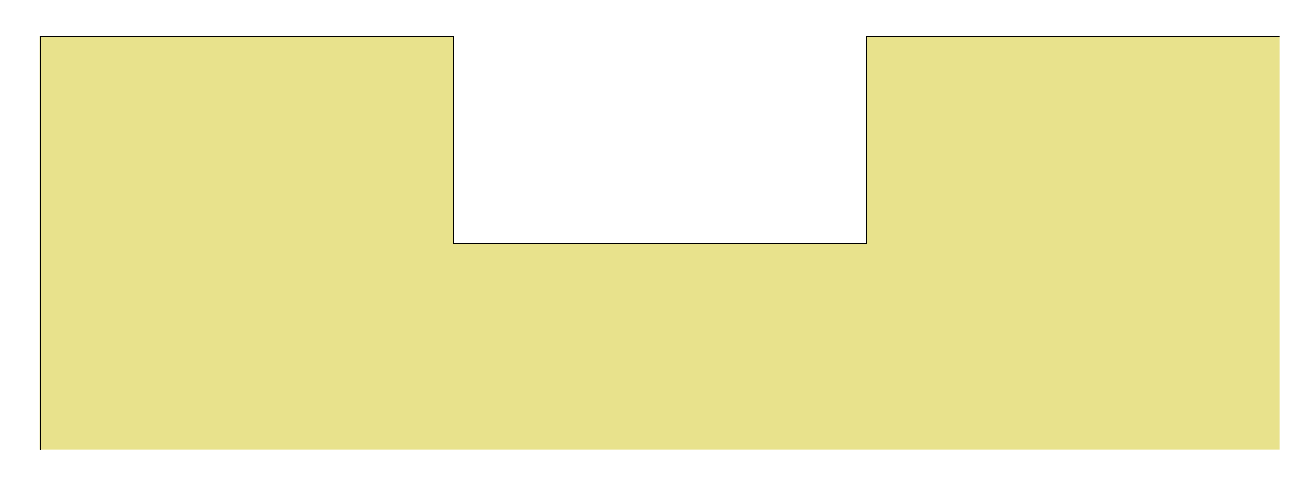} $\ $
        \includegraphics[height=0.02\textheight]{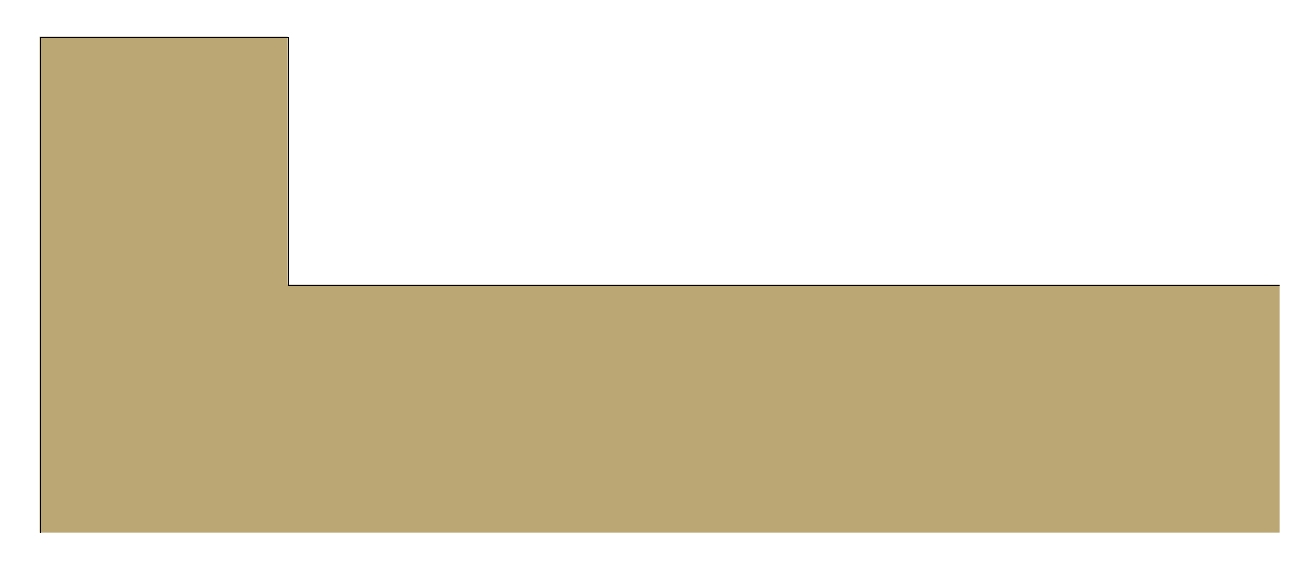} $\ $
        \includegraphics[height=0.02\textheight]{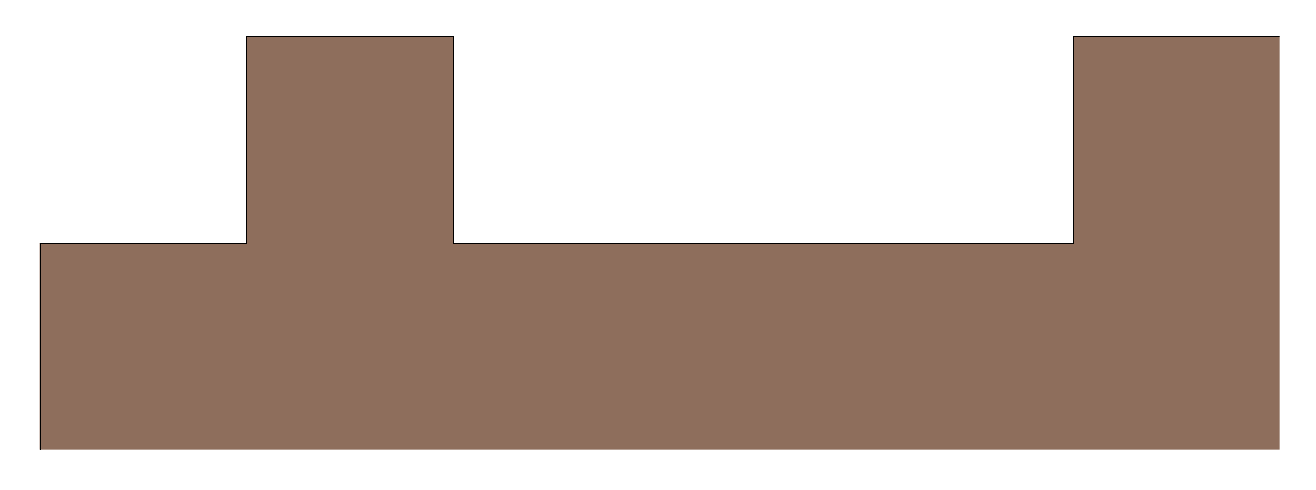} $\ $
        \includegraphics[height=0.02\textheight]{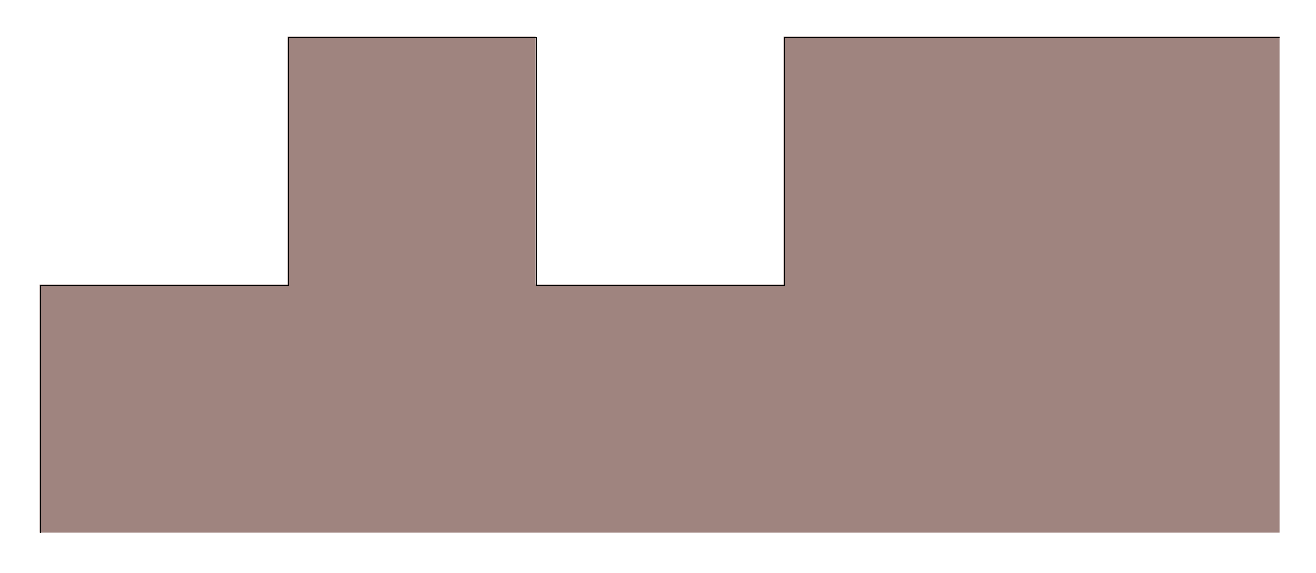} $\ $
        \includegraphics[height=0.02\textheight]{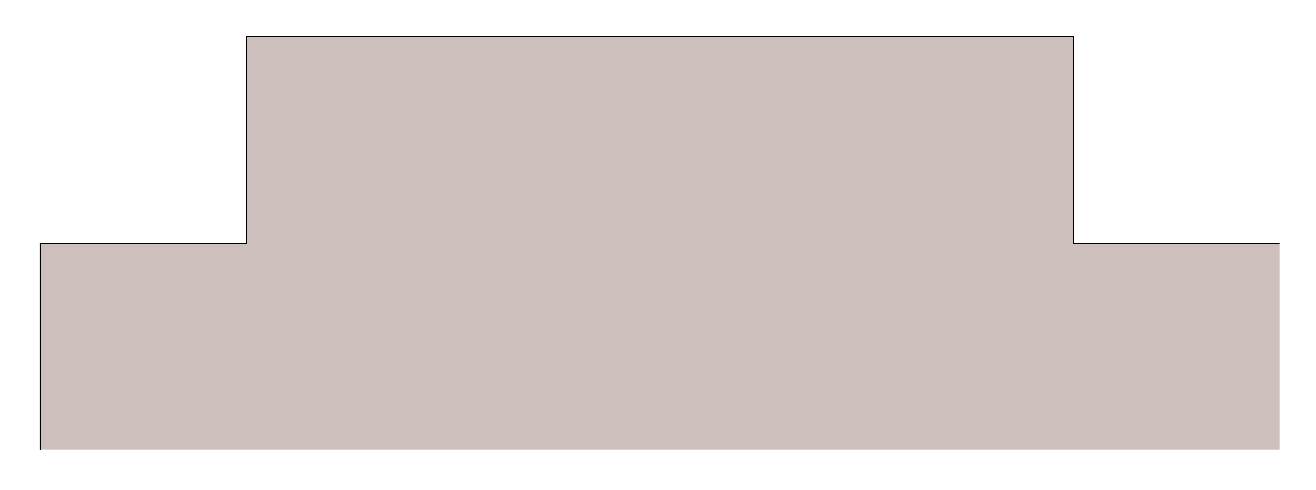}
    \end{center}
    \caption{Examples of admissible tiles}
    \label{fig:admissible:tiles}
\end{figure}

We begin the construction with defining the set $T$ of admissible tiles.  Each
tile $t \in T$ consists of two horizontal layers. The base layer is a single
connected block of width $w_t \leq 6$.  The second layer, placed on top of the
base one, is a subset (possibly empty) of $w_t$ blocks,
see \autoref{fig:admissible:tiles}.  For presentation purposes each
tile is given a unique, distinguishable colour.

Next, we construct the asserted rational specification following the general
construction method of defining a deterministic automaton with one state per
each possible partial tiling configuration using the set $T$ of available
tiles.  Tracking the evolution of attainable configurations while new tiles
arrive, we connect relevant configurations by suitable transition rules in the
automaton.  Finally, we (partially) minimise the constructed automaton removing
states unreachable from the initial empty configuration.  Once the automaton is
created, we tune the tiling sampler such that the target colour frequencies are
uniform, i.e.~each colour occupies, on average, approximately $\tfrac{1}{126}
\approx 0.7936\%$ of the outcome tiling area. \autoref{fig:tilings-7-126}
depicts an exemplary tiling generated by our sampler.
\begin{figure}[ht!]
\begin{subfigure}{.11\textwidth}
\centering
  \includegraphics[scale=0.8]{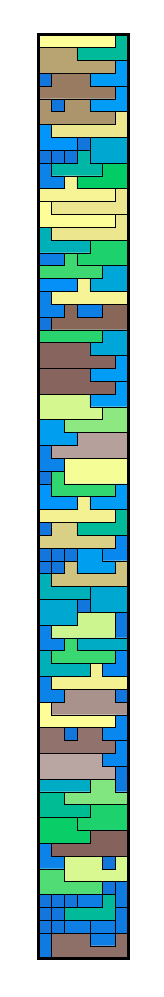}
\end{subfigure}
\begin{subfigure}{.11\textwidth}
\centering
  \includegraphics[scale=0.8]{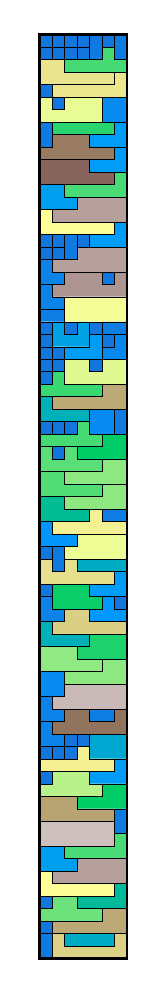}
\end{subfigure}
\begin{subfigure}{.11\textwidth}
\centering
  \includegraphics[scale=0.8]{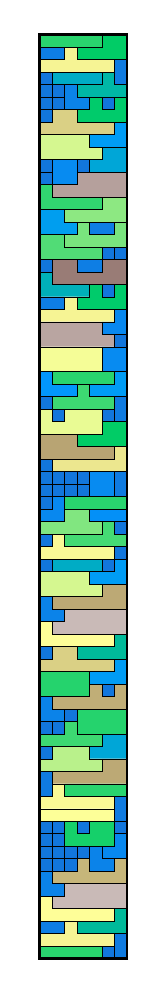}
\end{subfigure}
\begin{subfigure}{.11\textwidth}
\centering
  \includegraphics[scale=0.8]{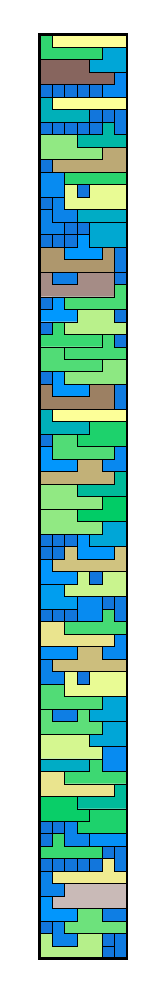}
\end{subfigure}
\begin{subfigure}{.11\textwidth}
\centering
  \includegraphics[scale=0.8]{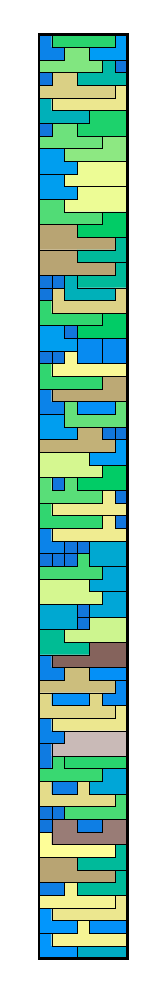}
\end{subfigure}
\begin{subfigure}{.11\textwidth}
\centering
  \includegraphics[scale=0.8]{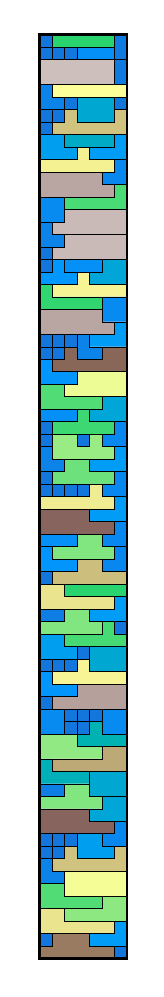}
\end{subfigure}
\begin{subfigure}{.11\textwidth}
\centering
  \includegraphics[scale=0.8]{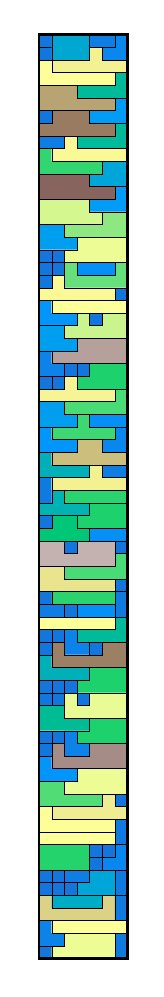}
\end{subfigure}
\begin{subfigure}{.11\textwidth}
\centering
  \includegraphics[scale=0.8]{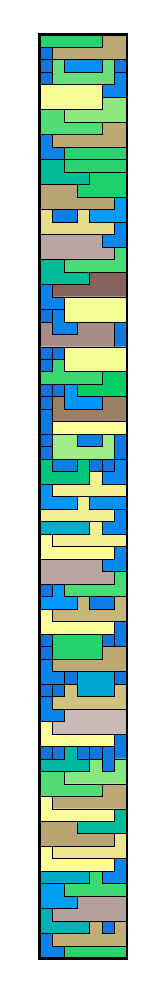}
\end{subfigure}
\caption{Eight random ${n \times 7}$ tilings of areas in the interval
$[500; 520]$ using in total $95$ different tiles.}
  \label{fig:tilings-7-126}
\end{figure}

\begin{Remark}
    The automaton corresponding to our ${n \times 7}$ tiling sampler consists of
    more than $2 000$ states and $28, 000$ transitions. Pushing this toy
    construction to its extreme, we were able to develop a sampler for ${n
    \times 9}$ tilings, using $1022$ different tiles. The corresponding
    automaton has more than $19, 000$ states and $357, 000$ edges.

    We remark that both examples are a notable improvement over the work of
    Bodini and Ponty~\cite{BodPonty} who were able to sample ${n \times 6}$
    tilings using $7$ different tiles with a corresponding automaton consisting
    of roughly $1 500$ states and $3 200$ transitions.
\end{Remark}

\subsection{Simply-generated trees with node degree constraints.}
Next, we give an example of simple varieties of plane trees with fixed sets of
admissible node degrees, satisfying the general equation
\begin{equation}
y(z) = z \phi(y(z))
\quad \text{for some polynomial} \quad \phi \colon \mathbb{C} \to \mathbb{C}\,
.
\end{equation}
Let us consider the case of plane trees where nodes have degrees in the set $D
= \set{0,\ldots,9}$, i.e.~$\phi(y(z)) = a_0 + a_1 y(z) + a_2 {y(z)}^2 + \cdots
+ a_{9} {y(z)}^{9}$.  Here, the numbers \( a_0, a_1, a_2, \ldots, a_{9} \) are
non-negative real coefficients.  We tune the corresponding algebraic
specification so to achieve a target frequency of $1\%$ for all nodes of
degrees $d \geq 2$. Frequencies of nodes with degrees $d \leq 1$ are left
undistorted. For presentation purposes all nodes with equal degree are given
the same unique, distinguishable colour. \autoref{fig:tree} depicts two
exemplary trees generated in this manner.
 \begin{figure}[ht!]
  \begin{subfigure}{.32\paperwidth}
      \centering
  \includegraphics[scale=0.07]{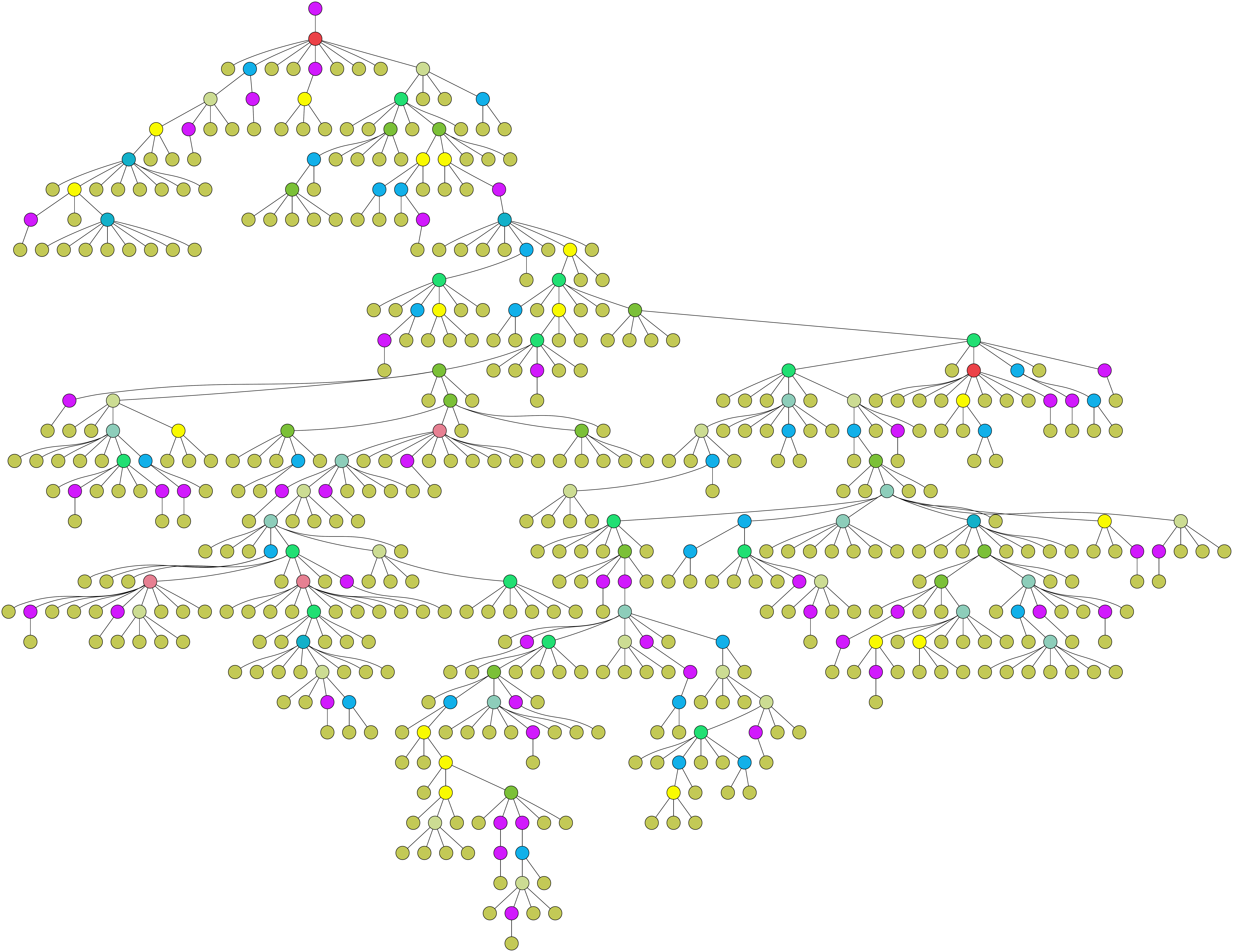}
\end{subfigure}
\begin{subfigure}{.4\paperwidth}
    \centering
  \includegraphics[width=\textwidth]{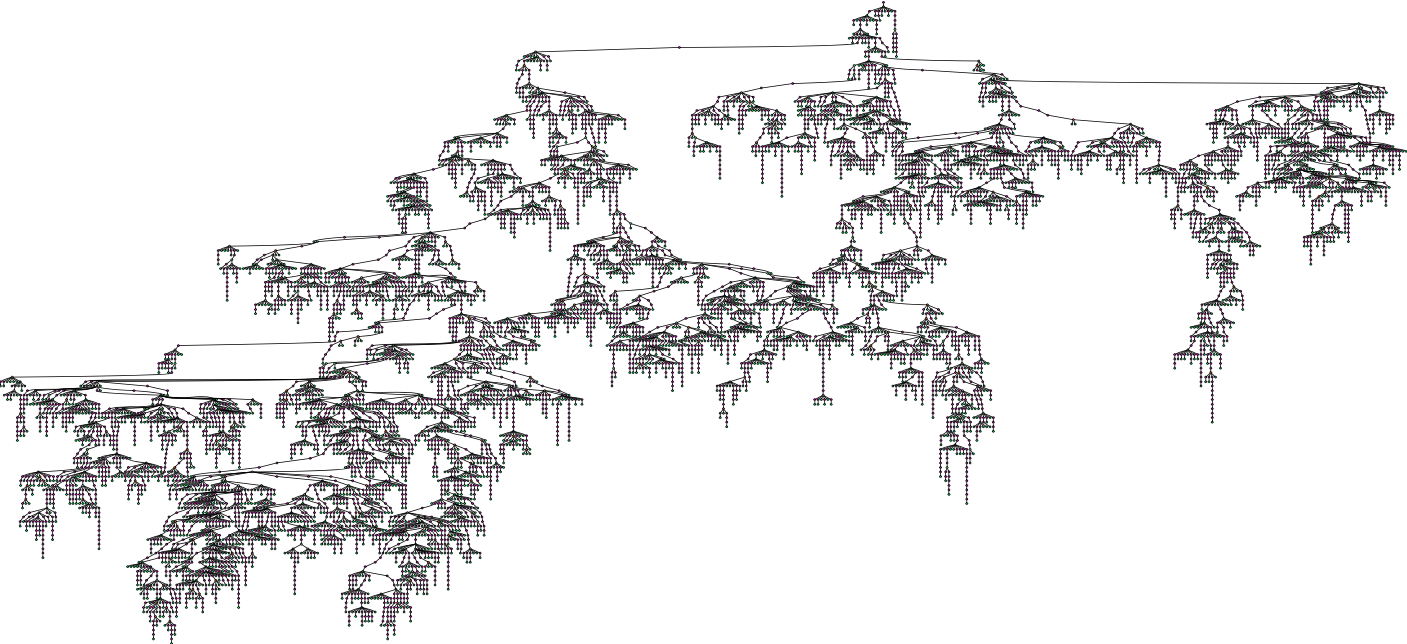}
\end{subfigure}
 \caption{Two random plane trees with degrees in the set ${D =
     \set{0,\ldots,9}}$. On the left, a tree of size in between $500$ and $550$;
     on the right, a tree of size in the interval $[10, 000; 10, 050]$.}
\label{fig:tree}
 \end{figure}

Empirical frequencies for the right tree of~\autoref{fig:tree} and a
simply-generated tree of size in between $10,000$ and $10,050$ with default
node degree frequencies are included in~\autoref{fig:tree-freqs}.

\begin{table*}[ht!]
\scalebox{.8}{
\begin{tabular}{c | c | c | c | c | c | c | c | c | c | c}
Node degree & $0$ & $1$ & $2$ & $3$ & $4$ & $5$ &
$6$ & $7$ & $8$ & $9$ \\ \hline
Tuned frequency & - - - & - - - & $1.00\%$ & $1.00\%$ & $1.00\%$ &
$1.00\%$ & $1.00\%$ & $1.00\%$ & $1.00\%$ & $1.00\%$\\
Observed frequency & $35.925\%$ & $56.168\%$ & $0.928\%$ & $0.898\%$ & $1.098\%$ &
$0.818\%$ & $1.247\%$ & $0.938\%$ & $1.058\%$ & $0.918\%$\\
Default frequency & $50.004\%$ & $24.952\%$ & $12.356\%$ & $6.322\%$ &
$2.882\%$ & $1.984\%$ & $0.877\%$ & $0.378\%$ & $0.169\%$ & $0.069\%$
\end{tabular}}
\caption{Empirical frequencies of the node degree distribution.}
\label{fig:tree-freqs}
\end{table*}

We briefly remark that for this particular problem, Bodini, David and Marchal
proposed a different, bit-optimal sampling procedure for random trees with
given partition of node degrees~\cite{DBLP:conf/caldam/BodiniJM16}.

\subsection{Variable distribution in plain $\lambda$-terms.}
To exhibit the benefits of distorting the intrinsic distribution of various
structural patterns in algebraic data types, we present an example
specification defining so-called plain $\lambda$\nobreakdash-terms with explicit control
over the distribution of de~Bruijn indices.

In their nameless representation due to de~Bruijn~\cite{deBruijn1972}
$\lambda$\nobreakdash-terms are defined by the formal grammar $L ::= \lambda L~|~(L L)~|~D$
where $D = \set{\idx{0},\idx{1},\idx{2},\ldots}$ is an infinite denumerable set
of so-called indices (cf.~\cite{BendkowskiGLZ16,GittenbergerGolebiewskiG16}).
Assuming that we encode de~Bruijn indices as a sequence of successors of zero
(i.e.~use a unary base representation), the class $\mathcal{L}$ of plain
$\lambda$\nobreakdash-terms can be specified as \( \mathcal{L} = \mathcal{Z} \mathcal{L} +
\mathcal{Z} {\mathcal{L}}^2 + \mathcal{D} \) where \( \mathcal{D} = \mathcal{Z}
\Seq(\mathcal{Z}) \).  In order to control the distribution of de~Bruijn
indices we need a more explicit specification for de~Bruijn indices. For
instance: \[ \mathcal{D} = \mathcal{U}_0 \mathcal{Z} + \mathcal{U}_1
    {\mathcal{Z}}^2 + \cdots + \mathcal{U}_k {\mathcal{Z}}^{k+1} +
\mathcal{Z}^{k+2} \Seq(\mathcal{Z})\, . \] Here, we roll out the $k+1$ initial indices and
assign distinct marking variables to each one of them, leaving the remainder
sequence intact. In doing so, we are in a position to construct a sampler tuned
to enforce a uniform distribution of $8\%$ among all marked indices,
i.e.~indices $\idx{0},\idx{1},\ldots,\idx{8}$, distorting in effect their
intrinsic geometric distribution.

\autoref{fig:lambda-term}~illustrates two random $\lambda$\nobreakdash-terms with
such a new distribution of indices. For presentation purposes, each index
in the left picture is given a distinct colour.

 \begin{figure}[ht!]
  \begin{subfigure}{.4\textwidth}
\centering
  \includegraphics[scale=0.1]{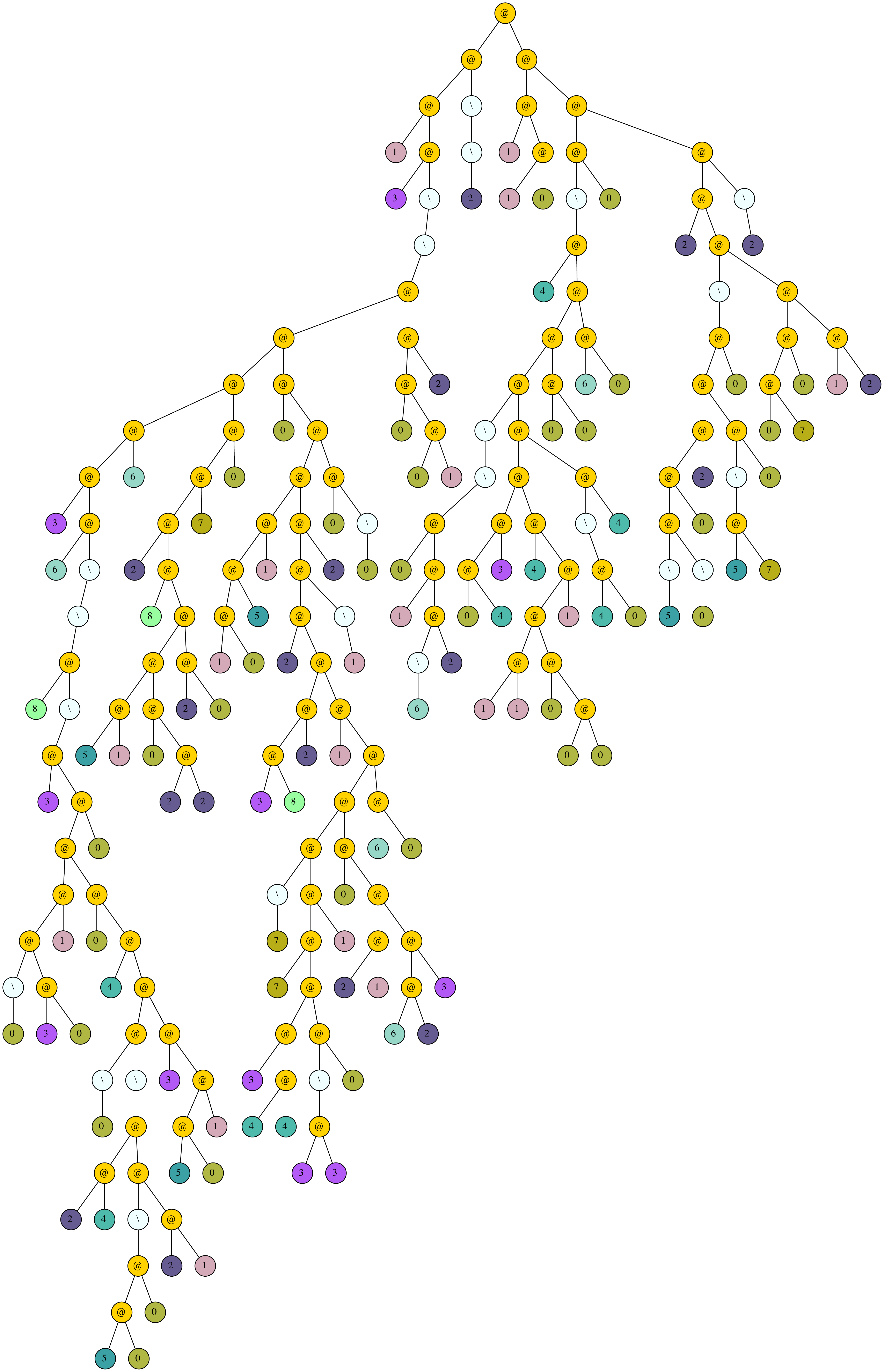}
\end{subfigure}
\begin{subfigure}{.4\textwidth}
\centering
  \includegraphics[width=\textwidth]{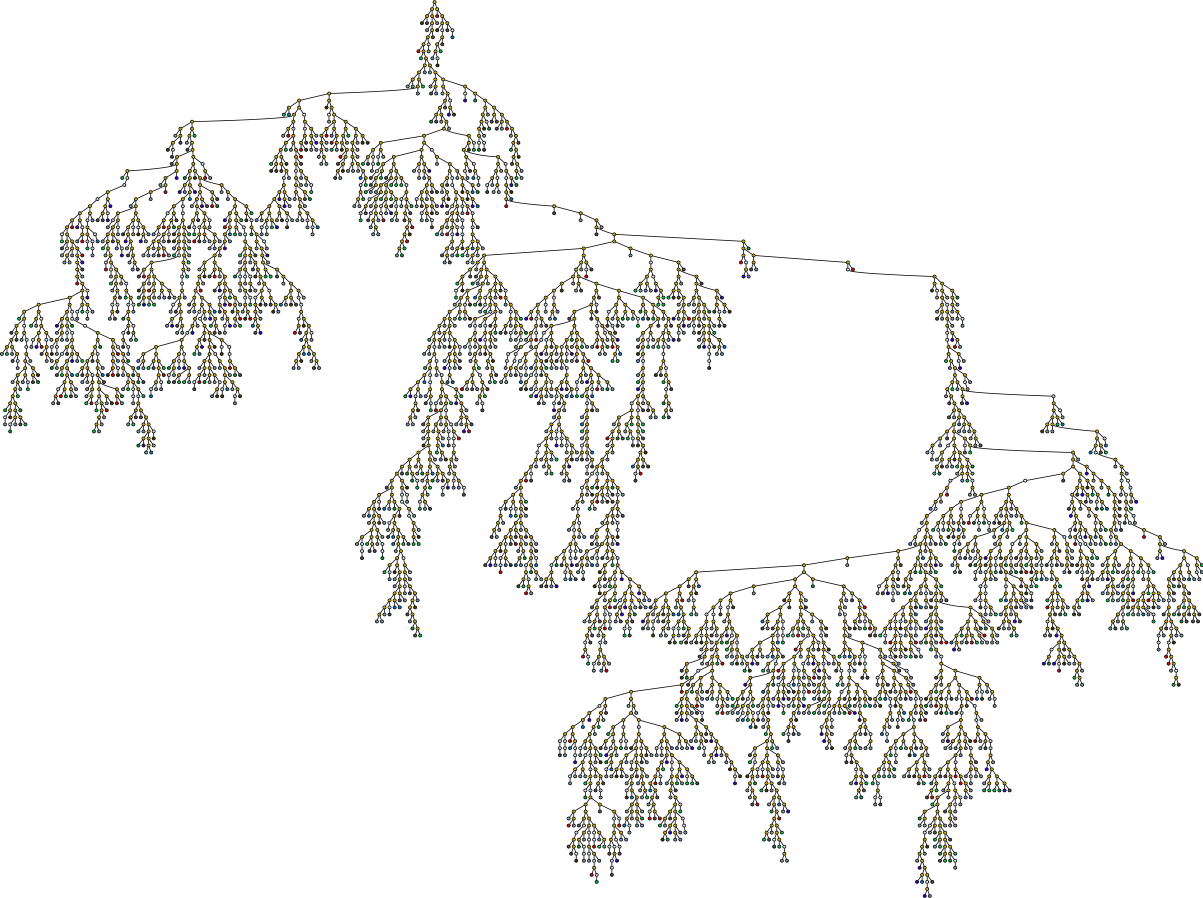}
\end{subfigure}
 \caption{On the left, a random $\lambda$\nobreakdash-term of size in the interval
 $[500;550]$; on the right, a larger example of a random $\lambda$\nobreakdash-term of size between $10,000$ and $10,050$.}
\label{fig:lambda-term}
 \end{figure}

Empirical frequencies for the right term of~\autoref{fig:lambda-term}
and a plain $\lambda$\nobreakdash-term of size in between $10,000$ and $10,050$ with
default de~Bruijn index frequencies are included
in~\autoref{fig:lambda-term-freqs}.

\begin{table*}[ht!]
\scalebox{.85}{
\begin{tabular}{c | c | c | c | c | c | c | c | c | c}
Index & $\idx{0}$ & $\idx{1}$ & $\idx{2}$ & $\idx{3}$ & $\idx{4}$ & $\idx{5}$ &
$\idx{6}$ & $\idx{7}$ & $\idx{8}$\\ \hline
Tuned frequency & $8.00\%$ & $8.00\%$ & $8.00\%$ & $8.00\%$ & $8.00\%$ &
$8.00\%$ & $8.00\%$ & $8.00\%$ & $8.00\%$\\
Observed frequency & $7.50\%$ & $7.77\%$ & $8.00\%$ & $8.23\%$ & $8.04\%$ &
$7.61\%$ & $8.53\%$ & $7.43\%$ & $9.08\%$\\
Default frequency & $21.91\%$ & $12.51\%$ & $5.68\%$ & $2.31\%$ &
$0.74\%$ & $0.17\%$ & $0.20\%$ & $0.07\%$ & - - -
\end{tabular}}
\caption{Empirical frequencies (with respect to the term size) of index distribution.}
\label{fig:lambda-term-freqs}
\end{table*}

Let us note that algebraic data types, an essential conceptual ingredient of
various functional programming languages such as Haskell or OCaml, and the
random generation of their inhabitants satisfying additional structural or
semantic properties is one of the central problems present in the field of
property-based software testing (see, e.g.~\cite{Claessen-2000,palka2012}). In
such an approach to software quality assurance, programmer-declared function
invariants (so-called properties) are checked using random inputs, generated
accordingly to some predetermined, though usually not rigorously controlled,
distribution. In this context, our techniques provide a novel and effective
approach to generating random algebraic data types with fixed average
frequencies of type constructors. In particular, using our methods it is
possible to \emph{boost} the intrinsic frequencies of certain desired subpatterns or
\emph{diminish} those which are unwanted.

\subsection{Weighted partitions.}
\label{subsection:weighted:partitions}
Integer partitions are one of the most intensively studied objects in number
theory, algebraic combinatorics and statistical physics. Hardy and Ramanujan
obtained the famous asymptotics which has later been refined by
Rademacher~\cite[Chapter VIII]{flajolet09}. In his article~\cite{Vershik1996},
Vershik considers several combinatorial examples related to statistical
mechanics and obtains the limit shape for a random integer partition of size \(
n \) with \( \alpha \sqrt{n} \) parts and summands bounded by \( \theta
\sqrt{n} \).
Let us remark that Bernstein, Fahrbach, and
Randall~\cite{doi:10.1137/1.9781611975062.10} have recently analysed
the complexity of exact-size Boltzmann sampler for weighted partitions.
In the model of ideal gas, there are several particles (bosons)
which form a so-called assembly of particles. The overall energy of the system
is the sum of the energies
\(
\Lambda = \sum_{i=1}^N \lambda_{\vec i}
\)
where \( \lambda_i \) denotes the energy of \( i \)-th particle.  We assume
that energies are positive integers. Depending on the energy level \( \lambda
\) there are \( j(\lambda) \) possible available states for each particle; the
function \( j(\lambda) \) depends on the physical model. Since all the
particles are indistinguishable, the generating function \( P(z) \) for the
number of assemblies \( p(\Lambda) \) with energy \( \Lambda \) takes the form
\begin{equation}
    P(z) = \sum_{\Lambda=0}^\infty p(\Lambda) z^\Lambda =
        \prod_{\lambda > 0}
        \dfrac{1}{(1 - z^\lambda)^{j(\lambda)}}
    \enspace .
\end{equation}
In the model of \( d \)-dimensional harmonic trap (also known as the
Bose-Einstein condensation) according to \cite{chase1999canonical,
haugset1997bose, lucietti2008asymptotic} the number of states for a particle
with energy \( \lambda \) is \( { d + \lambda - 1 \choose \lambda } \) so that
each state can be represented as a multiset with \( \lambda \) elements having
\( d \) different colours.  Accordingly, an assembly is a multiset of particles
(since they are bosons and hence indistinguishable) therefore the generating
function for the number of assemblies takes the form
\begin{equation}
    P(z) =
    \MSet(\MSet\nolimits_{\geq 1}(\mathcal Z_1 + \cdots + \mathcal Z_d))
    \enspace .
\end{equation}

It is possible to control the expected frequencies of colours using our tuning
procedure and sample resulting assemblies as Young tableaux. Each row
corresponds to a particle whereas the colouring of the row displays the
multiset of included colours, see~\autoref{fig:bose:einstein}. We also
generated weighted partitions of expected size \( 1000 \) (which are too large
to display) with tuned frequencies of \( 5 \) colours,
see~\autoref{table:partition:frequencies}.

 \begin{figure}[ht!]
\begin{subfigure}{0.2\textwidth}
\centering
    \includegraphics[height=0.2\textheight]{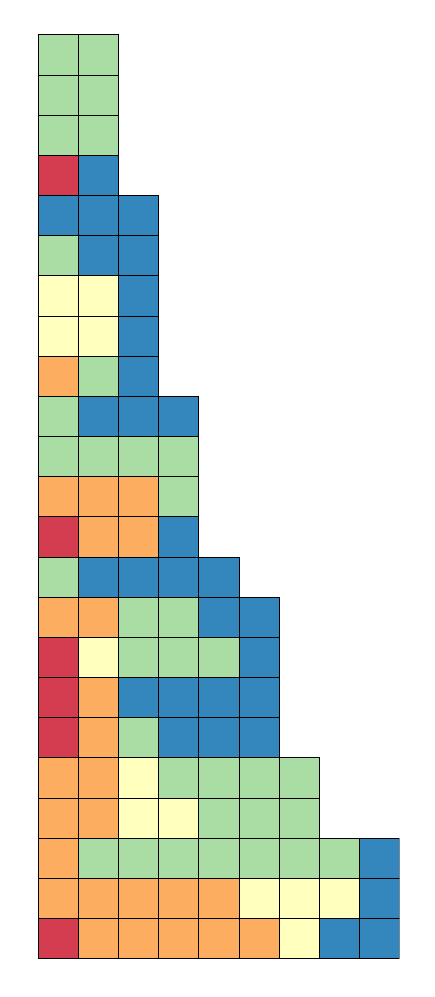}
    \caption{\tiny [5, 10, 15, 20, 25]}
\end{subfigure}
\begin{subfigure}{0.25\textwidth}
\centering
    \includegraphics[height=0.2\textheight]{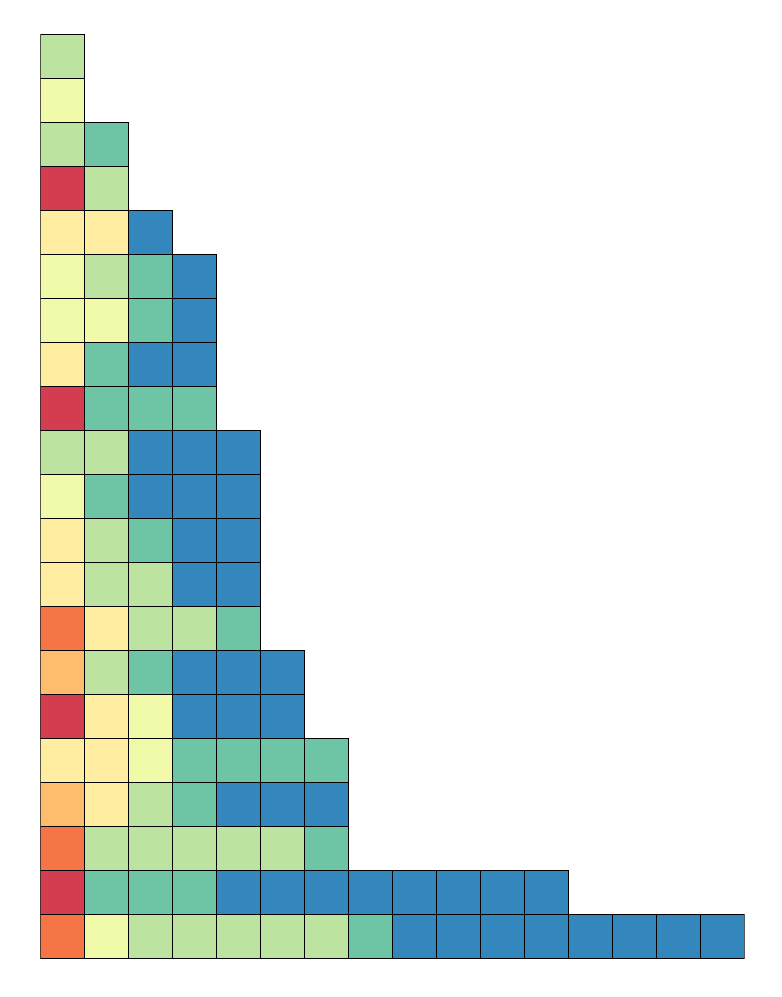}
    \caption{\tiny [4,4,4,4, 10, 20, 30, 40]}
\end{subfigure}
\begin{subfigure}{0.25\textwidth}
\centering
    \includegraphics[height=0.2\textheight]{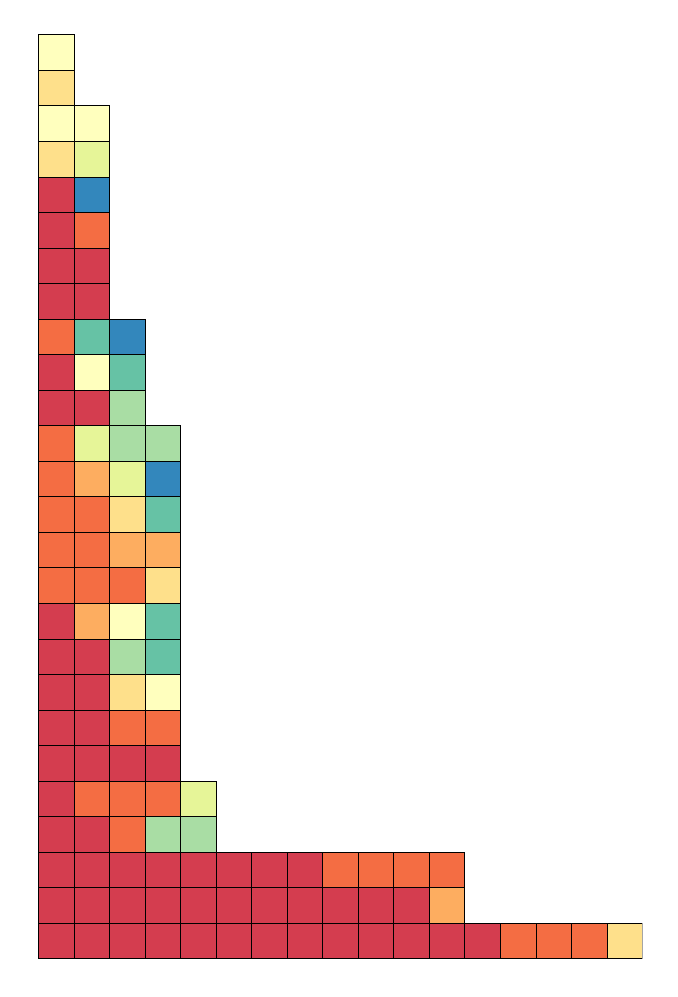}
    \caption{\tiny [80, 40, 20, 10, 9, 8, 7, 6, 5]}
\end{subfigure}
\begin{subfigure}{0.25\textwidth}
\centering
    \includegraphics[height=0.2\textheight]{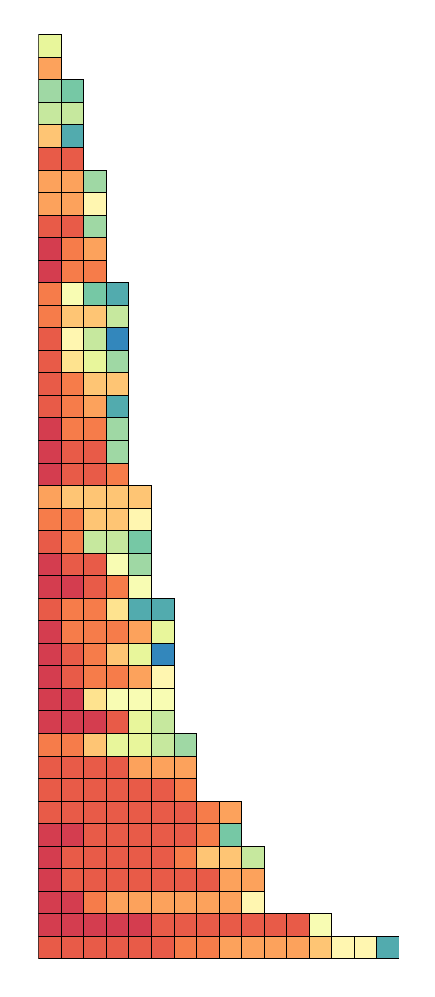}
    \caption{\tiny [20, 60, 30, 20, 10, $5^9$]}
\end{subfigure}
    \caption{Young tableaux corresponding to Bose--Einstein condensates with
        expected numbers of different colours. Notation \( [c_1, c_2, \ldots,
        c_k] \) provides the expected number \( c_j \) of the \( j \)-th
        colour, \( c_k^m \) is a shortcut for \( m \) occurrences of \( c_k \).}
    \label{fig:bose:einstein}
\end{figure}

\begin{table}[ht!]
\begin{tabular}{c | c | c | c | c | c | c}
Colour index & $\idx{1}$ & $\idx{2}$ & $\idx{3}$ & $\idx{4}$ & $\idx{5}$ & size\\ \hline
    Tuned frequency    & $0.03$ & $0.07$ & $0.1$ & $0.3$ & $0.5$ & 1000\\ \hline
                       & $0.03$ & $0.08$ & $0.07$ & $0.33$ & $0.49$ & 957\\
                       & $0.03$ & $0.06$ & $0.09$ & $0.28$ & $0.54$ & 1099\\
    Observed frequency & $0.03$ & $0.08$ & $0.09$ & $0.34$ & $0.46$ & 992\\
                       & $0.04$ & $0.07$ & $0.1$  & $0.31$ & $0.49$ & 932\\
                       & $0.04$ & $0.09$ & $0.1$  & $0.25$ & $0.52$ & 1067\\
\end{tabular}
\caption{Empirical frequencies of colours observed in random partition.}
\label{table:partition:frequencies}
\end{table}

Let us briefly explain our generation procedure. Boltzmann sampling for the outer \(
\MSet \) operator is described in~\cite{flajolet2007boltzmann}.
The sampling of
inner \( \MSet_{\geq 1}(\mathcal Z_1 + \ldots + \mathcal Z_d) \) is more
delicate. The generating function for this multiset can be written as
\begin{equation}
    \MSet\nolimits_{\geq 1}( z_1 + \cdots + z_d ) = \prod_{i=1}^d \dfrac{1}{1 - z_i} - 1
    \enspace .
\end{equation}
In order to correctly calculate the branching probabilities, we introduce slack
variables \( s_1, \ldots, s_d \) satisfying \( (1 + s_i) = (1 - z_i)^{-1}
\). Boltzmann samplers for the newly determined combinatorial classes \( \Gamma
\mathcal S_i \) are essentially Boltzmann samplers for \( \Seq_{\geq
1}(\mathcal Z_i) \). Let us note that after expanding brackets the expression
becomes
\begin{multline}
    \MSet\nolimits_{\geq 1} (z_1 + \cdots + z_d) ={}
    (s_1 + \cdots + s_d)
     + (s_1 s_2 + \cdots + s_{d-1} s_d)
     + \cdots + s_1 s_2 \ldots s_d.
\end{multline}
The total number of summands is \( 2^{d} - 1 \) where each summand corresponds
to choosing some subset of colours. Finally, we pre-compute all the symmetric
polynomials and efficiently handle the branching process in quadratic time using
a dynamic programming approach discussed
in~\cref{section:symmetric:polynomials}.

\subsection{Multi-partite rooted labelled trees}
\label{section:cayley:trees}
Consider a family of rooted labelled trees, such that the children of each node
are not ordered. The exponential generating function of
such trees \( T(z) \) satisfies the equation
  \begin{equation}
    T(z) = z e^{T(z)}.
  \end{equation}
In this example, we suggest an alteration of this model, where the nodes on each
level have distinct colours. We consider a periodic system of colouring, where
the levels \( 1, 2, \ldots, d \) have distinct colours, and then the colours
repeat, \ie the levels \( d+1, \ldots, 2d \) have the same colours as the levels
\( 1, 2, \ldots, d \). Let \( u_1, \ldots, u_d \) be the marking variables for
the respective colours, and let \( T_1, \ldots, T_d \) denote multivariate
exponential generating functions for trees whose root is coloured respectively,
with the colour \( 1, 2, \ldots, d \). Then, these functions satisfy the system
of functional equations
\begin{align}
  \begin{split}
    T_1(z, u_1, \ldots, u_d) &= z u_1 e^{T_{2}(z, u_1, \ldots, u_d)}
    ,\\
    T_2(z, u_1, \ldots, u_d) &= z u_2 e^{T_{3}(z, u_1, \ldots, u_d)}
    ,\\
    & \vdots \\
    T_d(z, u_1, \ldots, u_d) &= z u_d e^{T_{1}(z, u_1, \ldots, u_d)}.
  \end{split}
\end{align}
Using our software (see~\cref{sec:paganini}), we implement the multi-parametric
tuning and sampling when \( d = 10 \) and the proportions of vertices of the
respective colours are sorted in an arithmetic progression
$0.01, 0.03, 0.05,
0.07, 0.09, 0.11, 0.13, 0.15, 0.17, 0.19$.

An example of a resulting tree of size 1665 is shown in~\cref{fig:cayley}, and
the empirical frequencies of the colours inside this tree are given
in~\cref{fig:cayley:colours}.

\begin{figure}[hbt!]
    \RawFloats
\centering
\begin{minipage}[t]{0.5\textwidth}
    \includegraphics[width=\textwidth]{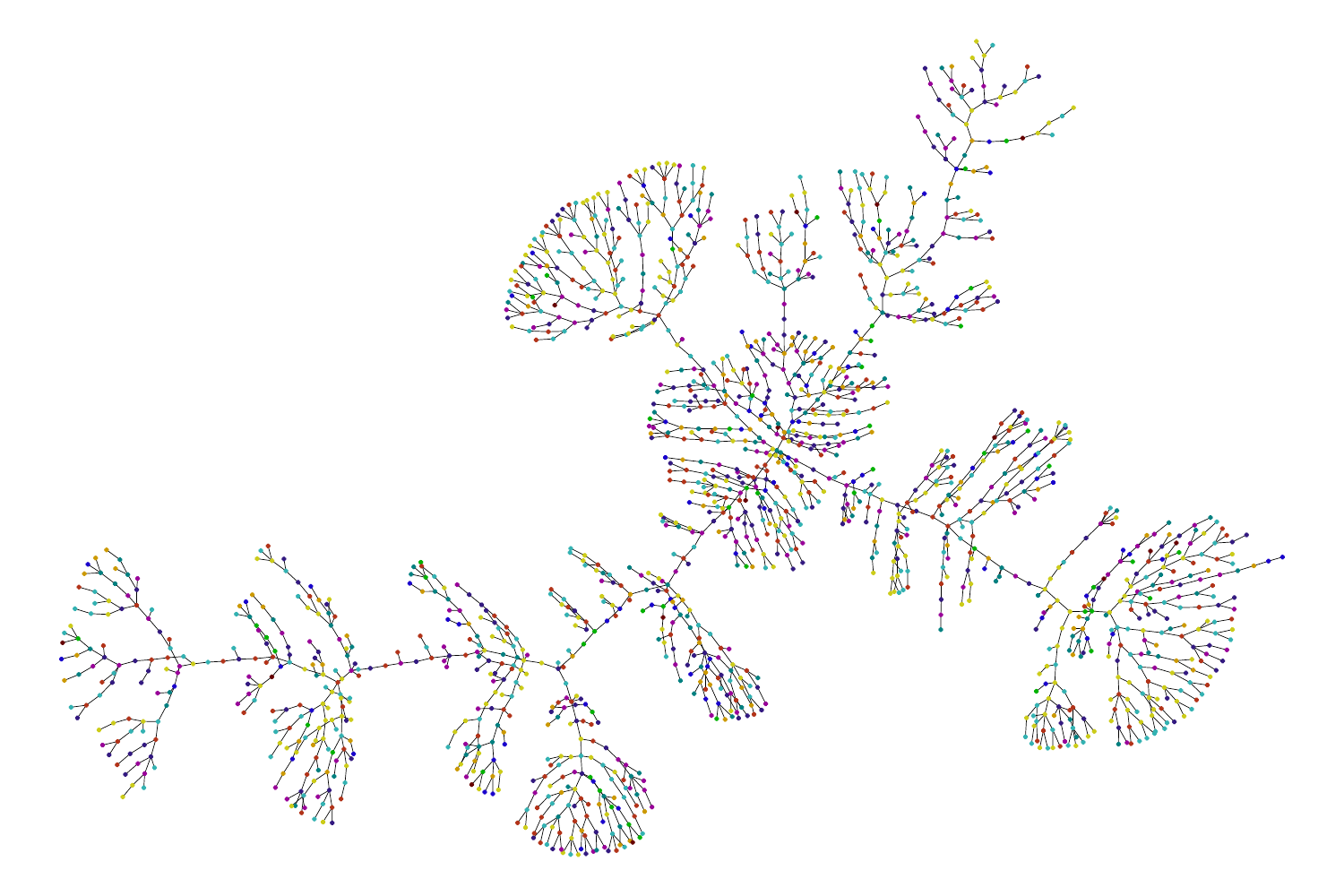}
    \caption{Random coloured Cayley tree of size 1665 drawn with Kamada--Kawai
    algorithm~\cite{kamada1989algorithm}.}
    \label{fig:cayley}
\end{minipage}%
\begin{minipage}[t]{0.5\textwidth}
    \includegraphics[width=\textwidth]{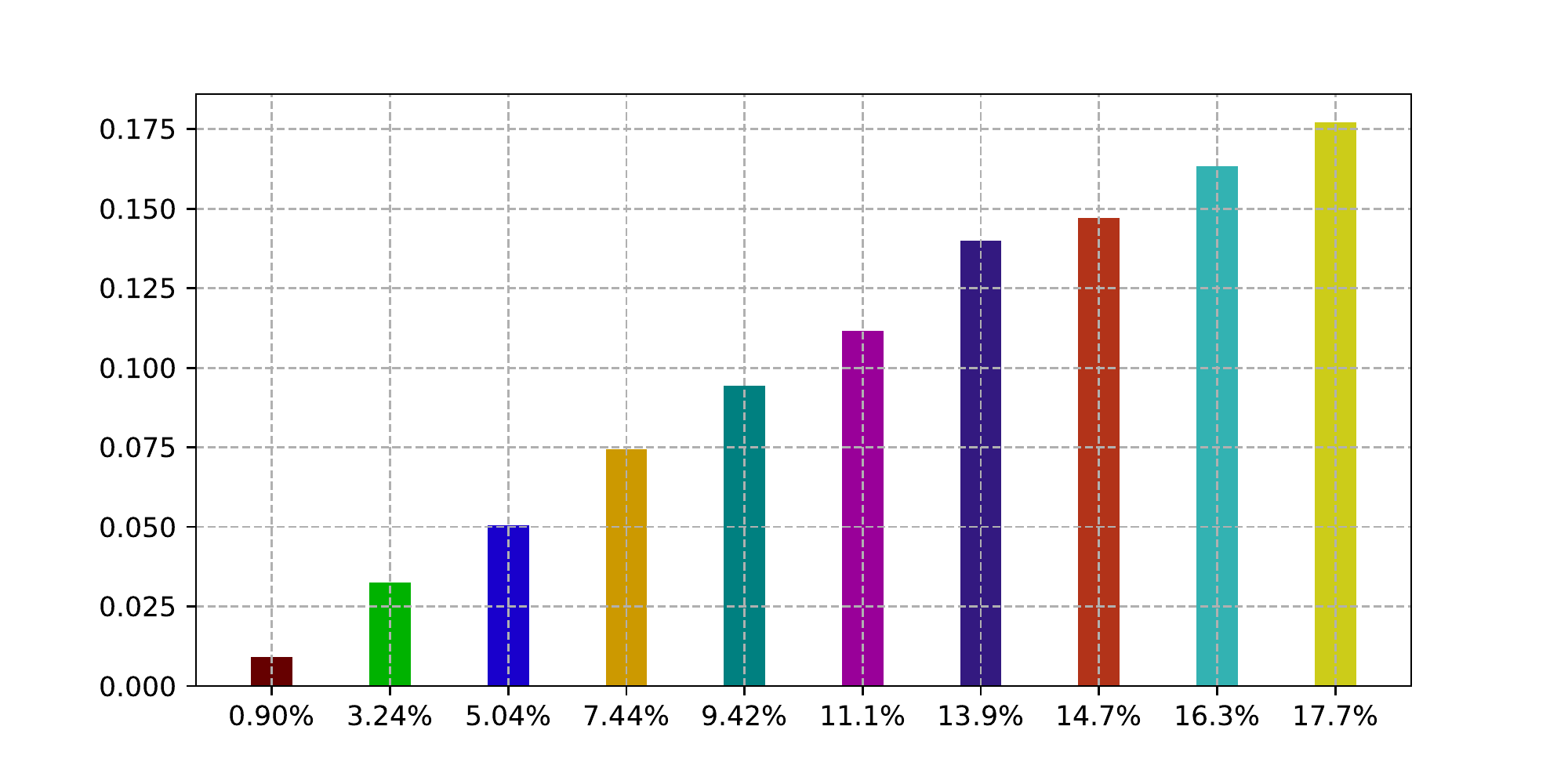}
    \caption{Histogram of colour frequencies for Cayley tree.}
    \label{fig:cayley:colours}
\end{minipage}
\end{figure}

Note that even in such a simple example where the dependency graph forms a
cycle, and the structure of the equations is symmetrical, no simpler tuning
procedure is available --- no variable can be tuned separately from the others,
and the resulting variable values do not form the same arithmetic pattern as the
target weights.
Also, let us point to a curious pattern for the numerical values of the tuning
parameters can be observed in~\cref{tab:numerical:cayley}.

\begin{table}
    \centering
    \begin{tabular}{c|c|c|c|c|c|c|c|c|c|c|c}
        &
        $ z $  &
        $ u_1 $ &
        $ u_2 $ &
        $ u_3 $ &
        $ u_4 $ &
        $ u_5 $ &
        $ u_6 $ &
        $ u_7 $ &
        $ u_8 $ &
        $ u_9 $ &
        $ u_{10} $
        \\
        \hline
        value &
        0.3 &
        0.009 &
        1.88 &
        1.37 &
        1.29 &
        1.26 &
        1.25 &
        1.24 &
        1.23 &
        1.23 &
        3.52
    \end{tabular}
    \caption{Numerical values for multi-partite rooted labelled
    trees parameters.}
    \label{tab:numerical:cayley}
\end{table}

\subsection{Otter trees}
Starting from the seminal paper of Otter~\cite{otter1948number}, unlabelled
tree-like structures play an important role in chemistry,
phylogenetics~\cite{penny1982testing} and synthetic
biology~\cite{fichtner2017tip}. Their study also helps to discover new methods
for numerical solution of partial differential equations, involving automatic
differentiation and construction of expression
trees~\cite{mullier2018validated}.

In this section, we consider a relatively simple example of rooted unlabelled
binary trees, such that the children of each node are not ordered. An additional
assumption that the leaves are coloured in \( d \) distinct colours gives the
following specification:
\begin{align}
    T(z, u_1, \ldots, u_d) =
    z \sum_{i = 1}^d u_i + \MSet\nolimits_2(T(z, u_1, \ldots, u_d)),
\end{align}
where the \( \MSet\nolimits_2\) operator is defined as
  \begin{equation}
    \MSet\nolimits_2(T(z, u_1, \ldots, u_d)) =
    \dfrac{T(z, u_1, \ldots, u_d)^2 + T(z^2, u_1^2, \ldots, u_d^2)}{2}.
  \end{equation}
Recall that the corresponding univariate model was considered
in~\cite{BodLumRolin} where the authors constructed a system of quadratic
equations which could be solved in reverse.
As in the previous example (see~\cref{section:cayley:trees}), we solve the
multi-parametric tuning problem when \( d = 10 \), and the expected colour
frequencies form an arithmetic progression \( 0.01, 0.03, \ldots, 0.19 \).

After setting an appropriate truncation level (the probability of having nodes
on a level \( h \) is exponentially decreasing in \( h \)), we solve the tuning
problem and generate the corresponding trees, see~\cref{fig:otter} for the
generated tree of size \( 1434 \), and~\cref{fig:otter:colours} for empirical
frequency histogram. The numerical values of the tuning parameters are given
in~\cref{tab:numerical:otter}.

\begin{figure}[!hbt]
    \RawFloats
\centering
\begin{minipage}[t]{0.5\textwidth}
    \includegraphics[width=\textwidth]{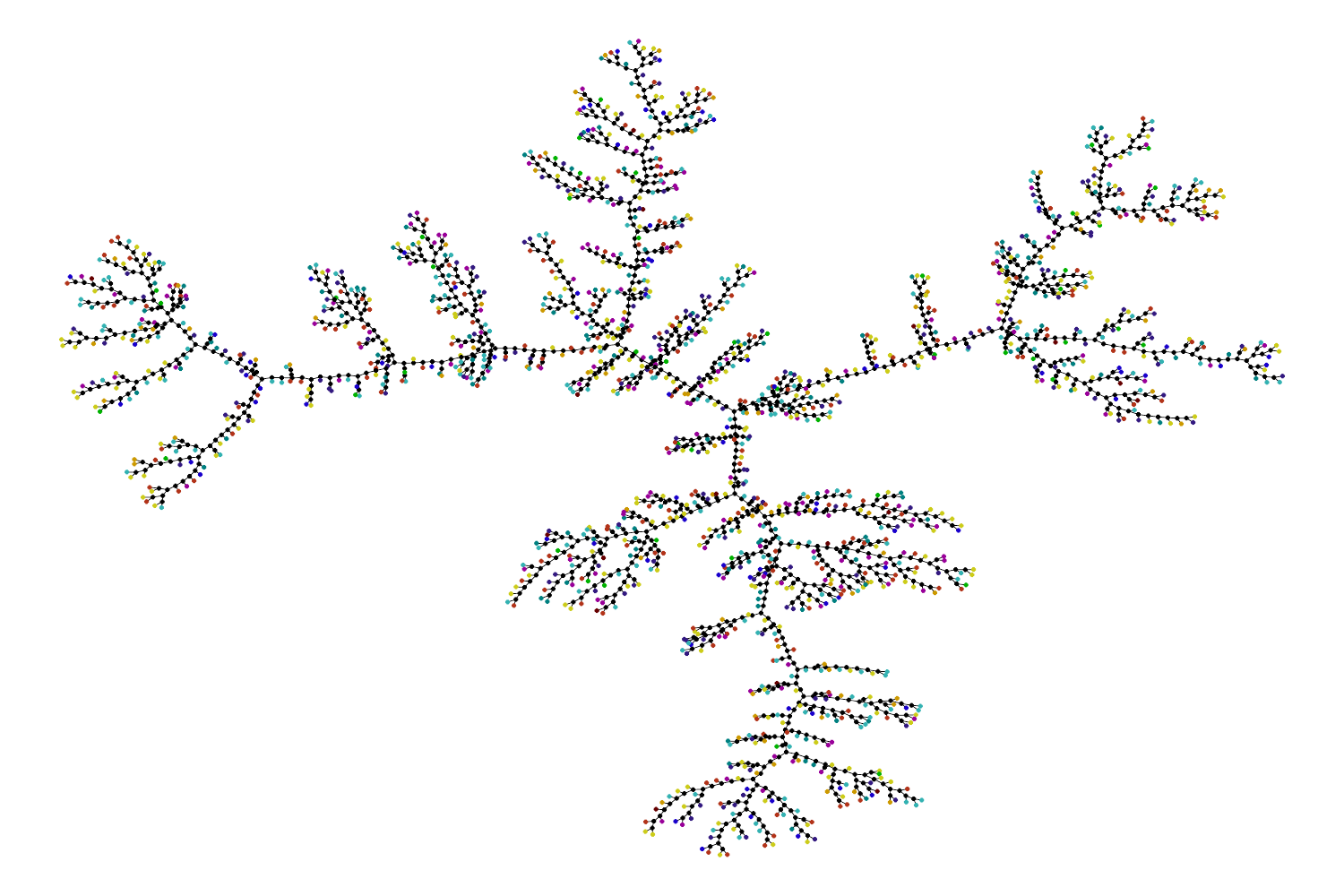}
    \caption{Random coloured Otter tree of size 1434 drawn with Kamada--Kawai
    algorithm~\cite{kamada1989algorithm}.}
    \label{fig:otter}
\end{minipage}%
\begin{minipage}[t]{0.5\textwidth}
    \includegraphics[width=\textwidth]{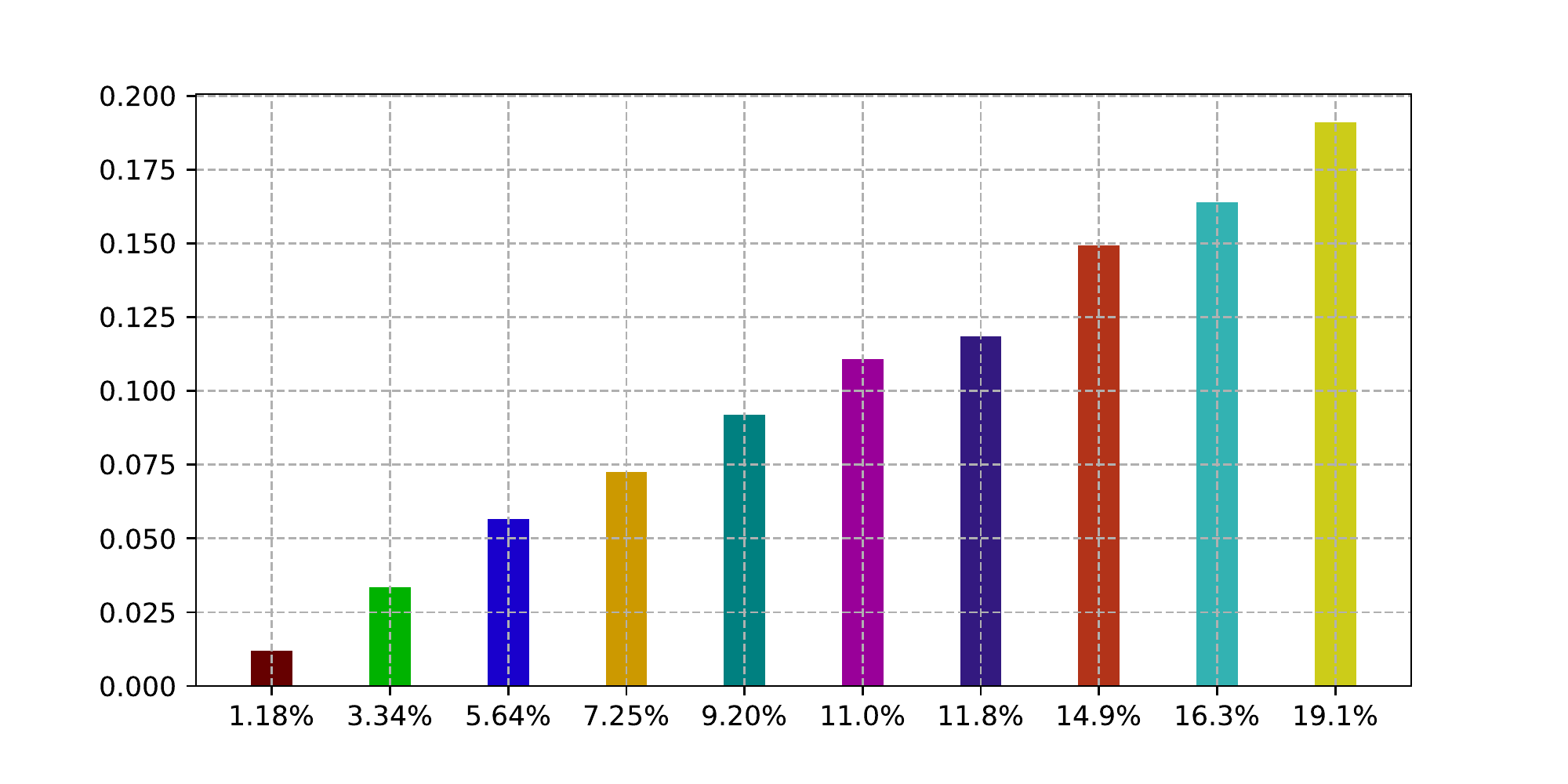}
    \caption{Histogram of colour frequencies for Otter tree.}
    \label{fig:otter:colours}
\end{minipage}
\end{figure}

\begin{table}
    \centering
    \begin{tabular}{c|c|c|c|c|c|c|c|c|c|c}
        &
        $ u_1 z $ &
        $ u_2 z $ &
        $ u_3 z $ &
        $ u_4 z $ &
        $ u_5 z $ &
        $ u_6 z $ &
        $ u_7 z $ &
        $ u_8 z $ &
        $ u_9 z $ &
        $ u_{10} z $
        \\
        \hline
        value &
        0.005 &
        0.015 &
        0.025 &
        0.035 &
        0.044 &
        0.054 &
        0.063 &
        0.072 &
        0.081 &
        0.09
    \end{tabular}
    \caption{Numerical values of tuning parameters for multicoloured Otter
    trees.}
    \label{tab:numerical:otter}
\end{table}

\subsection{Substitution-closed permutation classes}
Permutation patterns stem from a growing body of research which originated from
Knuth's study of sorting algorithms and also later from the study of sorting
networks by Tarjan and Pratt. Recently, the authors
of~\cite{bassino2017algorithm} have obtained a method which allows to
automatically construct a combinatorial specification for permutation classes
avoiding given set of permutations.

Let us start by giving a few basic definitions. A permutation \( \sigma =
(\sigma(1), \sigma(2), \ldots, \sigma(n)) \) is said to \emph{contain a pattern}
\( \pi = (\pi(1), \ldots, \pi(k)) \) if \( \sigma \) has a subsequence whose
terms have the same relative ordering as \( \pi \). A permutation of length \( n
\) is said to be \emph{simple} if it does not contain \emph{intervals} of length
strictly in between $1$ and $n$, where an \emph{interval} is a contiguous
sequence of indices \( \{ i \, \mid \, a \leq i \leq b\} \) such that the set of
values \( \{ \sigma(i) \, \mid \, a \leq i \leq b\} \) is also contiguous. And
so, for instance, the permutation from~\cref{fig:random:permutation} is not
simple, since it contains an interval \( \{ 1, 2, 3\} \) whose image is \(
\{5,6,7\} \).

\begin{figure}[!hbt]
    \begin{center}
        \includegraphics[width=0.4\textwidth]{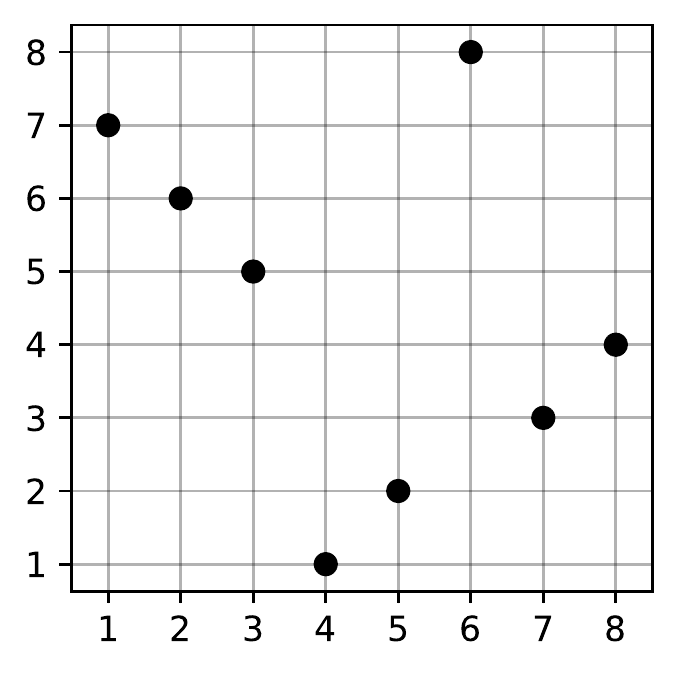}
    \end{center}
    \caption{Visualisation of a random permutation \textsc{76512834}.}
    \label{fig:random:permutation}
\end{figure}

Many interesting permutation classes can be described in the augmented language
of context-free specifications. Given a permutation \( \sigma \in S_m \) and
non-empty permutations \( \tau_1, \ldots, \tau_m \), the \emph{inflation} of \(
\sigma \) by \( \tau_1, \ldots, \tau_m \) is the permutation obtained by
replacing each entry \( \sigma(i) \) with an interval having the same relative
ordering as \( \tau_i \). If \( \tau_1, \ldots, \tau_m \) belong, respectively,
to the classes \( \CS T_1, \ldots, \CS T_m \), such an inflation is denoted as
\( \sigma[\CS T_1, \ldots, \CS T_m] \). While, from the counting point of view,
\( \sigma[\CS T_1, \ldots, \CS T_m] \) is isomorphic to a Cartesian product \(
\CS T_1 \times \cdots \times \CS T_m \), it is useful to explicitly keep the
external permutation \( \sigma \) for sampling and construction purposes.

Interestingly, Albert and Atkinson describe a specification for any
substitution-closed class of permutations~\cite{albert2005simple}.

\begin{proposition}[{\cite[Lemma 11]{albert2005simple}}]
    Suppose that a class \( \CS C \) is substitution-closed and contains \( 12
    \) and \( 21 \). Let \( \CS S \) be the class of all simple permutations
    contained in \( \CS C \). Then, \( \CS C \) satisfies the following system
    of equations
    \begin{align}
        \begin{split}
            \CS C &= \{ \bullet \} + 12[\CS C^+, \CS C] + 21[\CS C^-, \CS C]
            + \sum_{\pi \in \CS S} \pi[\CS C, \CS C, \ldots, \CS C] \\
            \CS C^+ &= \{ \bullet \} + 21[\CS C^-, \CS C]
            + \sum_{\pi \in \CS S} \pi[\CS C, \CS C, \ldots, \CS C] \\
            \CS C^- &= \{ \bullet \} + 12[\CS C^+, \CS C]
            + \sum_{\pi \in \CS S} \pi[\CS C, \CS C, \ldots, \CS C].
        \end{split}
    \end{align}
\end{proposition}

It is, therefore, possible to endow each of the simple permutations \( \pi \in
\CS S \) by a distinguished marking variable \( u_{\pi} \) and insert them into
the specification:
\begin{align}
    \begin{split}
        \CS C &= \{ \bullet \} + 12[\CS C^+, \CS C] + 21[\CS C^-, \CS C]
        + \sum_{\pi \in \CS S} u_{\pi}\cdot \pi[\CS C, \CS C, \ldots, \CS C] \\
        \CS C^+ &= \{ \bullet \} + 21[\CS C^-, \CS C]
        + \sum_{\pi \in \CS S} u_{\pi}\cdot \pi[\CS C, \CS C, \ldots, \CS C] \\
        \CS C^- &= \{ \bullet \} + 12[\CS C^+, \CS C]
        + \sum_{\pi \in \CS S} u_{\pi}\cdot \pi[\CS C, \CS C, \ldots, \CS C].
    \end{split}
\end{align}

Finally, note that by tuning the expectations attached to \( u_{\pi} \) we alter
the expected frequencies of inflations used during the construction of a
permutation.

\section*{Acknowledgements}
We are grateful to Anne Bouillard, Matthieu Dien, Michael Grant, Hsien-Kuei
Hwang, \'Elie de Panafieu, Martin P\'epin, Bruno Salvy for their valueable
feedback and fruitful discussions. We are especially grateful to Cedric Chauve,
Yann Ponty, and Sebastian Will for showing us numerous applications of Boltzmann
sampling in biology, and to Sergey Tarasov for asking a question about
computational complexity of random sampling. We appreciate the anonymous referee's comments that helped to improve the paper and in particular a comment about the existence of the tuning problem solution which resulted in~\Cref{remark:non:tunable}.

\printbibliography

\end{document}